\documentclass{article}

\usepackage{arxiv}

\usepackage[utf8]{inputenc} 
\usepackage[T1]{fontenc}    

\usepackage[labelfont=bf]{caption}

\usepackage[numbers]{natbib}

\makeatletter
\newcommand*{\textlabel}[2]{%
  \edef\@currentlabel{#1}
  \phantomsection
  #1\label{#2}
}
\makeatother

\usepackage{url}            
\usepackage{booktabs}       
\usepackage{amsfonts}       
\usepackage{nicefrac}       
\usepackage{microtype}      
\usepackage{graphicx}
\graphicspath{ {./images/} }
\usepackage{tikz}
\usepackage{mathtools}
\usepackage{bbm}
\usepackage[shortlabels]{enumitem}

\usepackage{textpos}

\usepackage{amsmath}
\usepackage{amsthm}
\usepackage{amssymb}
\usepackage{dashbox}

\usepackage{aliascnt} 
\usepackage[hidelinks]{hyperref}
\usepackage[capitalise]{cleveref}

\newcommand\dboxed[1]{\dbox{\ensuremath{#1}}}

\usepackage{multicol}
\usepackage{booktabs}
\usepackage{pgfplotstable} 
\usepackage{makecell}
\usepackage{colortbl}

\usepackage[most]{tcolorbox} 

\usepackage{tikz}
\usetikzlibrary{arrows,positioning,shapes, arrows.meta, cd}

\usepackage{float}

\newtheorem{theorem}{Theorem}[section]

\newaliascnt{corollary}{theorem}
\newtheorem{corollary}[corollary]{Corollary}
\aliascntresetthe{corollary}

\newaliascnt{proposition}{theorem}
\newtheorem{proposition}[proposition]{Proposition}
\aliascntresetthe{proposition}

\newaliascnt{lemma}{theorem}
\newtheorem{lemma}[lemma]{Lemma}
\aliascntresetthe{lemma}

\theoremstyle{definition}

\newaliascnt{definition}{theorem}
\newtheorem{definition}[definition]{Definition}
\aliascntresetthe{definition}

\newaliascnt{remark}{theorem}
\newtheorem{remark}[remark]{Remark}
\aliascntresetthe{remark}

\DeclareMathOperator{\dist}{dist}
\DeclareMathOperator{\Lip}{Lip}
\newcommand{\C}{\mathbb{C}}
\newcommand{\K}{\mathbb{K}}
\newcommand{\R}{\mathbb{R}}
\newcommand{\N}{\mathbb{N}}
\newcommand{\X}{\mathcal{X}}
\newcommand{\DD}{\mathcal{D}} 
\newcommand{\clip}{\textup{clip}}
\newcommand{\Y}{\mathcal{Y}}
\newcommand{\Z}{\mathcal{Z}}
\newcommand{\PP}{\mathcal{P}}
\newcommand{\CC}{\mathcal{C}} 
\newcommand{\FF}{\mathcal{F}}
\newcommand{\T}{T}
\newcommand{\diff}{\mathop{} \! \mathrm{d}}

\pgfplotsset{compat=1.18}

\newcommand{\raa}[1]{\renewcommand{\arraystretch}{#1}}
\pgfplotsset{compat=1.18}

\title{New universal operator approximation theorem for encoder-decoder architectures}

\author{
  Janek Gödeke \\
  Center for Industrial Mathematics\\
  University of Bremen\\
  \texttt{\href{mailto:janek-goedeke@uni-bremen.de}{janek-goedeke@uni-bremen.de}}\\
  \And
  Pascal Fernsel \\
  Center for Industrial Mathematics\\
  University of Bremen\\
  \texttt{\href{mailto:p.fernsel@uni-bremen.de}{p.fernsel@uni-bremen.de}}
}

\begin{document}
\maketitle
\begin{abstract}
Motivated by the rapidly growing field of mathematics for operator approximation with neural networks, we present a novel universal operator approximation theorem for broad classes of encoder-decoder architectures and a wide range of input and output spaces. In this study, we focus on the approximation of continuous operators between infinite-dimensional normed or metric spaces in the topology of uniform convergence on compact sets. Unlike standard results in the operator learning literature, we additionally investigate the case where the approximating sequence of encoder-decoder architectures can be chosen independently of the compact sets. Taking a topological perspective, we point out that compact-set-independent approximation is a strictly stronger property in most relevant operator learning frameworks. To establish our results, we introduce new approximation properties of input and output spaces tailored to encoder-decoder architectures. These properties enable us to prove a universal operator approximation theorem ensuring uniform convergence on every compact subset of the input space. Our results unify and extend existing universal operator approximation theorems for various encoder-decoder architectures, including classical DeepONets, BasisONets, MIONets, architectures based on frames and other related approaches. A notable feature of our framework is that it also applies to metric spaces beyond the normed setting. In particular, it allows the consideration of $p$-Wasserstein spaces of probability measures as input or output spaces, and Skorohod spaces of c\`adl\`ag functions as input spaces. This generality also opens up potential applications in optimal transport.\\[0.2cm]
\end{abstract}

\section{Introduction}
By convention, an operator typically refers to a possibly non-linear mapping $G:D\subseteq \X\to \Y$ between infinite-dimensional normed spaces $\X$ and $\Y$. Over the past five years, using neural networks for approximating and learning such operators, particularly between function spaces, has received increasing attention. For example, operator learning has been investigated in the field of partial differential equations (PDEs) when learning parameter-to-state maps, that map the parameter function of a PDE to the corresponding solution \cite{Lu2021_DeepONet, Wang2021_PIDeepONet, Bhattacharya2021_PCANet, LI2021_FNO_Intro, Anandkumar2019_neuraloperator, Roseberry2022, Raonic2023_CNO, Tripura2023_WNO_Intro}.  
For some review on deep learning for PDEs, covering also operator learning and its application in parameter identification problems, we refer to \cite{NganyuTanyu_2023}.

For approximating operators $G$, lots of neural network architectures have been developed. One of the historical starting points has been set by T. Chen and H. Chen \cite{Chen1995_approx_thm} in 1995, whose approach has been rediscovered by Lu et al. \cite{Lu2021_DeepONet} and generalized to the well-known deep operator networks (DeepONets). DeepONets fall under the category of encoder-decoder architectures, i.e., they consist of three building blocks
\begin{align*}
    G_\theta \coloneqq D \circ \varphi \circ E,
\end{align*}
where the encoder $E:\X \to \R^m$ extracts finite information about the input, $\varphi:\R^m\to \R^n$ is a neural network and the decoder $D:\R^n\to \Y$ maps into the output space of $G$. This is illustrated in \Cref{fig:commutative_diagram}. The trainable parameters $\theta$ of $G_\theta$ belong to $\varphi$, but can also parameterize the encoder and decoder. For example, in case of DeepONets and suitable function spaces $\X$ and $\Y$, the encoder evaluates the input function at finitely many sampling points, whereas the decoder outputs a linear combination of the so-called trunk networks, see Section \ref{sec:DeepONet_classic} for more details. An example in which both the encoder and decoder are represented by neural networks, are the BasisONets presented in \cite{Hua2023}.
Further encoder-decoder approaches are, e.g., Principle Component Analysis Networks (PCANets, \cite{Bhattacharya2021_PCANet}), Multi Input Operator Networks (MIONets, \cite{Jin2022_MIONET}) which uses encoders corresponding to Schauder bases in Banach spaces, Deep-H-ONets \cite{Castro2023_ONBs} using encoders and decoders based on orthonormal bases, or the approach in \cite{Schwab2023} based on Riesz bases. Usually, the considered encoders and decoders are linear, but there are a few exceptions \cite{Seidman2022,zappala2024_projectionmethod}.

\begin{figure}[t]
    \centering
    {\Large
    \begin{tikzcd}[column sep={2.5cm, between origins}, row sep=1cm]
        \X \arrow{d}[swap]{\displaystyle E} \arrow{r}{\displaystyle G_\theta} & \Y  \\
        \R^n \arrow{r}[]{\displaystyle \varphi} & \R^m \arrow{u}[swap]{\displaystyle D}
    \end{tikzcd}
    }
    \caption{Commutative diagram of the parameterized encoder-decoder architecture $G_\theta = D \circ \varphi \circ E,$ which approximates a given operator $G:\X \to \Y.$}
    
    \label{fig:commutative_diagram}
\end{figure}

As alternatives to encoder-decoder approaches we mention here, e.g., Neural Operators (NOs, \cite{Anandkumar2019_neuraloperator}), Fourier Neural Operators (FNOs, \cite{LI2021_FNO_Intro}), Wavelet Neural Operators (WNOs, \cite{Tripura2023_WNO_Intro}), Representation Invariant Neural Operators (ReNOs, \cite{Bartolucci2023, Raonic2023_CNO}), Injective Integral Neural Operators \cite{furuya2023_injectiveNO}, Causal Neural Operators \cite{Galimberti_2026} or Optimal Transport Neural Operators \cite{Kovachki2025_OTNO}. Some overview on operator learning methods can be found in \cite{Kovachki2024_review}. 

\subsection{Related work on operator approximation theory}
On the theoretical side, a fundamental question is which classes of operators $G:\X\to \Y$ can be approximated by neural networks and under which topology. 
When both $\X$ and $\Y$ are finite dimensional spaces ($\R^d$ or $\X$ being a subset thereof) this question falls within the well-established field of function approximation, which is not the primary focus of our study. Nevertheless, extensive results exist on the approximation properties of neural networks in various function spaces, including spaces of continuously differentiable functions, Lebesgue spaces, and Sobolev spaces, see for example \cite{Cybenko1989, Hornik1991, Barron1993, Pinkus_1999, yarotsky2017error, DeVore2021relu_overview}. Regarding the field of operator approximation, so when $\X$ or $\Y$ are infinite-dimensional, we recap below the two most commonly studied types of approximation, where the first is the focus of our study.

\textbf{Uniform convergence on compacta.}
Let $G:\X\to \Y$ be a continuous operator and $K\subseteq \X$ be compact. The goal is to construct a sequence of approximants $G_n:\X\to \Y$, such as encoder-decoder architectures as described above, that converges uniformly to $G$ on $K$. That is, 
\begin{align*}
    \sup_{f\in K} \big \Vert G(f) - G_n(f) \big \Vert \xrightarrow{n\to\infty} 0.
\end{align*} 
Universal approximation theorems have been derived for a variety of neural network architectures and choices of the spaces $\X$ and $\Y$. Such theorems state that every continuous operator $G$ can be approximated uniformly on compact subsets of $\X$ by approximants $G_n$ from the corresponding architecture class, where $G_n$ usually depends on the respective compact set.
Universal approximation results have been derived for
\begin{itemize}
    \item DeepONets when $\X$ and $\Y$ are spaces $\CC(\Omega, \R)$ of continuous real-valued functions on compact domains $\Omega\subset \R^d$ \cite{Chen1995_approx_thm, Lu2021_DeepONet};
    \item MIONets when $\X$ is a Banach space having Schauder bases and $\Y=\CC(\Omega, \R)$ \cite{Jin2022_MIONET};
    \item Riesz-basis encoder-decoder networks when $\X$ and $\Y$ are separable Hilbert spaces \cite{Schwab2023};
    \item FNOs when $\X$ and $\Y$ are Sobolev spaces $W^{s,2}(\Omega, \R)$ with smoothness $s\geq 0$ \cite{Kovachki2021_approxFNO};
    \item NOs when $\X$ and $\Y$ are Sobolev spaces $W^{s,p}(\Omega, \R)$ (with $1\leq p< \infty$ and $s\geq 0$) or spaces $\CC^k(\Omega, \R)$ of continuously differentiable functions \cite{Kovachki2023, Lanthaler2024_ANO};
    \item Injective Integral Neural Operators when $\X$ and $\Y$ are Lebesgue spaces $L^2(\Omega, \R)$ \cite{furuya2023_injectiveNO};
    \item Neural operators based on nonlinear projection operators when $\X$ and $\Y$ are Banach spaces \cite{zappala2024_projectionmethod};
    \item Transformers when $\X=\PP(\Omega)\times \Omega$ and $\Y=\R^k$, where $\Omega\subseteq \R^d$ is compact and $\PP(\Omega)$ denotes the set of Borel probability measures equipped with any $p$-Wasserstein distance for $p\geq 1$ \cite{furuya2024}.
    \item Neural Filters when $\X$ and $\Y$ are Fr\'echet spaces having Schauder bases \cite[Theorem 1]{Galimberti_2026}. Related results when $\X$ is a Fr\'echet space having a Schauder basis and when $\Y$ is a Banach space can be found in \cite[Section 4]{Benth_2022}.
\end{itemize}
Although approximation theorems are typically stated for spaces of $\R$-valued functions, they often naturally extend to $\R^a$-valued functions.

\textbf{Approximation in Bochner spaces.}
Let $\X$ be a measurable space, $\Y$ a Banach space and $\mu$ a probability measure on $\X$. For a given operator $G$ belonging to the Bochner space $L^p(\X, \Y; \mu)$, one is looking for a sequence $G_n$, being representable by neural networks, that converges to $G$ in the Bochner norm (for suitable $1\leq p <\infty)$, i.e.,
\begin{align*}
    \big \Vert G - G_n\big \Vert_{L^p(\X, \Y; \mu)} = \left(\int_{\X} \big \Vert G(f) - G_n (f) \big \Vert^p \diff \mu(f)\right)^{1/p} \xrightarrow{n\to\infty} 0.
\end{align*}
Also for convergence in Bochner space, universal approximation theorems have been derived for different architectures, e.g., for DeepONets \cite{Lanthaler2022}, PCANets \cite{Bhattacharya2021_PCANet}, Riesz-basis encoder-decoder networks \cite{Schwab2023}, FNOs \cite{Kovachki2021_approxFNO}, NOs \cite{Kovachki2023, Lanthaler2024_ANO} or Deep-H-Onets \cite{Castro2023_ONBs}.

For both types of approximation, there do not only exist universal approximation theorems, but also results on estimating the required size of the neural network to achieve a desired approximation accuracy. Typically, specific input spaces of the operator $G$ and additional assumptions on $G$ are required, e.g., that $G$ arises from specific PDEs, is holomorphic, or Lipschitz/Hölder continuous, see e.g., \cite{Lanthaler2022, Kovachki2021_approxFNO, Cohen_DeVore_2015, Liu2024, Schwab2023, kratsios2023, Galimberti_2026}. For more information about this important field of research we refer to the review \cite{Kovachki2024_review}, in which also the curse of dimensionality is discussed.
Within this work, we restrict our attention to universal approximation results only. 

There are also other types of approximating operators. For example, in \cite[Theorem 3.3]{kratsios2023} continuous operators $G$ between Polish spaces $\X, \Y$ are approximated by encoder-decoder architectures with respect to the 1-Wasserstein metric.
Last but not least, in \cite{Korolev2022}, dual spaces $\X$ and $\Y$ of separable Banach spaces were considered. If $\Y$ is equipped with the (metrizable) weak$^\ast$-topology, a sequence of generalized ReLU-networks has been found that uniformly approximates a given operator (belonging to some variation norm space) on bounded sets. Further, approximation in Bochner spaces $L^p(\X, (\Y, d_\ast); \mu)$ with suitable probability measures $\mu$ was shown.

\subsection{Contribution}\label{sec:Contribution}
To derive universal operator approximation results, we consider the space $\CC(\X, \Y)$ of continuous mappings between normed or metric spaces $\X$ and $\Y$. Throughout this work, we study approximation in the topology on $\CC(\X, \Y)$ induced by the uniformity of uniform convergence on compact sets; see \cite{Engelking_Topologie} for a detailed definition. With respect to this topology, there are two distinct notions of approximating any continuous operator $G\in \CC(\X, \Y)$ by some class $S$ of approximators $G_n:\X\to \Y$. These notions are illustrated in \Cref{fig:approximation_types}. In this work, we focus on $S$ being diverse classes of encoder-decoder architectures.
\begin{figure}[ht]
    \begin{tcolorbox}[width=0.85\textwidth, center,arc=13pt, colback=gray!10]
        \vspace{-0.35cm}
        \hspace{0.3cm}
        $\displaystyle
            \begin{tabular}{p{.535\textwidth}}
                \begin{enumerate}[leftmargin=0.2cm, , label=(\textbf{\Alph*})]
                    \item $\displaystyle \hspace{0.25cm} \forall G\in \CC(\X, \Y) \hspace{0.3cm} \boxed{\forall K\subseteq \X \textup{ compact }} \hspace{0.2cm} \dboxed{\exists G_n \in S} \vspace{0.3cm}$ \label{fig:approximation_types:(A)}
                    \item $\displaystyle \hspace{0.25cm} \forall G\in \CC(\X, \Y) \hspace{0.3cm}\dboxed{\exists G_n \in S}\ \ \ \boxed{\forall K\subseteq \X \textup{ compact }\hspace{-0.1cm}}$ \label{fig:approximation_types:(B)}
                \end{enumerate}
            \end{tabular} 
            \hspace{0.8cm}
            \raisebox{-0.1cm}{{\scalebox{1.5}{\Bigg\}}}} \hspace{0.3cm} \raisebox{-0.08cm}{$\begin{aligned}
                \sup_{f\in K} d_\Y &\big(G (f)\, ,\, G_n (f)\big)\\[-0.1cm]
                &\xrightarrow{n\to\infty} 0.
            \end{aligned}$}
        $
        
        \vspace{-0.2cm}
    \end{tcolorbox}
    \caption{Different types of universal operator approximation in $\CC(\X, \Y)$ by some class $S$ of approximators.}
    \label{fig:approximation_types}
\end{figure}
Our main contributions to the approximation theory of encoder-decoder architectures in the topology of uniform convergence on compact sets are as follows:
\begin{itemize}
    \item \textbf{Proof of statement \labelcref{fig:approximation_types:(B)}.}\\
    In most other studies, statement \labelcref{fig:approximation_types:(A)} in \Cref{fig:approximation_types} is shown for $S$ being, for example, DeepONets or (F)NOs. That is, for every compact $K\subseteq\X$ one can find a sequence $(G_n)_{n\in \N}$ in $S$ that converges uniformly to a given operator $G\in \CC(\X, \Y)$ on $K$. However, can the sequence $G_n$ be chosen independently of the compact set $K$? To the best of our knowledge, there exists only one result by Schwab et al. \cite[Theorem 3.1]{Schwab2023}, who have shown this statement \labelcref{fig:approximation_types:(B)} for $S$ being encoder-decoder networks based on Riesz bases in separable Hilbert spaces $\X$ and $\Y$.
    We show in our main \Cref{thm:sequential_density_normed} that statement \labelcref{fig:approximation_types:(B)} - which is the major focus of our study - can indeed be achieved by much more diverse choices of encoder-decoder architectures. For example, our theory also covers classical DeepONets \cite{Chen1995_approx_thm, Lu2021_DeepONet, Lanthaler2022}, special cases of MIONets \cite{Jin2022_MIONET}, BasisONets \cite{Hua2023} and Deep-H-ONets \cite{Castro2023_ONBs}.\newline  
    Some advantage of \labelcref{fig:approximation_types:(B)} over \labelcref{fig:approximation_types:(A)} is that it could simplify the derivation of approximation theorems in Bochner spaces. This is because \labelcref{fig:approximation_types:(B)} implies pointwise convergence, allowing one to invoke the dominated convergence theorem. We leave this application in Bochner approximation for future research. 
    
    \item \textbf{Proof that \labelcref{fig:approximation_types:(B)} really is stronger than \labelcref{fig:approximation_types:(A)}.}\\
    A natural question arises: Is \labelcref{fig:approximation_types:(B)} really a stronger statement than \labelcref{fig:approximation_types:(A)} for metric spaces $\X, \Y$? In other words, is there a set $S$ of approximators that satisfies \labelcref{fig:approximation_types:(A)} but not \labelcref{fig:approximation_types:(B)}? The answer is yes if and only if $\X$ is not hemicompact. In particular, the answer is yes if the input space $\X$ is any infinite-dimensional normed space, which is typically the case for operator approximation tasks. In the case of $\Y=\R$ and any separable, metrizable topological space $\X,$ this question has been answered in \cite[Corollary 3.7]{Osipov2018}, which easily extends to any finite dimensional $\Y$. However, we could not find a result for infinite-dimensional $\Y$ within the literature. Therefore, we provide a compact proof in \Cref{thm:dense_Frechet_hemicompact} for any metric space $\X$ and normed space $\Y$. Nevertheless, these topological considerations are not required for understanding the other parts of this study.
    
    \item \textbf{One Theorem for many architectures and spaces.}\\
    Within the literature, approximation results are often shown individually for different operator learning architectures. This means that an approximation theorem for a certain architecture class may not directly apply to a different architecture class, see for example \Cref{rem:deeponet_vs_mionet}.
    In contrast, we provide theorems that cover many classes of encoder-decoder architectures, as well as diverse normed and metric input and output spaces $\X$ and $\Y$, at once: For approximation as in statement \labelcref{fig:approximation_types:(B)} we provide our main \Cref{thm:sequential_density_normed}, whereas for the weaker statement \labelcref{fig:approximation_types:(A)} we provide \Cref{thm:density_metric} and \Cref{thm:mionet}. Applicability to well-known architectures is outlined in Section \ref{sec:examples_architectures}.
    Noteworthy, in \cite[Theorem 3.7, Lemma 2.6]{kratsios2023} and \cite[Proposition 4.15]{galimberti2024} the authors proved statement \labelcref{fig:approximation_types:(A)}, but not \labelcref{fig:approximation_types:(B)} (whenever $\X$ is not hemicompact), for broad classes of encoder-decoder architectures and diverse metric/topological input and output spaces. Nevertheless, neither their results nor \cref{thm:density_metric} are more general than the other regarding possible choices of encoders and decoders, see \cref{rem:comparison_galimberti}.

    \item \textbf{Identify sufficient properties of input and output spaces.}\\
    Regarding statement \labelcref{fig:approximation_types:(B)}, our results significantly extend the existing theory for encoder-decoder architectures regarding possible input and output spaces, as well as possible encoder and decoder constructions. Regarding the theory related to statement \labelcref{fig:approximation_types:(A)}, we also make contributions, although the existing literature in this area is more extensive.
    For that, we introduce in Sections \ref{sec:approximation_properties} and \ref{sec:CEDAP} sufficient properties of the normed or metric input and output spaces $\X$ and $\Y$, which enable approximation of continuous operators by encoder-decoder architectures as in statements \labelcref{fig:approximation_types:(B)} and \labelcref{fig:approximation_types:(A)}, respectively. 
    Inspired by the (bounded) approximation property of normed spaces, our conditions require approximation of the identity mappings on the input and output spaces by encoder-decoder pairs rather than by linear finite-rank operators. The corresponding encoders and decoders can then be used directly in the approximation \Cref{thm:sequential_density_normed,thm:density_metric,thm:mionet}, which allows for a straight-forward application to many classes of architectures. Further, we relax several assumptions often imposed in the operator-learning literature. In particular, our framework accommodates metric spaces, permits nonlinear encoders and decoders, allows encoders to be discontinuous, and does not require decoders to be Lipschitz continuous.  We give examples of metric spaces along with natural, constructive choices of encoders and decoders that benefit from these relaxations, see the point below.

    \item \textbf{Theory covers metric spaces: $p$-Wasserstein and Skorohod spaces.}\\
    To the best of our knowledge, the existing operator learning theory for statement \labelcref{fig:approximation_types:(B)} handles operators between separable Hilbert spaces. Our frameworks for both \labelcref{fig:approximation_types:(B)} and \labelcref{fig:approximation_types:(A)} additionally applies to every normed space possessing the (bounded) approximation property, including, for example, Lebesgue spaces, Sobolev spaces, and spaces of continuously differentiable functions. In addition, also metric spaces can be treated as input or output spaces within our frameworks.
    For example, we prove that spaces of probability measures $\PP_p(\Omega)$ equipped with the $p$-Wasserstein ($p\geq1$) distance as well as Fr\'echet spaces having Schauder bases can be treated as input and output spaces. Moreover, we show that Skorohod spaces of c\`adl\`ag functions equipped with the Skorohod topology can be treated as input spaces in our theory.
    In many spaces, explicit and natural constructions of suitable encoders and decoders are possible, which we outline in Sections \ref{sec:normed_examples_EDAP} and \ref{sec:metric_EDAP_spaces} for normed and metric spaces, respectively.
    Note that some of our natural encoders for Wasserstein and Skorohod spaces are currently not covered by the existing theory even for statement \labelcref{fig:approximation_types:(A)}, which is due to their discontinuity, see also \cref{rem:comparison_galimberti}.
    Furthermore, in Section \ref{sec:applications_OT} we point out some potential applications in optimal transport, for learning operators $\PP_p(\Omega)\to \Y$ with encoder-decoder architectures. In particular, we discuss how \Cref{thm:mionet} provides a possible route towards a first approximation theorem for geodesic operator networks \cite{Gracyk_geonet}.
\end{itemize}

\subsection{Outline}
In \Cref{sec:topology}, we discuss the topological difference between statements \labelcref{fig:approximation_types:(A)} and \labelcref{fig:approximation_types:(B)} shown in \Cref{fig:approximation_types}. We highlight that \labelcref{fig:approximation_types:(B)} is generally a stronger result in most contexts relevant to operator approximation. Nevertheless, this section is not required for understanding the remainder of this work.
In Section \ref{sec:approximation_properties}, we introduce sufficient properties of the input and output spaces $\X$ and $\Y$ that enable the approximation of continuous operators $G\in \CC(\X, \Y)$ by encoder-decoder architectures as in statement \labelcref{fig:approximation_types:(B)}.
We also discuss several examples of normed and metric spaces satisfying these properties and present explicit choices of encoders and decoders that can be used within our approximation theorems.
In Section \ref{sec:approx_theorem}, we establish our main universal approximation results: \Cref{thm:sequential_density_normed}, which yields approximation as in statement \labelcref{fig:approximation_types:(B)}, and \Cref{thm:density_metric} for approximation as in \labelcref{fig:approximation_types:(A)}.
In Section \ref{sec:examples_architectures}, we discuss the applicability of our theoretical framework to existing encoder-decoder architectures and show how our results unify and extend several universal approximation theorems from the literature. Finally, in Section \ref{sec:applications_OT}, we discuss some applications in optimal transport. This is motivated by the fact that our theoretical framework allows the consideration of $p$-Wasserstein spaces as input spaces.

\subsection{Notation}
Given a metric $d$ on a set $\X$, we write $(\X, d)$ for the corresponding metric space. We sometimes omit $d$ in the notation. Given a metric space $(\X, d)$, the open ball of radius $r>0$ around some $x\in \X$ is denoted by $B_r(x;d)$. If there is no ambiguity in the choice of $d$, we write $B_r(x).$ The closed ball is written as $\overline{B}_r(x)$. Whenever the symbol $\K$ is used, it is meant that the according definitions or results are valid for both fields $\K=\R$ or $\K=\C.$ All vector spaces considered within this study are meant to be over the field(s) $\K$.
For metric spaces $\X$ and $\Y$, we denote by $\CC(\X, \Y)$ the space of continuous mappings $\X \to \Y$, and by $\CC_b(\X, \Y)$ its subset of bounded continuous mappings. Whenever $\Y$ is a vector space, both spaces are regarded as vector spaces with pointwise addition and scalar multiplication. Finally, we denote the Euclidean norm on $\K^n$ as $\vert\cdot \vert$. Nevertheless, due to the equivalence of norms, also different choices of norms are possible throughout our study. The expressions $\vert \cdot \vert_1$ and $\vert\cdot\vert_\infty$ are used for the $\ell^1$- and $\ell^\infty$-norm, respectively.

\section{Perspective from the compact-open topology}\label{sec:topology}
In this section, we will have a closer look on the topological difference between the two types of operator approximation given in \Cref{fig:approximation_types} by a set $S$ of approximants. For simplicity, we assume that $S\subseteq \CC(\X, \Y)$. In the following two sections, we recall what the statements \labelcref{fig:approximation_types:(A)} and \labelcref{fig:approximation_types:(B)} in \Cref{fig:approximation_types} mean in terms of the so called compact-open topology. 
Finally, in Section \ref{sec:denseFU_FU_not_equivalent} we show that \labelcref{fig:approximation_types:(B)} is indeed a stronger statement in general, at least for settings relevant for operator approximation tasks, in which $\X$ and $\Y$ are typically infinite-dimensional normed (function) spaces.
Regarding the definition of the compact-open topology, we follow \cite[Chapter 3.4]{Engelking_Topologie}.
\begin{definition}[Compact-open topology]\label{def:compact_open_topology}
    Let $\X$ and $\Y$ be topological spaces. For compact $K\subseteq \X$ and open $V\subseteq \Y$ associate the set
    \begin{align*}
        U_{K, V} \coloneqq \{ G\in \CC(\X, \Y): G(K) \subseteq V\}.
    \end{align*}
    The compact-open topology on $\CC(\X, \Y)$ is the topology that has all such $U_{K, V}$ as a subbase. In other words, it is the coarsest topology on $\CC(\X, \Y)$ containing all $U_{K, V}$.
\end{definition}

It is to be mentioned that the compact-open topology coincides with the topology induced by the uniformity of uniform convergence on compacta, e.g., when $\X$ and $\Y$ are metric spaces \cite[Theorem 8.2.6]{Engelking_Topologie}. Nevertheless, we choose the perspective from the compact-open topology, as its definition is straightforward and since the studies we reference use this perspective, too.

\subsection{Density}
In this section, we recall in \Cref{thm:density_compact_open_topology} that statement \labelcref{fig:approximation_types:(A)} in \Cref{fig:approximation_types} is equivalent to $S\subseteq \CC(\X, \Y)$ being dense with respect to the compact-open topology. We start with basic definitions of adherent points and dense sets in topological spaces.
\begin{definition}[Adherent point]\label{def:adherent_point}
    Let $(\Z, \mathfrak{T})$ be a topological space and $S\subseteq\Z$. A point $x\in \Z$ is called adherent point of $S$ (in the topology $\mathfrak{T})$ if for every open neighborhood $U\in \mathfrak{T}$ of $x$ it holds that $S\cap U \neq \emptyset$.
\end{definition}

\begin{definition}[Density]
    Let $(\Z, \mathfrak{T})$ be a topological space. A subset $S\subseteq\Z$ is called dense in $\Z$ if every $z\in \Z$ is an adherent point of $S$.
\end{definition}
In other words, $S\subseteq \Z$ is dense if its closure is the whole space $\Z$. The following Lemma characterizes adherent points in the compact-open topology. It seems to be a well-known fact, but we could not find an explicit statement within the literature. Therefore, for the interested reader, we provide a proof in the appendix (see \Cref{subsec:proof_adherent_point_compact_open_topology}).
\begin{lemma}\label{lem:adherent_point_compact_open_topology}
    Let $\X$ and $(\Y, d_\Y)$ be metric spaces and consider a subset $S\subseteq \CC(\X, \Y)$. A mapping $G\in \CC(\X, \Y)$ is an adherent point of $S$ with respect to the compact-open topology if and only if for every compact $K\subseteq \X$ there exists a sequence $(G_n)_{n\in \N}$ in $S$ such that 
        \begin{align*}
            \sup_{x\in K} d_\Y\big(G(x)\, ,\, G_n(x) \big) \xrightarrow{n\to\infty} 0.
        \end{align*}
\end{lemma}
An immediate consequence of the preceding lemma is the following characterization of dense sets. 
\begin{theorem}\label{thm:density_compact_open_topology}
    Let $\X$ and $(\Y, d_\Y)$ be metric spaces. For a subset $S\subseteq \CC(\X, \Y)$ the following statements are equivalent:
    \begin{itemize}
        \item[(i)] $S$ is dense in $\CC(\X, \Y)$ with respect to the compact-open topology.
        \item[(ii)] For every $G\in \CC(\X, \Y)$ and compact $K\subseteq \X$ there exists a sequence $(G_n)_{n\in \N}$ in $S$ such that 
        \begin{align*}
            \sup_{x\in K} d_\Y\big(G(x)\, ,\, G_n(x) \big) \xrightarrow{n\to\infty} 0.
        \end{align*}
    \end{itemize}
    Note that the choice of the sequence $(G_n)_{n\in\N}$ may depend on $K$.
\end{theorem}

\subsection{Sequential density}
We have seen that with a dense set $S\subseteq \CC(\X, \Y)$ one can approximate every $G\in \CC(\X, \Y)$ on any compact $K\subseteq \X$ by a sequence, which depends on $K$. If we require these sequences to be independent on $K$, so if statement \labelcref{fig:approximation_types:(B)} is supposed to hold, we will recall in \Cref{thm:sequential_density_compact_open_topology} below that this is equivalent to sequential density of $S$ in $\CC(\X, \Y)$ with respect to the compact-open topology. 
\begin{definition}[Convergence of sequences]
    Let $(\Z,  \mathfrak{T})$ be a topological space. A sequence $(x_n)_{n\in\N}$ in $\Z$ converges to some $x\in \Z$ if for every open neighborhood $U\in \mathfrak{T}$ there is an $N\in \N$ such that for all $n\geq N$ it is $x_n\in U.$
\end{definition}

\begin{definition}[Sequential density]
    Let $(\Z,  \mathfrak{T})$ be a topological space. A subset $S\subseteq \Z$ is called sequentially dense in $\Z$ if for each $x\in \Z$ there is a sequence $(x_n)_{n\in \N}$ in $S$ that converges to $x$. That is, for each open neighborhood $U\in \mathfrak{T}$ of $x$ there exists an $N\in \N$ such that for all $n\geq N$ it holds that $x_n\in U$.
\end{definition}

In other words, a set $S$ is sequentially dense if its sequential closure is the whole space. The following characterization of convergent sequences in the compact-open topology can be found, for example, in \cite[Chapter XII.7]{Dugundji_Topologie}, which immediately allows for a characterization of sequentially dense sets in \Cref{thm:sequential_density_compact_open_topology}.
\begin{lemma}\label{lem:convergence_compact_open_topology}
    Let $\X$ and $(\Y, d_\Y)$ be metric spaces. A sequence $(G_n)_{n\in \N}$ in $\CC(\X, \Y)$ converges to some $G\in \CC(\X, \Y)$ in the compact-open topology if and only if for every compact $K\subseteq \X$ it holds that 
    \begin{align*}
        \sup_{x\in K} d_\Y\big(G(x)\, ,\, G_n(x) \big) \xrightarrow{n\to\infty} 0.
    \end{align*}
\end{lemma}

\begin{theorem}\label{thm:sequential_density_compact_open_topology}
    Let $\X$ and $(\Y, d_\Y)$ be metric spaces. For a subset $S\subseteq \CC(\X, \Y)$ the following statements are equivalent:
    \begin{itemize}
        \item[(i)] $S$ is sequentially dense in $\CC(\X, \Y)$ with respect to the compact-open topology.
        \item[(ii)] For every $G\in \CC(\X, \Y)$ there exists a sequence $(G_n)_{n\in \N}$ in $S$ such that for every compact $K\subseteq \X$ it holds that 
        \begin{align*}
            \sup_{x\in K} d_\Y\big(G(x)\, ,\, G_n(x) \big) \xrightarrow{n\to\infty} 0.
        \end{align*}
    \end{itemize}
    Note that the choice of the sequence $(G_n)_{n\in\N}$ is independent of $K$.
\end{theorem}

\subsection{Density versus sequential density in the compact-open topology}\label{sec:denseFU_FU_not_equivalent}
In the previous sections, we characterized the different types of operator approximation given in \Cref{fig:approximation_types} in the context of the compact-open topology.
According to \Cref{thm:density_compact_open_topology}, statement \labelcref{fig:approximation_types:(A)} in \Cref{fig:approximation_types} means that $S$ is dense with respect to the compact-open topology, whereas \labelcref{fig:approximation_types:(B)} is equivalent to sequential density of $S$ due to \Cref{thm:sequential_density_compact_open_topology}. 
The question whether it is an improvement to prove \labelcref{fig:approximation_types:(B)} and not only \labelcref{fig:approximation_types:(A)} hence corresponds to the question whether dense subsets of $\CC(\X, \Y)$ are always sequentially dense w.r.t. the compact-open topology.
In fact, this question is closely related to the Fréchet-Urysohn property. For the definition, we follow \cite[Chapter 1.6]{Engelking_Topologie}, in which the term Fréchet space is used. 
\begin{definition}[Fréchet-Urysohn space]
    A topological space $\Z$ is said to be a  Fréchet-Urysohn space if for every $S\subseteq \Z$ and adherent point $z\in \overline{S}$ there is a sequence in $S$ converging to $z$. In other words, the closure and sequential closure of every $S\subseteq \Z$ coincide.
\end{definition}
Clearly, if a space is a Fréchet-Urysohn space, then dense sets are particularly also sequentially dense. However, the converse may not be true in general, see for example \Cref{lem:example_dense_FU_not_FU}. On the other hand, for the compact-open topology on $\CC(\X, \Y)$, the converse holds for many choices of $\X$ and $\Y$, as revealed by the following theorem.

\begin{theorem}\label{thm:dense_Frechet_hemicompact}
    Let $\X$ be a metric space and $\Y$ be a normed space which is at least one-dimensional. Then the following statements are equivalent. 
    \begin{enumerate}[(i)]
        \item Every dense subset of $\CC(\X, \Y)$ is also sequentially dense (w.r.t. compact-open topology).\label{thm:dense_Frechet_hemicompact:i}
        \item $\X$ is hemicompact.\label{thm:dense_Frechet_hemicompact:ii}
        \item $\CC(\X, \Y)$ equipped with the compact-open topology is a Fréchet-Urysohn space.\label{thm:dense_Frechet_hemicompact:iii}
    \end{enumerate}
\end{theorem}
For $\Y=\R$ and separable, metrizable $\X$, this result has been derived in \cite[Corollary 3.7]{Osipov2018}, which easily extends to any finite dimensional $\Y$. In \Cref{thm:dense_Frechet_hemicompact}, we also treat infinite-dimensional $\Y$. Further, for the case that $\X$ is a separable metric space and for certain metrizable spaces $\Y,$ equivalence between \labelcref{thm:dense_Frechet_hemicompact:ii} and \labelcref{thm:dense_Frechet_hemicompact:iii} has been shown in \cite[Theorem 3.4, Remark 3.7]{Gabriyelyan2020}.

In the remainder of this section, we will prove the implication "\labelcref{thm:dense_Frechet_hemicompact:i} $\to$ \labelcref{thm:dense_Frechet_hemicompact:ii}" with \Cref{lem:denseFU_k_sequence,lem:select_k_sequence,lem:ksequence_hemicompact}. Implication "\labelcref{thm:dense_Frechet_hemicompact:ii} $\to$ \labelcref{thm:dense_Frechet_hemicompact:iii}" is handled in \Cref{lem:hemicompact_FU}. For the proofs we follow the ideas presented in \cite{Caserta2006} and \cite{Osipov2018}. Note that "\labelcref{thm:dense_Frechet_hemicompact:iii} $\to$ \labelcref{thm:dense_Frechet_hemicompact:i}" is evident. To start, let us recall the definitions of hemicompact and locally compact spaces, for which we follow \cite[Definition 3.3 in Chapter 11]{Morita_Topologie} and \cite[Section 3.3]{Engelking_Topologie}, respectively.
\begin{definition}[Hemicompact space]
    A topological space $\X$ is called hemicompact if there is a sequence of compact subsets $(K_n)_{n\in \N}$ of $\X$ such that each compact $K\subseteq \X$ is contained in some $K_n$.
\end{definition}

\begin{definition}[Locally compact space]
    A topological space $\X$ is called locally compact if each $x\in \X$ has a compact neighborhood, i.e., there is a compact set $K\subseteq \X$ and some open $U\subseteq\X$ such that $x\in U\subseteq K$.
\end{definition}
Note that every hemicompact metric space $\X$ must be separable, as it can be written as a countable union of compact metric spaces, where the latter are separable (see \Cref{lem:compact_metric_separable}). For non-separable $\X$, \Cref{thm:dense_Frechet_hemicompact} hence implies that there is always a dense set $S\subset \CC(\X, \Y)$ which is not sequentially dense. This shows that \labelcref{fig:approximation_types:(A)} and \labelcref{fig:approximation_types:(B)} above are not equivalent in general. An explicit construction of such a set is provided in \Cref{thm:X_nonseparableMetricSpace_equivalence}, for which we are grateful to Hendrik Vogt from the University of Bremen for his valuable contribution to this example.
Further, if $\X$ is an infinite-dimensional normed space, it follows that it cannot be locally compact, as the closed balls $\overline{B}_r(x)$ for $x\in \X$ and $r>0$ are never compact. A combination of \Cref{thm:dense_Frechet_hemicompact} as well as the \Cref{lem:denseFU_k_sequence,lem:ksequence_hemicompact} leads to the fact that every hemicompact metric space is locally compact (alternatively, see also \cite[Section 8]{Arens1946}). Therefore, \Cref{thm:dense_Frechet_hemicompact} also reveals that for any infinite-dimensional normed space $\X$, there must be dense sets in $\CC(\X, \Y)$ which are not sequentially dense.

In order to prove \Cref{thm:dense_Frechet_hemicompact}, we need to settle down some more terminology and introduce the notion of a so-called open $k$-cover and a $k$-sequence, for which we follow \cite{Caserta2006} and \cite{Osipov2018}.

\begin{definition}[Open $k$-cover]
    Let $\X$ be a topological space. An open $k$-cover for $\X$ is a collection of open sets $\mathfrak{U} = \{U_i: i\in I\}$ such that $\X\notin \mathfrak{U}$ and for each compact $K\subseteq \X$ there exists some $i\in I$ such that $K\subseteq U_i$.
\end{definition}
Let us remark that a compact space $\X$ cannot have a $k$-cover, simply because $\X \neq U_i$ for every $U_i$ in the $k$-cover. Further, every $k$-cover must be an infinite set. Otherwise, one could choose $x_i\in \X\setminus U_i$ for all $i\in I$ (finitely many), so the compact set $\{x_i: i\in I\}$ would not be contained in any $U_i$. Further, the union over all elements of an open $k$-cover is the whole space (so a $k$-cover really is a cover of the space).

\begin{definition}[$k$-sequence]
    Let $\X$ be a topological space. A $k$-sequence for $\X$ is a sequence of subsets $\mathfrak{U} = \{U_n: n\in \N\}$ such that $\X\notin \mathfrak{U}$ and for each compact $K\subseteq \X$ there exists some $N\in \N$ such that $K\subseteq U_n$ for all $n\geq N$.
\end{definition}
Clearly, every $k$-sequence of open sets is also an open $k$-cover. We say that an open $k$-cover $\mathfrak{U} = \{U_i: i\in I\}$ contains a $k$-sequence if there is a sequence $(U_{i_n})_{n\in \N}$ in $\mathfrak{U}$ which is a $k$-sequence. The following Lemma draws a connection between \Cref{thm:dense_Frechet_hemicompact} \labelcref{thm:dense_Frechet_hemicompact:i} and the property that every open $k$-cover of $\X$ contains a $k$-sequence. Note that within the literature, the latter characteristic of the space $\X$ is also sometimes referred to as $\X$ being a $\gamma_k$-set \cite{Osipov2018} or $\X$ fulfilling the $\gamma_\mathfrak{I}\text{-property}$ for the ideal $\mathfrak{I}$ consisting of all compact subsets of $\X$ \cite{Gabriyelyan2020}. The idea for the proof is inspired by \cite[Theorem 3.3]{Osipov2018} in which the case $\Y=\R$ was considered.
\begin{lemma}\label{lem:denseFU_k_sequence}
    Let $\X$ be a metric space and $\Y$ be a normed space which contains some $y_1\neq 0$. Assume that every dense subset of $\CC(\X, \Y)$ is sequentially dense (w.r.t. the compact-open topology). Then every open $k$-cover of $\X$ contains a $k$-sequence.
\end{lemma}
\begin{proof}
    Let $\mathfrak{U}=\{U_i: i\in I\}$ be some open $k$-cover of $\X$. We first show that    \begin{align*}
        D \coloneqq \{f\in \CC(\X, \Y): f = y_1 \textup{ on } \X \setminus U_i \textup{ for some } i\in I\}
    \end{align*}
    is dense in $\CC(\X, \Y)$ with respect to the compact-open topology. For that, let $K\subseteq \X$ be compact. Since $\mathfrak{U}$ is an open $k$-cover of $\X$, there exists some $i\in I$ such that $K\subseteq U_i$. According to Urysohn's lemma, see for example \cite[Chapter 4]{Dugundji_Topologie}, there exists some continuous mapping $p:\X \to [0,1]$ which satisfies
    \begin{align*}
        p(x) &= 0 \textup{ for } x\in K;\\
        p(x) &= 1 \textup{ for } x\in \X \setminus U_i.
    \end{align*}
    Therefore, for any $f\in \CC(\X, \Y)$, the mapping $\tilde{f} = (f - y_1) (1- p) + y_1$ is an element of $D$ and coincides with $f$ on the compact set $K$.  Hence, density of $D$ follows from \Cref{thm:density_compact_open_topology}.
    By assumption, $D$ must also be sequentially dense. In particular, there is a sequence $(f_n)_{n\in\N}$ in $D$ which converges to the zero function $f_0\in \CC(\X, \Y)$ uniformly on every compact set, see \Cref{thm:sequential_density_compact_open_topology}. For each $n\in \N$ there is some $i(n)\in I$ such that $f_n = y_1$ on $\X\setminus U_{i(n)}$. We claim that $\mathfrak{U}'\coloneqq \{U_{i(n)}:n\in \N\}$ is the desired $k$-sequence for $\X$ which is contained in $\mathfrak{U}$. For that, let $K\subseteq \X$ be compact. If $K$ is not a subset of $U_{i(n)}$ for some $n\in \N$, it follows that 
    \begin{align*}
        \sup_{x\in K} \Vert f_0(x) - f_n(x) \Vert = \sup_{x\in K} \Vert f_n(x) \Vert \geq \Vert y_1 \Vert.
    \end{align*}
    Since $f_n$ uniformly converges to $f_0$ on $K$, there must be an $N\in \N$ such that for all $n\geq N$ it holds that $K\subseteq U_{i(n)}$. Hence, $\mathfrak{U}'$ is a $k$-sequence.
\end{proof}

The next lemma is similar to \cite[Theorem 18]{Caserta2006}. However, we prove a weaker statement which will simplify the proof, but  still be sufficient for our purpose of proving implication "\labelcref{thm:dense_Frechet_hemicompact:i}$\rightarrow$\labelcref{thm:dense_Frechet_hemicompact:ii}" in \Cref{thm:dense_Frechet_hemicompact}.
\begin{lemma}\label{lem:select_k_sequence}
    Let $\X$ be a metric space, in which every open $k$-cover contains a $k$-sequence. Consider any sequence $(\mathfrak{U}_n)_{n\in \N}$ of open $k$-covers. Then there is some subsequence $(\mathfrak{U}_{n_m})_{m\in \N}$ such that for each $m\in \N$ there is some $U_m\in \mathfrak{U}_{n_m}$ such that $(U_m)_{m\in \N}$ is a $k$-sequence.
\end{lemma}
\begin{proof}
     Let $(\mathfrak{U}_n)_{n\in \N}$ be a sequence of open $k$-covers. In particular, $\X$ cannot be compact, so there exists a sequence $(x_n)_{n\in \N}$ which has no cluster point. For $n\in \N$ define
     \begin{align*}
         \mathfrak{V}_n &\coloneqq \{ U \setminus \{x_n\}: U \in \mathfrak{U}_n \}.
     \end{align*}
      We first show that $\mathfrak{V} = \cup_{n\in \N} \mathfrak{V}_n$ is an open $k$-cover. For that, let $K\subseteq \X$ be compact. Since $(x_n)_{n\in \N}$ has no cluster point, there must be some $x_\ell \notin K$. As $\mathfrak{U}_\ell$ is a $k$-cover, there exists some $U\in \mathfrak{U}_\ell$ such that $K\cup\{x_\ell\}\subseteq U$. Therefore, $K\subseteq U\setminus\{x_\ell\}\in \mathfrak{V}$, which shows that $\mathfrak{V}$ is an open $k$-cover. By assumption on $\X$, we can choose a $k$-sequence $(U_m\setminus\{x_{n_m} \})_{m\in \N}$ in $\mathfrak{V}$, where $U_m\in \mathfrak{U}_{n_m}$. We observe that the set of indices $\{n_m: m\in \N\}\subseteq \N$ cannot be bounded, because otherwise $\{x_{n_m}:m\in \N\}$ would be a compact set which is not contained in any of the sets of the $k$-sequence $(U_m\setminus\{x_{n_m} \})_{m\in \N}$, which would be a contradiction. Therefore, there is a subsequence $(n_{m_j})_{j\in \N}$ which is strictly monotonically increasing and going to infinity. Since $(m_j)_{j\in \N}$ is also strictly monotonically increasing, and since every subsequence of a $k$-sequence is still a $k$-sequence, $(U_{m_j})_{j\in \N}$ is the desired $k$-sequence.
\end{proof}

Finally, we are able to show hemicompactness of $\X$ in the next lemma, where the idea of the proof is inspired by \cite[Proposition 5]{Caserta2006}.
\begin{lemma}\label{lem:ksequence_hemicompact}
    Let $(\X, d)$ be a metric space in which every open $k$-cover contains a $k$-sequence. Then $\X$ is locally compact and hemicompact.
\end{lemma}
\begin{proof}
    Assume $\X$ was not locally compact. Let $x\in\X$ be a point which has no compact neighborhood. Therefore, for each $n\in \N$ and compact $K\subset \X$ it must be
    \begin{align*}
        B_{\frac{1}{n}}(x) \setminus (K\cup \{x\}) \neq \emptyset, 
    \end{align*}
    since otherwise $K\cup \{x\}$ would be a compact neighborhood of $x$. 
    Therefore, choose $x_n(K)\in B_{\frac{1}{n}}(x) \setminus (K\cup \{x\})$ and define
    \begin{align*}
        \varepsilon_n(K) \coloneqq 0.5\cdot  \textup{dist}(x_n(K), K\cup\{x\})>0.
    \end{align*}
    Define an open neighborhood of $K\cup\{x\}$ via
    \begin{align*}
        U_n(K) \coloneqq \bigcup_{z\in K \cup \{x\}} B_{\varepsilon_n(K)}(z),
    \end{align*}
    which does not contain $x_n(K)$.
    Then for every $n\in \N$
    \begin{align*}
        \mathfrak{U}_n \coloneqq \{ U_n(K): K\subset \X \textup{ compact} \}
    \end{align*}
    is an open $k$-cover, since for any compact subset $K\subset \X$ it holds that $K\subseteq K\cup \{x\} \subset U_n(K)$. Due to \Cref{lem:select_k_sequence}, there is a subsequence $(\mathfrak{U}_{n_m})_{m\in \N}$ and $A_m\in \mathfrak{U}_{n_m}$ such that $(A_m)_{m\in \N}$ is a k-sequence. We can write $A_m = U_{n_m}(K_m)$ for some compact $K_m\subset \X$.
    Consider the set $$C\coloneqq\{x\} \cup \{x_{n_m}(K_m): m\in \N\}.$$
    As $(n_m)_{m\in\N}\subseteq \N$ is strictly increasing, we have that $x_{n_m}(K_m)\xrightarrow[]{m\to \infty}x$ since $d(x_{n_m}(K_m), x) \leq \nicefrac{1}{n_m}.$ Hence, $C$ is compact as it consists of the elements of the convergent sequence $(x_{n_m}(K_m))_{m\in \N}$ including its limit $x$. For more details on the proof of the compactness of $C$, see \Cref{lem:compact_set_convergent_sequence}.
    However, since $U_{n_m}(K_m)$ does not contain $x_{n_m}(K_m)$, the compact set $C$ is not contained in any $A_m=U_{n_m}(K_m)$, which is a contradiction to $(A_m)_{m\in \N}$ being a $k$-sequence.
    Thus, $\X$ must be locally compact. 
    
    Therefore, every $x\in \X$ has a compact neighborhood, which implies that there exists some $r(x)>0$ such that $\overline{B}_{r(x)}(x)$ is compact. We observe that
    \begin{align*}
       \mathfrak{D}\coloneqq \left\{ \bigcup_{i=1}^n B_{r(x_i)}(x_i): x_i\in \X, n\in \N\right\}
    \end{align*}
    is an open $k$-cover of $\X$, since for every compact $K\subseteq \X$ it is
    \begin{align*}
        K\subseteq \bigcup_{x\in K} B_{r(x)}(x)\subseteq \bigcup_{i=1}^n B_{r(x_i)}(x_i)\in \mathfrak{D},
    \end{align*}
    for finitely many $x_i\in \X$. By assumption, the $k$-cover $\mathfrak{D}$ must contain a $k$-sequence $(U_n)_{n\in \N}$. Since the closure of $U_n$ is compact for all $n\in \N$, it follows that $\X$ is hemicompact.
\end{proof}

Together with \Cref{lem:denseFU_k_sequence}, we have thus shown implication "\labelcref{thm:dense_Frechet_hemicompact:i}$\rightarrow$\labelcref{thm:dense_Frechet_hemicompact:ii}" in \Cref{thm:dense_Frechet_hemicompact}. Namely, we have shown that a metric space $\X$ must be hemicompact if every dense subset of $\CC(\X, \Y)$ is sequentially dense. For completing the proof of \Cref{thm:dense_Frechet_hemicompact}, it only remains to show the following.
\begin{lemma}\label{lem:hemicompact_FU}
    Let $\X$ be a hemicompact metric space and $(\Y, d_\Y)$ any metric space. Then $\CC(\X, \Y)$ equipped with the compact-open topology is a Fréchet-Urysohn space.
\end{lemma}
\begin{proof}
    Let $S\subseteq \CC(\X, \Y)$ and $f\in \overline{S}$ be an adherent point.
    By hemicompactness of $\X$, there are compact sets $(K_n)_{n\in \N}$ such that every compact $K\subseteq \X$ is contained in some $K_n$. Define
    \begin{align*}
        K'_n \coloneqq \bigcup_{i=1}^n K_i.
    \end{align*}
    Then for every compact $K$ there exists some $N\in \N$ such that $K\subseteq K'_n$ for all $n\geq N$.
    According to \Cref{lem:adherent_point_compact_open_topology}, for all $n\in \N$ there is some sequence $(f_{m,n})_{m\in \N}$ in $S\subseteq \CC(\X, \Y)$ such that 
    \begin{align*}
        \sup_{x\in K'_n} d_\Y(f_{m,n}(x)\, ,\, f(x)) \leq \frac{1}{m}.
    \end{align*}
    Hence, for any compact $K\subseteq \X$ there exists an $M\in \N$ such that for all $m\geq M$ it is $K\subseteq K'_m$ and it holds that
    \begin{align*}
        \sup_{x\in K} d_\Y(f_{m,m}(x)\, ,\, f(x)) \leq \sup_{x\in K'_m} d_\Y(f_{m,m}(x)\, ,\, f(x)) \leq \frac{1}{m}.
    \end{align*}
    Therefore, the diagonal sequence $(f_{m,m})_{m\in \N}$ converges uniformly to $f$ on every compact $K\subseteq \X$, which means that convergence is with respect to the compact-open topology according to \Cref{lem:convergence_compact_open_topology}. Thus, $\CC(\X, \Y)$ is a Fréchet-Urysohn space.
\end{proof}
It is to be mentioned that the previous result also follows from \cite[Theorem 7]{Arens1946}, in which it has been shown that $\CC(\X, \Y)$, equipped with the compact-open topology, is metrizable if $\X$ is hemicompact. Note that every metrizable space is a Fréchet-Urysohn space. Moreover, in \cite[Theorem 3.4]{Morita_Topologie} it was shown that for $\X$ being any topological space, $\CC(\X,\R)$ is metrizable if and only if $\X$ is hemicompact.

\section{Approximation properties} \label{sec:approximation_properties}
During this section, we present sufficient properties for metric spaces $\X$ and $\Y$, so that every continuous operator $G\in \CC(\X, \Y)$ can be approximated by the specific encoder-decoder architectures constructed in Section \ref{sec:approx_theorem}. Let us start by defining a general encoder-decoder architecture. 
\begin{definition}[Encoder-decoder architecture]\label{def:encoderDecoder_Architecture}
    Let $\X$ and $\Y$ be metric spaces. Further, for $n,m\in \N$ consider mappings $E:\X\to \K^n$, $D:\K^m\to \Y$ and $\varphi:\K^n \to \K^m$. The corresponding mapping $D\circ \varphi \circ E$ is called an encoder-decoder architecture with encoder $E$ and decoder $D$.
\end{definition}
A similar definition has been given, for example for normed spaces in \cite[Section 4]{Kovachki2024_review}, but therein, the encoder and decoder were supposed to be linear and continuous, which we both do not require here. Recall that one of our main goals is to approximate any operator $G$ by a sequence $G_n$ of encoder-decoder architectures uniformly on every compact $K\subseteq \X$, see statement \labelcref{fig:approximation_types:(B)} in \Cref{fig:approximation_types}.
For that, one has to ensure that the encoding does not loose too much information from the input space $\X$, whereas the decoding can provide enough information in the output space $\Y$. We introduce the following \emph{encoder-decoder approximation properties (EDAP)} of metric spaces which will be sufficient for deriving our universal approximation \Cref{thm:sequential_density_normed}. If approximation as in statement \labelcref{fig:approximation_types:(A)} is desired, we provide similar properties in Section \ref{sec:CEDAP}.

\begin{definition}[Input-EDAP]\label{def:input_EDAP}
    A metric space $(\X, d)$ is said to have the input-EDAP if there are sequences of mappings $E_n^\X:\X\to \K^{w_\X(n)}$ and $D_n^\X:\overline{E^\X_n(\X)}\to \X$ with the following properties:
    \begin{itemize}
        \item[(i)] There is $r:\N\to (0,\infty)$ such that for each compact $K\subseteq \X$ there is $N_K\in \N$ such that for all $n\geq N_K$ it is
        \begin{align*}
            \sup_{f\in K} \vert E_n^\X(f)\vert \leq r(n).
        \end{align*}
        Note that $E_n^\X$ is allowed to be discontinuous.
        \item[(ii)] $D_n^\X$ is continuous.
        \item[(iii)] The mappings $T_n^\X \coloneqq D_n^\X \circ E_n^\X$ satisfy that for every compact $K\subseteq \X$ it is 
        \begin{align*}
            \sup_{f\in K} d\left(f\, ,\, T_n^\X (f)\right) \xrightarrow{n\to\infty}0.
        \end{align*}
    \end{itemize}
\end{definition}

\begin{definition}[Output-EDAP]\label{def:output_EDAP}
     A metric space $(\X, d)$ is said to have the output-EDAP if there are sequences of mappings $E_n^\X:\X\to \K^{w_\X(n)}$ and $D_n^\X:\K^{w_\X(n)}\to \X$ with the following properties:
    \begin{itemize}
        \item[(i)] $E_n^\X$ is continuous.
        \item[(ii)] $D_n^\X$ is uniformly continuous, i.e., for every $\varepsilon>0$ there exists $\delta_{n, \varepsilon}>0$ such that for all $a,b\in \K^{w_\X(n)}$ with $\vert a-b\vert < \delta_{n, \varepsilon}$ it follows that $d(D_n^\X(a), D_n^\X(b)) \leq \varepsilon$.
        \item[(iii)] The mappings $T_n^\X \coloneqq D_n^\X \circ E_n^\X$ satisfy that for every compact $K\subseteq \X$ it is 
        \begin{align*}
            \sup_{f\in K} d\left(f\, ,\, T_n^\X (f)\right) \xrightarrow{n\to\infty}0.
        \end{align*}
    \end{itemize}
\end{definition}
If the space $\X$ is clear from the context, we simply write $E_n\coloneqq E_n^\X$ and $D_n\coloneqq D_n^\X$ for the encoders and decoders, respectively.
\begin{remark}
    In contrast to the input-EDAP, for the output-EDAP, the decoders $D_n$ must be defined on whole $\K^{w_\X(n)}$. This is important to obtain well-defined concatenations $D_n^\Y\circ \varphi_n\circ E_n^\X$ in \Cref{thm:sequential_density_normed}.
\end{remark}

\begin{remark}
    Every space having the input- or output-EDAP must be separable: Let $\X$ have the input- or output-EDAP, which implies that
    \begin{align*}
        \X = \overline{\bigcup_{n=1}^\infty (D_n\circ E_n)(\X)}.
    \end{align*}
    Since $\K^{w_\X(n)}$ is separable and all $D_n$ are continuous, it follows that $(D_n\circ E_n)(\X)$ is separable for all $n\in \N$. Hence, also the union over $n\in \N$, which implies separability of $\X$, as it is the closure of a separable set.
\end{remark}

\begin{remark}
    At first glance, one might think that the output-EDAP implies the input-EDAP. However, for the latter, the function $r:\N\to(0,\infty)$ must be independent of the compact sets $K$, whereas continuity of the encoders in the definition of the output-EDAP leads only to compact-dependent functions $r_K$. 
\end{remark}
Below we state two simple observations. First, both EDAPs are preserved when switching to any equivalent metric. Further, we provide a simple criterion for a subset $M\subset \X$ to inherit the input-EDAP from $\X$. 
\begin{lemma}\label{lem:EDAP_equivalent_metrics}
    Consider two metrics $d_1$ and $d_2$ on a set $\X.$ 
    \begin{itemize}
        \item[(i)] If $\textup{id}:(\X, d_1)\to (\X, d_2)$ is a homeomorphism, then $(\X, d_1)$ has the input-EDAP if and only if $(\X, d_2)$ has the input-EDAP. 
        \item[(ii)] If $\textup{id}:(\X, d_1)\to (\X, d_2)$ and $\textup{id}:(\X, d_1)\to (\X, d_2)$ are uniformly continuous, then $(\X, d_1)$ has the output-EDAP if and only if $(\X, d_2)$ has the output-EDAP. 
        \end{itemize} 
        In any case, the choices for suitable encoders and decoders are the same.
\end{lemma}
\begin{proof}
    This lemma is easy to verify. Note that if the identity $\textup{id}:(\X, d_1)\to (\X, d_2)$ is a homeomorphism, $\X$ has the same compact sets with respect to $d_1$ or $d_2$.
    Further, uniform convergence on compact sets is preserved as the identity is continuous and hence uniformly continuous on any compact set.
\end{proof}
\begin{lemma}\label{lem:subspace_EDAP}
    Let $(\X, d)$ be a metric space having the input-EDAP with encoders $E_n$ and decoders $D_n$. Assume that for $M\subset \X$ it holds that $D_n\left( \overline{E_n(M)}\right)\subseteq M$ for all $n\in \N$. Then $(M, d)$ also has the input-EDAP.
\end{lemma}

The definitions of the input- and output-EDAP are inspired by and resemble the approximation property (AP) of locally convex  topological vector spaces, which has been intensively studied in \cite{Grothendieck}, see also \cite[Chapter 43]{Koethe_TopVecSpaces}. 
For \Cref{def:AP} of the AP in normed spaces, we follow \cite{Casazza2001}, which also contains a comprehensive survey of various types of approximation properties and their interrelations.
In contrast to the EDAPs, the standard AP requires the approximating mappings $T_n$ to be linear and continuous, and allows them to depend on the compact set $K$.
A key feature of the EDAP framework is that it is formulated entirely in terms of the metric structure and does not require compatibility with any underlying linear structure. This allows one to treat spaces such as $p$-Wasserstein spaces ($p \geq 1$), as well as Skorohod spaces. Note that Skorohod spaces are vector spaces, but their topology is not compatible with the linear structure, in the sense that the addition of vectors is not continuous. 
That these metric spaces indeed have the input- or output-EDAP is outlined in Section \ref{sec:metric_EDAP_spaces}. There, we also present natural examples of discontinuous encoders for the input-EDAP and non-Lipschitz decoders suitable for the output-EDAP.

\begin{definition}[Approximation property]\label{def:AP}
    A normed space $\X$ has the approximation property (AP) if for every compact $K\subseteq\X$ there are mappings $T_{K, n}^\X:\X\to \X$ with the following properties:
    \begin{itemize}
        \item[(i)] All $T_{K, n}^\X$ are linear and bounded.
        \item[(ii)] All $T_{K, n}^\X$ map into a finite dimensional subspace of $\X$.
        \item[(iii)] It holds that $$\sup_{f\in K} \Vert T_{K, n}^\X (f) - f \Vert \xrightarrow{n\to\infty} 0.$$
    \end{itemize}
\end{definition}
Below we state \cite[Lemma 22]{Kovachki2023}, which is a universal approximation theorem, in which the AP has been used for constructing suitable encoder-decoder architectures. We provide a similar universal approximation result for more diverse encoder-decoder architectures in \Cref{sec:CEDAP}, and a more relaxed version of the AP.
\begin{theorem}
    Let $\X$ and $\Y$ be $\R$-Banach spaces both having the AP. For every $G\in \CC(\X, \Y)$ and compact set $K\subseteq \X$, there exists a sequence of encoder-decoder architectures $G_{K, n}\in \CC(\X, \Y)$ such that 
    \begin{align*}
        \sup_{f\in K} \big \Vert G(f) - G_{K, n}(f) \big \Vert \xrightarrow{n\to\infty} 0. 
    \end{align*}
\end{theorem}
In other words, the AP enabled universal operator approximation as in statement \labelcref{fig:approximation_types:(A)} of \Cref{fig:approximation_types}, where the approximating sequence $G_{K, n}$ depends on the compact set $K$. However, since our goal is to get rid of this dependence on $K$ in statement \labelcref{fig:approximation_types:(B)}, the standard AP is not a sufficient property, which motivated our definition of the input- and output-EDAP with $K$-independent mappings $T_n$. 
Important examples of spaces possessing both EDAPs are separable normed spaces that have the bounded approximation property (BAP). Since we have already observed that any space with either the input-EDAP or the output-EDAP is necessarily separable, we restrict the definition of the BAP below to separable normed spaces.
For separable Banach spaces, an overview of the relationships among different approximation properties is shown in \Cref{fig:relations_approximation_properties}.

\begin{definition}[Bounded approximation property]\label{def:BAP}
    Let $\lambda>0$. A separable normed space $\X$ has the $\lambda$-bounded approximation property ($\lambda$-BAP) if there exists a sequence of mappings $\T^\X_n: \X \to \X$ with the following properties:
    \begin{itemize}
        \item[(i)] All $\T^\X_n$ are linear with $\Vert T_n^\X (f) \Vert \leq \lambda \Vert f\Vert$ for all $f\in \X$.
        \item[(ii)] All $T_n^\X$ have finite dimensional range.
        \item[(iii)] For every compact $K\subseteq \X$ it holds that 
        \begin{align*}
            \sup_{f\in K}\big\Vert \T^\X_n (f) - f \big\Vert \xrightarrow{n\to\infty}0.
        \end{align*}
    \end{itemize}
    We say that $\X$ has the BAP if there exists some $\lambda>0$ such that it has the $\lambda$-BAP.
    If the space $\X$ is clear from the context, we simply write $\T_n$ instead of $\T^\X_n.$ 
\end{definition}
\begin{remark}
    Note that in the usual definition of the $\lambda$-BAP, covering also non-separable spaces, the mappings $T_n$ depend on the compact set $K$, see for example \cite[Definition 3.1]{Casazza2001} or \cite[Chapter 43.8]{Koethe_TopVecSpaces}. Nevertheless, for separable normed spaces, it is due to the uniform Lipschitz constant $\lambda$, that these are equivalent definitions, see \Cref{lem:Lipschitz_K_independence} or \cite[Corollary 3.4]{Casazza2001}.
\end{remark}

\begin{remark}[Lipschitz BAP]\label{rem:LBAP}
    As a last, but not least, type of approximation property, we want to mention the so-called Lipschitz bounded approximation property ($\lambda$-LBAP). It is similarly defined as the $\lambda$-BAP, but with allowing $T_n$ to be non-linear with Lipschitz constants $\lambda$, see \cite{Godefroy2003_LBAP, Godefroy2015_survey_LBAP}.
    Nevertheless, for a Banach space $\X$, it has been shown in \cite[Theorem 5.3]{Godefroy2003_LBAP} that $\X$ has the $\lambda$-BAP if and only if it has the $\lambda$-LBAP.
\end{remark}

\begin{figure}
    \centering
    \begin{tikzpicture}[node distance=10mm, thick,
                        rblock/.style = {draw, rectangle, rounded corners=0.5em, minimum height=7mm, minimum width=17mm, align=center, fill=gray!10},
                        ]
    \node [rblock]                      (lambdaBAP)     {\large \hyperref[def:BAP]{$\lambda$-BAP}};
    \node [rblock, above=of lambdaBAP]  (inputEDAP)    {\large\raisebox{-2ex}{\hyperref[def:input_EDAP]{Input-EDAP}}};
    \node [rblock, below=of lambdaBAP]  (AP)            {\large \hyperref[def:AP]{AP}};
    \node [rblock, right=of lambdaBAP]  (lambdaLBAP)    {\large \hyperref[rem:LBAP]{$\lambda$-LBAP}};
    \node [rblock, left=of lambdaBAP]   (outputEDAP)     {\large\raisebox{-2ex}{\hyperref[def:output_EDAP]{Output-EDAP}}};
    \draw [-{Computer Modern Rightarrow[length=3mm, width=3mm]}, shorten >=3pt, shorten <=3pt]      (lambdaBAP) -- (AP);
    \begin{scope}[transform canvas={yshift=.5em}]
        \draw [-{Computer Modern Rightarrow[length=3mm, width=3mm]}, shorten >=3pt, shorten <=3pt]      (lambdaBAP) -- (lambdaLBAP);
    \end{scope}
    \begin{scope}[transform canvas={yshift=-.5em}]
        \draw [-{Computer Modern Rightarrow[length=3mm, width=3mm]}, shorten >=3pt, shorten <=3pt]      (lambdaLBAP) -- (lambdaBAP);
    \end{scope}
    \draw [-{Computer Modern Rightarrow[length=3mm, width=3mm]}, shorten >=3pt, shorten <=3pt]      (lambdaBAP) -- (outputEDAP);
    \draw [-{Computer Modern Rightarrow[length=3mm, width=3mm]}, shorten >=3pt, shorten <=3pt]      (lambdaBAP) -- (inputEDAP);
    \end{tikzpicture}
    \caption{Relations between different types of approximation properties for separable Banach spaces. Implications are indicated with an arrow.}
    \label{fig:relations_approximation_properties}
\end{figure}

In the following, we discuss different examples of spaces having the input- and output-EDAP, and describe various explicit choices for suitable encoders and decoders. First, normed spaces are considered, followed by metric spaces. 

\subsection{Examples of normed spaces}\label{sec:normed_examples_EDAP}
Before discussing examples of normed spaces having the EDAPs, we start with a helpful lemma which simplifies proving that a Banach space $\X$ has the BAP, so in particular also both the input- and output-EDAP (if $\X$ is separable). It follows from the Banach-Steinhaus theorem, also known as the uniform boundedness principle. For the latter, we refer to \cite[Chapter 4.2]{Kreyszig_FuncAna}.
\begin{lemma}\label{lem:BAP_pointwise_convergence_suffices}
    Let $\Y$ be a normed and $\X$ be a Banach space. Assume that a sequence of linear, bounded operators $A_n:\X\to\Y$ converges pointwise to some $A:\X\to \Y$. Then $A$ defines a linear and bounded operator and for every compact $K\subseteq \X$ it holds that 
    \begin{align*}
        \sup_{f\in K} \Vert A_n (f) - A(f) \Vert \xrightarrow{n\to\infty} 0.
    \end{align*}
    Further, there is some $\lambda>0$ such that for all $n\in \N$ it holds that $\Vert A_n (f) \Vert \leq \lambda\Vert f\Vert$ for all $f\in \X$.
\end{lemma}
An overview of the constructed encoders and decoders for the provided examples below can be found in \Cref{tab:encoder_decoder_overview}.

\subsubsection{Banach spaces with Schauder bases}\label{sec:Tn_schauder}
As a first straightforward example of spaces that fulfill the BAP, hence both the input- and output-EDAP, we discuss Banach spaces with Schauder bases in what follows. We restrict the attention to infinite-dimensional spaces and for a definition, we follow \cite{Semadeni_SchauderBases}. However, finite-dimensional spaces can be treated analogously.
\begin{definition}[Schauder basis]\label{def:schauder_basis}
    Let $\X$ be a Banach space over $\K$. A sequence $(b_i)_{i\in \N}$ in $\X$ is called a Schauder basis of $\X$ if for every $f\in \X$ there are unique $c_i(f)\in \K$ such that
    \begin{align*}
        f = \sum_{i=1}^\infty c_i(f) b_i.
    \end{align*}
\end{definition}
Note that the uniqueness of the coefficients already implies that the coefficient functionals $c_i:\X\to \K$ are linear and countinuous, since $\X$ is a Banach space. Therefore, and due to \Cref{lem:BAP_pointwise_convergence_suffices}, the following well-known observation is immediate.

\begin{theorem}\label{thm:schauder_basis_BAP}
    Every Banach space $\X$ with a Schauder basis $(b_i)_{i\in \N}$ has the BAP with the projections $T_n=D_n\circ E_n$ defined by
    \begin{align*}
        E_n:\X &\longrightarrow \K^n 
        & \hspace{-1cm}
        D_n:\K^n &\longrightarrow \X \\
        f &\longmapsto \big(c_1(f),\dots,c_n(f)\big)^\intercal,
        & \hspace{-1cm}
        \mu &\longmapsto \sum_{i=1}^n \mu_i b_i,
    \end{align*}
     where $c_i$ are the linear, continuous coefficient functions with respect to the Schauder basis. We call $E_n$ and $D_n$ basis encoders and basis decoders, respectively.
\end{theorem}

Lots of separable Banach spaces possess Schauder bases and hence also have the BAP, input- and output-EDAP. For example, Schauder bases exist in every separable Hilbert space, in several spaces of continuously differentiable functions and  Lebesgue spaces (except $L^\infty$) \cite[Section 1]{Lindenstrauss_classical_banach_spaces}, and in some Sobolev spaces \cite{Fucik1972}.
It is to be noted that not every separable Banach space has a Schauder basis, nor fulfills the BAP. A first counterexample has been found by P. Enflo in 1973 \cite{Enflo1973}.

\begin{remark}[Fr\'echet spaces]\label{rem:schauder_basis_frechet_EDAP}
    For the sake of simplicity, we considered Banach spaces above. Yet, more general spaces can be considered: Let $\X$ be a separable Fr\'echet space, that is, a complete metrizable locally convex space, see e.g. \cite[Chapter 18.2]{Koethe_TopVecSpaces_I}. Let $d$ be any metric inducing the topology on $\X$. Assume that $\X$ has a Schauder basis $(b_i)_{i\in \N}$, which is similarly defined as in \cref{def:schauder_basis}, but in addition, one assumes that all $c_i$ are continuous, see \cite[Chapter 43.5]{Koethe_TopVecSpaces}. Then $(\X, d)$ has both the input- and output-EDAP, where the encoders and decoders can be chosen as in \cref{thm:schauder_basis_BAP}. Uniform convergence of $D_n\circ E_n$ to the identity mapping on any compact $K\subseteq \X$ has been outlined, for example, in \cite[Section 4]{Benth_2022}. Uniform continuity of each $D_n$ follows from \cite[Chapter 15.2 (4)]{Koethe_TopVecSpaces_I}. Continuity of each $E_n$ is clear. According to \cite[Chapter 15.6 (5)]{Koethe_TopVecSpaces_I}, $E_n$ maps bounded sets to bounded sets, which implies property (i) in \cref{def:input_EDAP}.
\end{remark}

\subsubsection{Frames in Hilbert spaces}\label{sec:Tn_frames}
Frames are a generalization of Schauder bases, as they relax the uniqueness constraint on the basis coefficients $c_i$. For the sake of a simpler notation, we consider in what follows the classical case of an infinite-dimensional, separable Hilbert space $\X$ with inner product $\langle \cdot, \cdot\rangle.$ However, note that finite-dimensional spaces can be handled in a similar way \cite[Chapter 1]{Christensen_Frames}. Furthermore, frames can be generalized to non-separable Hilbert spaces \cite{Bilalov_Frames_nonSeparableSpace} as well as to a Banach space setting \cite[Chapter 24]{Christensen_Frames}.
\begin{definition}[Frame]\label{def:frame}
     A frame of an infinite-dimensional, separable Hilbert space $\X$ is a sequence $(f_i)_{i\in \N}$ in $\X$, if there exist constants $A,B>0$ such that for all $f\in \X$ it holds that 
    \begin{align*}
        A\Vert f \Vert^2 \leq \sum_{i=1}^\infty \vert \langle f, f_i\rangle\vert^2 \leq B \Vert f\Vert ^2.
    \end{align*}
    The values $A$ and $B$ are called frame bounds.
\end{definition}

For a given a frame $(f_i)_{i\in \N}$ in $\X$, it is well-known that there exists a canonical dual frame of $(f_i)_{i\in \N}$ defined by $\big(S^{-1}(f_i)\big)_{i\in \N} \subset \X$ via the inverse of the frame operator $S: \X \to\X$ given by $S(f) \coloneqq \sum_{i=1}^\infty\langle f, f_i\rangle f_i,$ so that every $f\in \X$ can be expressed as
\begin{equation*}
    f = \sum_{i=1}^\infty \big\langle f,\, S^{-1} (f_i)\big\rangle f_i,
\end{equation*}
see \cite[Theorem 5.1.6]{Christensen_Frames}. This motivates the definition of, more generally, dual frames.

\begin{definition}[Dual Frame]\label{def:dualFrame}
    Let $\X$ be an infinite-dimensional, separable Hilbert space and $(f_i)_{i\in \N} \subset \X$ a given frame of $\X.$ A frame $(f^*_i)_{i\in \N} \subset \X$ of $\X$ is called a dual frame of $(f_i)_{i\in \N},$ if for all $f\in \X$ it holds that
    \begin{equation}\label{eq:def:dualFrame}
        f = \sum_{i=1}^\infty \langle f, f^*_i\rangle f_i.
    \end{equation}
\end{definition}

A frame has a dual frame different from the canonical dual frame if and only if it is not a Riesz basis, see, e.g., \cite[Lemma 6.3.1]{Christensen_Frames}. 
In the following, we state the main result of this section, which will give rise to the construction of encoders and decoders based on frames.
\begin{theorem}\label{thm:frames}
    Let $\X$ be an infinite-dimensional, separable Hilbert space and $(f_i)_{i\in \N} \subset \X$ a given frame of $\X$ along with a dual frame $(f^*_i)_{i\in \N} \subset \X$ of $(f_i)_{i\in \N}.$ Then $\X$ has the BAP with the projections $T_n = D_n\circ E_n$ defined by
    \begin{align*}
        E_n:\X &\longrightarrow \K^n 
        & \hspace{-1cm}
        D_n:\K^n &\longrightarrow \X \\
        f &\longmapsto \big( \langle f, f^*_1 \rangle, \dots, \langle f, f^*_n \rangle \big)^\intercal,
        & \hspace{-1cm}
        \mu &\longmapsto \sum_{i=1}^n \mu_i f_i.
    \end{align*}
    We call $E_n$ and $D_n$ frame encoders and frame decoders, respectively.
\end{theorem}
\begin{proof}
    The statement follows immediately by the definition of the projections $T_n$ and by applying \Cref{lem:BAP_pointwise_convergence_suffices}.
\end{proof}

\subsubsection{Sampling in spaces of continuous functions}\label{sec:Tn_sampling}
Given some compact metric space $\Omega$ and $a\in \N$, we consider the space $\X=\CC(\Omega, \K^a)$ equipped with the supremum norm throughout this section.
If $\Omega$ is uncountable, it is due to Milutin's theorem that $\CC(\Omega, \K^a)$ is isomorphic to $\CC([0,1], \K^a)$, see for example \cite[Chapter 4.i]{Lindenstrauss_classical_banach_spaces}. Since the latter has a Schauder basis, according to \cite[Chapter 1]{Semadeni_SchauderBases}, also $\CC(\Omega, \K^a)$ must possess a Schauder basis, meaning that it has the BAP according to \Cref{sec:Tn_schauder}.
However, the corresponding choice of $T_n(f) = \sum_{i\leq n} c_i(f) b_i$ is not the only possible choice for $T_n$ in \Cref{def:BAP} of the BAP, which we want to outline in what follows. 
The constructions within this section are inspired from \cite[Lemma 7]{Chen1995_approx_thm}, in which the case $\CC(\Omega, \R)$ and $\Omega$ being a compact subset of a Banach space was considered.

\begin{definition}[$\varepsilon$-covering]
    Let $\Omega$ be a metric space and $\varepsilon>0$. An $\varepsilon$-covering of $\Omega$ is a set $M \subseteq \Omega$ such that 
    \begin{align*}
        \Omega \subseteq \bigcup_{x\in M}B_\varepsilon(x)
    \end{align*}
    Note that if $\Omega$ is compact, there exists a finite $\varepsilon$-covering for every $\varepsilon>0$.
\end{definition}

\begin{lemma}\label{lem:partition_unity}
    Let $(\Omega, d)$ be a compact metric space. Further, let $\varepsilon>0$ and $\{y_1,\dots, y_k\}$ be an $\varepsilon$-covering for $\Omega$. Then there exist finitely many continuous mappings $P_{\varepsilon, i}:\Omega \to [0,1]$ for $i\in \{1,\dots,k\}$ such that 
    \begin{enumerate}[(i)]
        \item For every $y\in \Omega$ it is $\sum_{i=1}^{k} P_{\varepsilon, i}(y) = 1$.\label{lem:partition_unity:i}
        \item For every $y\in \Omega$ it holds that $P_{\varepsilon, i}(y) = 0$ whenever $d( y\, ,\, y_i) \geq \varepsilon.$\label{lem:partition_unity:ii}
    \end{enumerate}
    If $\Omega\subset \R^p$, the mappings $P_{\varepsilon, i}$ can be chosen to be a smooth function belonging to $\CC^\infty(\Omega, \R)$.
\end{lemma}
\begin{proof}
    For the construction, we start by defining the continuous functions
    \begin{align*}
        \Tilde{P}_{\varepsilon, i}:\Omega &\longrightarrow [0,\infty) \\
        y&\longmapsto \begin{cases}
            \exp{\left(-\frac{1}{\varepsilon^2 - d(y,y_i)^2}\right)} &\text{if}\quad d(y,y_i)<\varepsilon, \\
            0& \text{else}.
        \end{cases}
    \end{align*}
    Note that property \labelcref{lem:partition_unity:ii} is already satisfied for $\Tilde{P}_{\varepsilon, i}$. In order to ensure property \labelcref{lem:partition_unity:i}, we normalize the functions by defining
    \begin{equation}\label{eq:P_eps_i}
         P_{\varepsilon, i}(y) \coloneqq \frac{ \Tilde{P}_{\varepsilon, i}(y)}{\sum_{\ell=1}^{k} \Tilde{P}_{\varepsilon, \ell}(y)}.
    \end{equation}
    By definition of $\{y_1,\dots, y_k\}$ being an $\varepsilon$-covering for $\Omega,$ note that for each $y\in \Omega$ there exists $i\in \{1,\dots,k\}$ such that $y\in B_\varepsilon(y_i)$. Hence, $\sum_{\ell=1}^{k} \Tilde{P}_{\varepsilon, \ell}(y)> 0$ for all $y\in \Omega$ so that the $P_{\varepsilon, i}$ are well-defined. Additionally, the $P_{\varepsilon, i}$ fulfill both properties \labelcref{lem:partition_unity:i} and \labelcref{lem:partition_unity:ii} leading to the desired partition of unity.\newline
    In the case of $\Omega\subseteq \R^p$, it is well-known that the so-called bump functions $\Tilde{P}_{\varepsilon, i}$ belong to $\CC^\infty(\Omega, \R)$ (see e.g.\ \cite[Page 36]{Adams_SobolevSpaces}). Together with the fact that $\sum_{\ell=1}^{k} \Tilde{P}_{\varepsilon, \ell}(y)> 0$ for all $y\in \Omega$, this leads to $P_{\varepsilon, i}\in \CC^\infty(\Omega, \R)$.
\end{proof}

\begin{theorem}\label{thm:sampling_BAP}
    Let $\Omega$ be a compact metric space and consider $\CC(\Omega, \K^a)$ equipped with the supremum-norm. For $n\in \N$ let $\left\{y_1^{(n)},\dots,y_{k(n)}^{(n)} \right\}$ be an $1/n$-covering for $\Omega$ and $P_{\frac{1}{n}, i}$ the partition of unity on $\Omega$ from \Cref{lem:partition_unity}. Then $\CC(\Omega, \K^a)$ has the BAP with the mappings $T_n = D_n\circ E_n$ defined by
    \begin{align*}
        E_n: \CC(\Omega, \K) &\longrightarrow \K^{k(n)}
        & \hspace{-1cm}
        D_n: \K^{k(n)} &\longrightarrow \CC(\Omega, \K) \\
        f &\longmapsto \Big(f(y_1^{(n)}),\dots, f(y_{k(n)}^{(n)})\Big)^\intercal,
        & \hspace{-1cm}
        \mu &\longmapsto \sum_{i=1}^{k(n)} \mu_i P_{\frac{1}{n}, i}.
    \end{align*}
    We call $E_n$ and $D_n$ sampling encoders and sampling decoders, respectively.
\end{theorem}

\begin{proof}
    First, we mention that all $P_{\frac{1}{n}, i}$ are continuous which means that $\T_n$ indeed maps into $\CC(\Omega, \K^a)$.
    Let $f\in \CC(\Omega, \K^a)$ and $\varepsilon>0$. As $\Omega$ is compact, $f$ is uniformly continuous due to the Heine-Cantor theorem. Hence, there exists some $N \in \N$ such that for all $n\geq N$ it holds that
    \begin{align}\label{eq:thm_sampling_UAP_1}
        \vert f(x) - f(y) \vert \leq \varepsilon \textup{ \, whenever \, } x,y\in \Omega \textup{ \, with \, } d(x,y) < \frac{1}{n}.
    \end{align}
    Given $y\in \Omega$ let $I(y,n)$ denote the set of all indices $i$ such that $d(y\, , y_i^{(n)}) < 1/n$.
    Due to the properties of $P_{\frac{1}{n}, i}$, see \Cref{lem:partition_unity}, we conclude for $n\geq N$ that
    \begin{align*}
        \Big\vert f(y) - (\T_n f)(y) \Big\vert &= \left\vert  \sum_{i\in I(y,n)} \left( f(y)- f(y_i^{(n)})\right) P_{\frac{1}{n}, i}(y)\right\vert\\
        &\leq \sum_{i\in I(y,n)} \left \vert  f(y)- f(y_i^{(n)})\right\vert  P_{\frac{1}{n}, i}(y) \\
        &\overset{\labelcref{eq:thm_sampling_UAP_1}}{\leq} \varepsilon \sum_{i\in I(y,n)} P_{\frac{1}{n}, i}(y)\leq \varepsilon \sum_{i=1}^{k(n)} P_{\frac{1}{n}, i}(y) = \varepsilon.
    \end{align*}
    Taking the supremum over all $y\in \Omega$ shows that $\T_n f$ converges to $f$. By \Cref{lem:BAP_pointwise_convergence_suffices}, this implies that $\T_n$ uniformly converges to the identity operator on every compact $K\subseteq \CC(\Omega, \K^a)$. The range space of $T_n$ is spanned by the vectors $b_{i,j} = e_j P_i$, where $\{e_j: j = 1,\dots,a\}$ is the standard basis of $\K^a.$ Hence, $T_n$ has finite rank, namely at most $k(n) a$.
\end{proof}

\begin{remark}
    In \cite[Lemma 7]{Chen1995_approx_thm}, a different partition of unity was used by defining
    \begin{align*}
        \Tilde{P}_{\varepsilon, i}:\Omega &\longrightarrow [0,1] \\
        y&\longmapsto \begin{cases}
            1- \varepsilon^{-1}d(y, y_i)\textup{ \, if \, } d(y, y_i) < \varepsilon, \\
            0 \textup{ \, else.}
        \end{cases}
    \end{align*}
    The advantage of using the $\Tilde{P}_{\varepsilon, i}$ as in the proof of \Cref{lem:partition_unity} is their smoothness if $\Omega\subset \R^p$, see \Cref{coro:sampling_BAP}.
\end{remark}

\begin{corollary}\label{coro:sampling_BAP}
    For compact $\Omega\subset \R^p$, the mappings $T_n$ in \Cref{thm:sampling_BAP} are also well-defined as mappings $\X\to \X$ for $\X$ being $\CC^k(\Omega, \R^a)$ or a Sobolev space $W^{k,2}(\Omega,\R)$ that continuously embeds into $\CC(\Omega, \R)$. In either case, $T_n$ uniformly converges to the identity on $\X$ on every compact $K\subset \X$, where $\X$ is equipped with the supremum norm. 
\end{corollary}

Note that the spaces in \Cref{coro:sampling_BAP} have the BAP, but are not complete.
When $\CC^k$ is equipped with the $\CC^k$-norm, instead of the supremum norm, one can also show that it has the BAP with suitable sampling encoders. We give an impression of how to construct suitable $T_n$ in Appendix \ref{app:sampling_Ck}.

Finally, we provide in Appendix \ref{sec:appendix_coeff_sampling} a comparison of the above sampling encoders and the basis encoders discussed in Section \ref{sec:Tn_schauder}. In \Cref{thm:sampling_schauder_differ_main} we show that, in general, Schauder basis encoders cannot coincide with sampling encoders for $n$ large enough.

\subsubsection{Substitution by dense sets}\label{sec:Tn_substitute_dense_set}
Throughout this section, let $\X$ be a separable normed space having the BAP, with suitable mappings $T_n$ from \Cref{def:BAP}. As $T_n$ maps into a finite dimensional space, we can choose a basis $(b_i^{(n)})_{i\in I_n}$ for $T_n(\X)$, with $I_n$ being a finite index set. Therefore, we can write
\begin{align}\label{eq:BAP_range_basis_representation}
    T_n f = \sum_{i\in I_n} c^{(n)}_i(T_n f) b_i^{(n)},
\end{align}
where $c_i^{(n)}:T_n(\X)\to \K$ are linear and continuous functionals.
For some dense set $S\subseteq \X$, we will see that one can also replace the basis $b_i^{(n)}$ of $T_n(\X)$ by some vectors from $S$, without loosing the ability to approximate the identity uniformly on compact sets. The decoders below will be useful in Section \ref{sec:examples_architectures} when discussing classical DeepONets or MIONets.

\begin{theorem}\label{thm:Tn_substitute_dense_set}
     Let $S\subseteq \X$ be dense. Then for every $n\in \N$ with $I_n$ being the finite index set in \Cref{eq:BAP_range_basis_representation}, there exist $(v_i^{(n)})_{i\in I_n}\subset S$ such that the mappings $\tilde{T}_n = D_n\circ E_n$ defined by
    \begin{align*}
        E_n: \X &\longrightarrow \K^{\vert I_n\vert}
        & \hspace{-1cm}
        D_n: \K^{\vert I_n \vert} &\longrightarrow \X \\
        f &\longmapsto \Big(c^{(n)}_1(T_n (f)),\dots, c^{(n)}_{\vert I_n\vert}(T_n (f))\Big)^\intercal,
        & \hspace{-1cm}
        \mu &\longmapsto \sum_{i=1}^{\vert I_n \vert} \mu_i v_i^{(n)},
    \end{align*}
    converge uniformly to the identity on every compact $K\subseteq \X$. Hence, $\tilde{T}_n$ are suitable choices in \Cref{def:BAP} of the BAP. We call $E_n$ and $D_n$ auxiliary encoders and dense decoders, respectively.
\end{theorem}
\begin{proof}
Since the $c_i^{(n)}$ are bounded, linear functionals on $T_n(\X)$, there exists some $p_n>0$ such that
\begin{align*}
    \sum_{i\in I_n} \vert c_i^{(n)}(T_n f) \vert \leq p_n \Vert T_n f \Vert \leq p_n \Vert T_n \Vert \Vert f\Vert,
\end{align*}
where $\Vert T_n\Vert$ denotes the operator norm of $T_n$. Since $S$ in dense in $\X$, there are $v_i^{(n)}\in S$ such that \begin{align*}
    \Vert b_i^{(n)} - v_i^{(n)} \Vert < \frac{1}{np_n \Vert T_n \Vert}.
\end{align*} 
 Further, for any $f\in \X$ it follows that 
\begin{align*}
    \Vert f - \tilde{T}_n f \Vert &\leq \Vert f - T_n f\Vert + \Vert T_n f - \tilde{T}_n f\Vert \\
    &< \Vert f - T_n f\Vert + \frac{\Vert f \Vert}{n}.
\end{align*}
Therefore, $\tilde{T}_n$ uniformly converges to the identity on every compact set $K\subseteq \X$, since $T_n$ has this property and since compact sets are bounded.
\end{proof}

If $\X$ is a Hilbert space, it is possible to construct an encoder that also uses $v_i^{(n)}$. This will be of interest in \Cref{sec:basisONet} when discussing BasisONets \cite{Hua2023}.
\begin{theorem}\label{thm:Tn_substitute_dense_set_HilbertSpace}
    Assume that $\X$ is a separable Hilbert space and $S\subseteq \X$ be dense. Then for every $n\in \N$ there are $v_1^{(n)},...,v_n^{(n)}\in S$ such that the mappings $\tilde{T}_n = D_n\circ E_n$ defined by
    \begin{align*}
        E_n: \X &\longrightarrow \K^{n}
        & \hspace{-1cm}
        D_n: \K^{n} &\longrightarrow \X\\
        f &\longmapsto \Big(\langle f, v_1^{(n)}\rangle,\dots,\langle f, v_{n}^{(n)}\rangle\Big)^\intercal,
        & \hspace{-1cm}
        \mu &\longmapsto \sum_{i=1}^{n} \mu_i v_i^{(n)},
    \end{align*}
    converge uniformly to the identity on every compact $K\subseteq \X$. Hence, $\tilde{T}_n$ are suitable choices in \Cref{def:BAP} of the BAP. We call $E_n$ dense encoders. Note that $D_n$ are dense decoders as in \Cref{thm:Tn_substitute_dense_set}
\end{theorem}
\begin{proof}
    Choose an orthonormal basis $(b_i)_{i\in \N}$ of $\X$. Then the mappings $T_n$ constructed in \Cref{sec:Tn_schauder} are given by
    \begin{align*}
        T_n:\X &\longrightarrow \X\\
        f &\longmapsto \sum_{i=1}^n \langle f, b_i\rangle b_i.
    \end{align*} 
    For any $n\in \N$, choose $v_1^{(n)},...,v_n^{(n)}\in S$ such that
    \begin{align*}
        \sum_{i=1}^n\Vert v_i^{(n)} - b_i\Vert \leq\frac{1}{3n}.
    \end{align*}
    Note that $\Vert v_i^{(n)}\Vert \leq 1 + \Vert v_i^{(n)} - b_i \Vert$. For all $f\in \X$, it follows from the triangular and Cauchy-Schwartz inequality that
    \begin{align*}
        \Vert \tilde{T}_n f - T_n f \Vert 
        &\leq \left \Vert \sum_{i=1}^n \langle f, v_i^{(n)}\rangle v_i^{(n)} - \langle f, v_i^{(n)}\rangle b_i \right \Vert + \left \Vert \sum_{i=1}^n \langle f, v_i^{(n)}\rangle b_i - \langle f, b_i\rangle b_i \right \Vert \\
        &\leq \Vert f \Vert \sum_{i=1}^n \Vert v_i^{(n)} \Vert  \Vert v_i^{(n)} - b_i \Vert + \Vert f\Vert\sum_{i=1}^n \Vert v_i^{(n)} - b_i \Vert \\
        &\leq \Vert f \Vert \sum_{i=1}^n (2+\Vert v_i^{(n)} - b_i \Vert) \Vert v_i^{(n)} - b_i \Vert \\
        &\leq 3\Vert f \Vert \sum_{i=1}^n \Vert v_i^{(n)} - b_i \Vert \leq \frac{\Vert f \Vert}{n}.
    \end{align*}
    Hence, the claim follows from $\Vert f - \tilde{T}_n f \Vert \leq \Vert f - T_n f\Vert + \Vert T_n f - \tilde{T}_n f\Vert$ with the same arguments as in the previous theorem.
\end{proof}
\begin{remark}
    Instead of an orthonormal basis $(b_i)_{i\in \N}$, one can also consider a frame $(f_i)_{i\in \N}$ along with a dual frame $(f_i^*)_{i\in \N}$ in \Cref{thm:Tn_substitute_dense_set_HilbertSpace}. For the decoder, choose $v_i^{(n)}\in S$ that approximates $f_i$. For the encoder, choose $w_i^{(n)}$ that approximates $f_i^*$. 
\end{remark}

\subsection{Examples of metric spaces}\label{sec:metric_EDAP_spaces}
In this section, we discuss several well-known metric spaces that are neither normed spaces nor even topological vector spaces, yet possess the input- or output-EDAP. These examples demonstrate the advantage of formulating the EDAPs in the general setting of metric spaces.
Classical approximation properties such as the AP and BAP are based on approximation by finite-rank linear operators and are therefore naturally associated with Banach spaces or more generally locally convex topological vector spaces \cite{Grothendieck, Koethe_TopVecSpaces, Casazza2001}.
Furthermore, we give examples of discontinuous, though natural, encoders that are nevertheless suitable for the input-EDAP. In addition, we provide examples of uniformly continuous but non-Lipschitz decoders that are suitable for the output-EDAP. Note that some examples were already given in \cref{rem:schauder_basis_frechet_EDAP}.

\subsubsection{Wasserstein spaces}\label{sec:wasserstein}
Throughout this section, $(\Omega, d)$ denotes a complete, separable metric space. 
For $p\geq 1$, denote by $\PP_p(\Omega;d)$ the set of Borel probability measures on $\Omega$ with finite moments of order $p$. So $\mu\in \PP_p(\Omega;d)$ if for some $x_0\in \Omega$ (hence for any $x_0$) it is
\begin{align*}
    \int_\Omega d(x, x_0)^p \diff\mu(x) < \infty.
\end{align*}
Note that if $d$ is a bounded metric, $\PP_p(\Omega;d)$ coincides with the set $\PP(\Omega)$ of all probability measures.
The main goal is to show in \Cref{thm:Wasserstein_DEDAP} that if $\Omega$ is proper,  $\PP_p(\Omega;d)$ has the input-EDAP when being equipped with the $p$-Wasserstein distance. In \Cref{coro:wasserstein_comp_support}, we show that the set of compactly supported measures also has the input-EDAP. If $\Omega$ is compact, also the output-EDAP is fulfilled, see \Cref{thm:wasserstein_compact_omega}. Further, we verify both EDAPs for the set of absolutely continuous measures in $\PP_p(\R^m; \vert \cdot \vert)$ in \Cref{coro:wasserstein_abs_cont_EDAP}.
For non-compact $\Omega$, it remains an open question whether $\PP_p(\Omega; d)$ also has the output-EDAP.

\begin{definition}[Proper metric space]
    A metric space $(\Omega, d)$ is called proper if for every $R>0$ and $x\in \Omega$, the closed ball $\overline{B}_R(x_0)=\{x\in \Omega: d(x, x_0)\leq R\}$ is compact.
\end{definition}
For example, $\Omega=\R^m$ or any of its closed subsets are proper, equipped with any norm. Note that every proper metric space is complete and separable.
We consider the $p$-Wasserstein metric, also called Kantorovich-Rubinstein distance, where we follow \cite[Section 7.1]{Villani_Optimal_Transport}. 
\begin{definition}[$p$-Wasserstein distance]
    For $p\geq 1$, the $p$-Wasserstein (or Kantorovich-Rubinstein) distance between probability measures $\mu, \nu\in \PP_p(\Omega; d)$ is defined as
    \begin{align*}
        W_p(\mu, \nu) \coloneqq \left(\inf_{\pi\in \Pi(\mu, \nu)} \int_{\Omega\times \Omega} d(x, y)^p \diff\pi(x, y)\right)^{1/p},
    \end{align*}
    where $\Pi(\mu, \nu)$ is the set of Borel probability measures on $\Omega\times \Omega$ such that $\pi(A\times \Omega) = \mu(A)$ and $\pi(\Omega\times A)=\nu(A)$ for every Borel set $A\subseteq\Omega$.
\end{definition}
\begin{remark}\label{rem:wasserstein_W1_dual}
    Although the definition of $W_p$ depends on the choice of the metric $d$ on $\Omega$, we omit this fact within the notation for the sake of readability. Furthermore, $W_p$ indeed defines a metric on $\PP_p(\Omega;d)$, see e.g.\ \cite[Theorem 7.3]{Villani_Optimal_Transport}. In case of $p=1$, the $1$-Wasserstein distance can be expressed as
    \begin{align*}
        W_1(\mu, \nu) = \sup_{f \in \Lip_1(\Omega)} \int_\Omega f \diff(\mu - \nu),
    \end{align*}
    where $\Lip_1(\Omega)$ is the set of $1$-Lipschitz functions $f:\Omega\to \R$, see \cite[Remark 7.5]{Villani_Optimal_Transport}.
\end{remark}
In order to show that $(\PP_p(\Omega; d), W_p)$ has the input-EDAP, we construct suitable encoders and decoders in \Cref{def:wasserstein_encoder_decoder}. We begin with the following elementary result.

\begin{lemma}\label{lem:cover_omega_I}
    Let $(\Omega, d)$ be a proper metric space. Fix $x^*\in \Omega$. For any $a>0$ and $r>0$, there exist finitely many disjoint, non-empty Borel sets $I_1^{(a,r)},...,I_{k(a,r)}^{(a,r)}\subseteq \overline{B}_r(x^*)$ with diameter $<a$ and which cover $\overline{B}_r(x^*)$.
\end{lemma}
\begin{proof}
    This follows from the assumption on $\Omega$ that all closed balls $\overline{B}_r(x^*)$ are compact.
\end{proof}

\begin{definition}\label{def:wasserstein_encoder_decoder}
    Let $(\Omega, d)$ be a proper metric space. For fixed $x^*\in \Omega$, $a,r>0$ and the corresponding sets $I_i^{(a,r)}$ from \Cref{lem:cover_omega_I}, we define encoders
    \begin{equation*}
        \begin{aligned}
            E_{a,r}:\PP_p(\Omega;d) &\longrightarrow A_{a,r}\coloneqq \{c\in \R^{k(a,r)}: c_i\geq 0,\ \sum_i c_i \leq 1\} \\
            \mu &\longmapsto \big(\mu(I^{(a,r)}_1),..., \mu(I^{(a,r)}_{k(a,r)})\big)^T
        \end{aligned}
    \end{equation*}
    and, after choosing $x_1^{(a,r)} \in I^{(a,r)}_1,..., x_{k(a,r)}^{(a,r)}\in I^{(a,r)}_{k(a,r)}$, decoders via
    \begin{equation*}
        \begin{aligned}
            D_{a,r}: A_{a,r} &\longrightarrow \PP_p(\Omega;d)\\
            c &\longmapsto \sum_{i=1}^{k(a,r)} c_i \delta_{x_i^{(a,r)}} + \left( 1 - \sum_{i=1}^{k(a,r)} c_i \right) \delta_{x^*}.
        \end{aligned}
    \end{equation*}
\end{definition}

Note that $E_{a,r}$ is not continuous with respect to the $p$-Wasserstein metric $W_p$, simply because $W_p$-convergence does not necessarily imply strong convergence of probability measures. Furthermore, note that $A_{a,r}\subset \R^{k(a,r)}$ is closed and bounded, hence compact.

\begin{lemma}\label{lem:wasserstein_equivalence}
    Assume the metric $d$ on $\Omega$ is bounded. Then for any $p\geq1$ and all $\mu, \nu\in \PP(\Omega)=\PP_p(\Omega;d)= \PP_1(\Omega;d)$ it is
    \begin{align*}
        W_1(\mu, \nu) \leq W_p(\mu, \nu) \leq \sup\{d(x,y):x,y\in \Omega\}^{1-1/p} W_1(\mu, \nu)^{1/p}.
    \end{align*}
    In other words, all $p$-Wasserstein distances are equivalent. 
\end{lemma}
\begin{proof}
    This follows from applying the Hölder inequality.
\end{proof}
The following two results are extracted from the proof in \cite[Theorem 7.12]{Villani_Optimal_Transport}. For $R>0$, denote by $W^{(R)}_p(\mu, \nu)$ the $p$-Wasserstein distance between $\mu, \nu\in \PP_p(\Omega;d_R)$ w.r.t. to the bounded metric $d_R(x,y)\coloneqq \min\{d(x,y), R\}$ on $\Omega$.
\begin{lemma}\label{lem:metric_p_estimate}
    For $p\geq 1$, there exists $c'_p>0$ such that for every $R>0$ and $x,y,x_0\in \Omega$ it is
    \begin{align*}
        d(x, y)^p \leq c_p'\Big(d_R(x,y)^p
        &+ d(x, x_0)^p \mathbbm{1}_{d(x, x_0)\geq R/2}
        \\
        &+ d(x_0, y)^p \mathbbm{1}_{d(x_0, y)\geq R/2}\Big).
    \end{align*}
\end{lemma}
\begin{proof}
    Let $R>0$ and $x,y,x_0\in \Omega$. Due to the triangular inequality it is either $d(x, y) \leq 2 d(x, x_0)$ or $d(x,y)\leq 2 d(x_0, y)$. Therefore, it is
    \begin{align*}
        d(x, y)&\leq \min\{d(x,y), R\} + 2d(x, x_0) \mathbbm{1}_{d(x, x_0)\geq R/2} + 2d(x_0, y)\mathbbm{1}_{d(x_0, y)\geq R/2}.
    \end{align*}
    The claim then follows from applying Jensen's inequality to the convex mapping $a\mapsto a^p$ for $a\geq 0$.
\end{proof}

\begin{corollary}\label{coro:bounded_wasserstein}
    For every $p\geq 1$, there exists $c'_p>0$ such that for every $R>0$, $x_0\in \Omega$ and every $\mu, \nu\in \PP_p(\Omega; d)$ it is
    \begin{align*}
        W_p(\mu, \nu)^p \leq c'_p \left( R^{p-1}W^{(R)}_1(\mu, \nu)  + \int_{d(x, x_0) \geq R/2} d(x, x_0)^p \diff(\mu+\nu)\right).
    \end{align*}
\end{corollary}
\begin{proof}
    Let $R>0$ and $\mu, \nu\in \PP_p(\Omega;d)$. According to in \cite[Theorem 1.3]{Villani_Optimal_Transport}, there exists an \emph{optimal transport plan} $\pi'\in \Pi(\mu, \nu)$ between $\mu$ and $\nu$ with respect to the cost function $d_R(x,y)^p=\min\{d(x,y), R\}^p$, i.e., 
    \begin{align*}
        W^{(R)}_p(\mu, \nu)^p = \int_{\Omega\times \Omega} d_R(x, y)^p\diff \pi'(x, y).
    \end{align*}
    Choose $c'_p>0$ as in \Cref{lem:metric_p_estimate}. It follows that
    \begin{align*}
        \frac{1}{c'_p}W_p(\mu, \nu)^p &\leq  \frac{1}{c'_p}\int_{\Omega\times \Omega} d(x, y)^p\diff \pi'(x, y) \\
        &\leq \int_{\Omega \times \Omega }d_R(x,y)^p\diff \pi'(x,y)+ \int_{d(x, x_0)\geq R/2} d(x, x_0)^p\diff \pi'(x, y) \\ & \phantom{as}+ \int_{d(x_0,y)\geq R/2} d(x_0, y)^p\diff \pi'(x, y) \\
        &= W^{(R)}_p(\mu, \nu)^p + \int_{d(x, x_0) \geq R/2} d(x, x_0)^p \diff \mu(x)  \\ & \phantom{as}+ \int_{d(x_0,y) \geq R/2} d(x_0, y)^p \diff \nu(y),
    \end{align*}
    where the latter follows by definition of $\Pi(\mu, \nu)\ni \pi'$. Hence, the claim follows from the symmetry of the metric $d$ and the estimate $W^{(R)}_p(\mu, \nu)\leq R^{1-1/p}W_1^{(R)}(\mu, \nu)^{1/p}$, pointed out in \Cref{lem:wasserstein_equivalence}.
\end{proof}

\Cref{coro:bounded_wasserstein} can be used to show Hölder continuity of the decoders $D_{a,r}$. Note that the assumption of $(\Omega,d)$ being proper in the next result is only relevant for the specific selection of points $x_i^{(a,r)}.$

\begin{corollary} \label{coro:decoder_HoelderContinuous_Wasserstein}
    Let $(\Omega, d)$ be a proper metric space. For every $p\geq 1$ and $a,r>0$, the decoder $D_{a,r}: A_{a,r} \to(\PP_p(\Omega;d), W_p)$ defined in \Cref{def:wasserstein_encoder_decoder} is Hölder continuous with exponent $1/p$. The domain $A_{a,r}$ can be equipped with any norm on $\R^{k(a,r)}$.
\end{corollary}
\begin{proof}
    For fixed $a, r>0$, one can choose $R>0$ large enough such that
    \begin{equation*}
        \max \{d(x_i^{(a,r)}, x^*): i=1,...,k(a,r)\}\leq r < R/2.
    \end{equation*}
    Hence, for all $c\in A_{a,r}$ it is
    \begin{align*}
        \int_{d(x, x^*) \geq R/2} d(x, x^*)^p \diff(D_{a,r}(c)) = 0.
    \end{align*}
    Therefore, and by using Corollary \ref{coro:bounded_wasserstein}, we have for arbitrary $b,c\in A_{a,r}$ 
    \begin{align*}
        W_p(D_{a,r}(b), D_{a,r}(c))^p &\leq c'_p R^{p-1}W^{(R)}_1(D_{a,r}(b), D_{a,r}(c))\\
        &\leq c'_p R^{p-1} W_1(D_{a,r}(b), D_{a,r}(c)) \\
        &=c'_p R^{p-1} \sup_{\substack{f\in \Lip_1(\Omega), \\ f(x^*)=0}} \int_\Omega f \diff(D_{a,r}(b)-D_{a,r}(c))\\
        &=c'_p R^{p-1} \sup_{\substack{f\in \Lip_1(\Omega), \\ f(x^*)=0}} \sum_{i=1}^{k(a,r)} (b_i-c_i) f(x_i^{(a,r)})\\
        &= c'_p R^{p-1} \sup_{\substack{f\in \Lip_1(\Omega), \\ f(x^*)=0}} \sum_{i=1}^{k(a,r)} \vert b_i-c_i\vert \cdot \vert f(x_i^{(a,r)})- f(x^*)\vert\\
        &\leq c'_p R^{p-1} \sum_{i=1}^{k(a,r)} \vert b_i-c_i\vert \cdot d(x_i^{(a,r)}, x^*)\\
        &\leq c'_p R^{p-1} r \vert b-c \vert_{1}
    \end{align*}
    leading to the desired Hölder continuity with exponent $1/p$.
\end{proof}

Furthermore, we need a characterization of compact sets in $(\PP_p(\Omega), W_p)$. To do so, we follow \cite[Proposition 7.1.5]{Ambrosio2005} and adjust the result to proper metric spaces.
\begin{proposition}\label{prop:wasserstein_compact}
    Let $(\Omega, d)$ be a proper metric space and $p\geq 1$. A set $K\subseteq \PP_p(\Omega;d)$ is relatively compact with respect to the $p$-Wasserstein distance $W_p$ if and only if it has uniformly integrable $p$-moments, i.e.\ for any $x_0\in \Omega$, it holds that
    \begin{equation}\label{eq:wasserstein_compact}
        \lim_{R\to\infty} \sup_{\mu\in K}\int_{d(x_0, x)\geq R} d(x_0, x)^p\diff\mu(x) = 0.
    \end{equation}
\end{proposition}
\begin{proof}
    The fact that every relatively compact set $K\subseteq \PP_p(\Omega;d)$ has uniformly integrable $p$-moments follows directly from the result in \cite[Proposition 7.1.5]{Ambrosio2005}.\\
    For the opposite direction, an additional tightness condition is needed \cite[Theorem 5.1.3]{Ambrosio2005}: A set $K\subseteq \PP(\Omega)$ is called \textit{tight}, if for all $\varepsilon>0$, there exists a compact set $\Omega_\varepsilon$ in $\Omega$ such that $\mu(\Omega \setminus \Omega_\varepsilon)\leq \varepsilon$ for all $\mu\in K.$ However, we will show in the following that for a given subset $K\subseteq \PP_p(\Omega;d)$, the condition in \Cref{eq:wasserstein_compact} together with the properness of the metric space $(\Omega, d)$ already yields tightness of $K.$\newline
    Consider any $K\subseteq \PP_p(\Omega;d)$ which has uniformly integrable $p$-moments. Let $\varepsilon>0$ and $x_0\in \Omega.$ Due to the condition in \eqref{eq:wasserstein_compact}, one can choose $R>0$ sufficiently large, so that
    \begin{equation*}
        \int_{d(x_0, x)\geq R} d(x_0, x)^p\diff\mu(x) < \varepsilon R^p \quad \text{for all}\ \mu \in K.
    \end{equation*}
    Hence, for all $\mu\in K$, we have the estimate
    \begin{align*}
        \mu(d(x_0, x)> R) \leq \int_{d(x_0, x)\geq R}1 \diff\mu \leq \int_{d(x_0, x)\geq R} \left( \frac{d(x_0, x)}{R} \right)^p \diff \mu(x) < \varepsilon.
    \end{align*}
    With $\mu(d(x_0, x)> R) = \mu(\Omega \setminus \overline{B}_R(x_0))$ and the fact that closed balls are compact due to the properness of $\Omega$, it follows that $K$ is indeed tight.
\end{proof}

\begin{theorem}\label{thm:Wasserstein_DEDAP}
    Let $(\Omega, d)$ be a proper metric space and $p \geq 1$. Then $\PP_p(\Omega;d)$ equipped with the metric $W_p$ has the input-EDAP.
\end{theorem}

\begin{proof}
    Choose sequences $(r_n)_{n\in \N}$ and $(a_n)_{n\in \N}$ of positive real numbers such that 
    \begin{align*}
        r_n \xrightarrow{n\to\infty} \infty \hspace{-7ex} &&\text{and}\hspace{-7ex} &&a_n r_n^{p-1} \xrightarrow{n\to\infty} 0.
    \end{align*}
    Furthermore, recall the encoders $E_n \coloneqq E_{a_n,r_n}$ and decoders $D_n\coloneqq D_{a_n,r_n}$ as well as the fixed $x^*\in \Omega$ from \Cref{def:wasserstein_encoder_decoder}.
     For any $\mu\in \PP_p(\Omega; d)$ we use the notation $\mu_n\coloneqq D_n\circ E_n (\mu)$ and
    \begin{align*}
        \varepsilon_n(\mu)\coloneqq \int_{d(x, x^*)> r_n}d(x, x^*)^p \diff\mu.
    \end{align*}
    Note that
    \begin{align*}
        \mu(d(x, x^*)> r_n) = \int_{d(x, x^*)> r_n}1 \diff\mu < \int_{d(x, x^*)> r_n}\left(\frac{d(x, x^*)}{r_n}\right)^p \diff\mu = \frac{\varepsilon_n(\mu)}{r_n^p}. 
    \end{align*}
    By choosing $R=3r_n$ in Corollary \ref{coro:bounded_wasserstein}, we obtain for any $\mu\in \PP_p(\Omega; d)$ that
    \begin{align*}
        W_p(\mu,\mu_n)^p &\leq c'_p \left( (3r_n)^{p-1} W_1^{(3r_n)}(\mu,\mu_n) + \int_{d(x, x^*)\geq 1.5r_n} d(x, x^*)^p \diff(\mu+\mu_n)(x) \right)\\
        &\leq c'_p \left( (3r_n)^{p-1}W_1^{(3r_n)}(\mu,\mu_n) + \varepsilon_n(\mu) \right),
    \end{align*}
    where the latter is due to the fact that, by definition, $\mu_n(d(x, x^*)>r_n) = 0.$
    For any $f:\Omega\to \R$ with $f(x^*)=0$ and which is $1$-Lipschitz with respect to the bounded metric $d^{(3r_n)}(x,y)=\min\{d(x,y), 3r_n\}$, we observe for every $\mu\in \PP_p(\Omega; d)$ that
    \begin{align*}
        \left \vert \int_{\Omega}f \diff(\mu - \mu_n) \right \vert &\leq \left \vert \int_{d(x, x^*)\leq r_n}f \diff(\mu - \mu_n) \right \vert + \left \vert \int_{d(x, x^*)> r_n}f \diff\mu \right \vert\\
        &\leq \scalebox{0.86}{$\displaystyle \left \vert \sum_{i=1}^{k(a_n,r_n)} \int_{I_i^{(a_n,r_n)}}f(x) - f(x_i^{(a_n,r_n)}) \diff\mu(x) \right \vert + \int_{d(x, x^*)> r_n}\vert f(x) - f(x^*)\vert \diff\mu(x)$}\\
        &\leq \scalebox{0.97}{$\displaystyle \sum_{i=1}^{k(a_n,r_n)} \int_{I_i^{(a_n,r_n)}} d(x, x_i^{(a_n,r_n)})\diff\mu(x) + \int_{d(x, x^*)> r_n}d^{(3r_n)}(x, x^*)\diff \mu(x)$}\\
        &\leq \max \{\textup{diam}(I_i^{(a_n,r_n)}):i=1,...,k(a_n,r_n)\} + 3r_n\mu(d(x, x^*)>r_n)\\
        &\leq a_n + 3\varepsilon_n(\mu) r_n^{1-p}.
    \end{align*}
    Hence, for any $\mu\in \PP_p(\Omega; d)$ it is
    \begin{align*}
        W_1^{(3r_n)}(\mu, \mu_n) \leq a_n + 3\varepsilon_n(\mu) r_n^{1-p},
    \end{align*}
    which implies that
    \begin{align*}
        W_p(\mu, \mu_n)^p \leq c'_p (3r_n)^{p-1} a_n + c'_p\varepsilon_n(\mu)(3^p+1).
    \end{align*}
    By definition of $a_n$, we have that $c'_p (3r_n)^{p-1} a_n$ converges to zero independently on $\mu$. For any compact $K\subseteq \PP_p(\Omega;d)$ with respect to the $W_p$ distance, we have that $\sup_{\mu\in K}\varepsilon_n(\mu) \xrightarrow{n\to\infty} 0$ due to $r_n\xrightarrow{n\to \infty}\infty$ and \cref{prop:wasserstein_compact}.
    Hence, for any compact $K\subseteq \PP_p(\Omega; d)$, we have that
    \begin{align*}
        \sup_{\mu\in K} W_p(\mu, D_n\circ E_n (\mu))= \sup_{\mu\in K} W_p(\mu, \mu_n) \xrightarrow{n\to\infty} 0,
    \end{align*}
    which shows property (iii) in Definition \ref{def:input_EDAP} of the input-EDAP. Property (ii) is covered by Corollary \ref{coro:decoder_HoelderContinuous_Wasserstein} and (i) is due to the fact, that
    \begin{equation*}
        \sup_{\mu\in \PP_p(\Omega; d)} \vert E_n(\mu)\vert_{1} \leq 1.
    \end{equation*}
    Hence, $\PP_p(\Omega; d)$ equipped with the metric $W_p$ has the input-EDAP.
\end{proof}

Below, we observe the input-EDAP for two common subsets of $\PP_p(\Omega;d)$. The first result is an immediate consequence of \Cref{lem:subspace_EDAP} and \Cref{thm:Wasserstein_DEDAP} as the decoders in \Cref{thm:Wasserstein_DEDAP} map to compactly supported measures.
\begin{corollary}\label{coro:wasserstein_comp_support}
    Let $(\Omega, d)$ be a proper metric space and let $p\geq 1$. The set of measures in $\PP(\Omega)$ with compact support has the input-EDAP, when being equipped with any $p$-Wasserstein distance.
\end{corollary}

Note that in the next result, one could also consider absolutely continuous measures in $\PP_p(\Omega;\vert \cdot \vert)$ for certain closed subsets $\Omega\subset \R^m$, which we avoid for better readability.
\begin{corollary}\label{coro:wasserstein_abs_cont_EDAP} 
    Let $p\geq 1$. Consider the set $\PP^{ac}_p$ of measures $\mu\in \PP_p(\R^m;\vert \cdot \vert)$ that are absolutely continuous w.r.t. the $m$-dimensional Lebesgue measure $\lambda$. Then $(\PP^{ac}_p, W_p)$ has both the input- and output-EDAP.
\end{corollary}
\begin{proof}
    Recall the sequence of encoders $E_n$ and decoders $D_n$ from \Cref{thm:Wasserstein_DEDAP}. We cannot deduce the input-EDAP of $\PP^{ac}_p$ from \Cref{lem:subspace_EDAP} because $D_n$ does not map into $\PP^{ac}_p$. Nevertheless, we can construct different decoders. Consider a null sequence $\varepsilon_n>0$ which satisfies that 
    \begin{align*}
        \varepsilon_nr_n^{p-1} \xrightarrow{n\to\infty}  0 \textup{ \, \, \, and \, \,  \,} B_{\varepsilon_n}(x_i^{(a_n, r_n)}) \subset B_{1.1 r_n}(x^*) \textup{\, for all }i.
    \end{align*}
    For $y\in \R^m$, consider $\mu^{(n)}_y\in \PP^{ac}_p$ which has the density $g_y^{(n)}$ with respect to $\lambda$, where
    \begin{align*}
        g_y^{(n)}(z)\coloneqq \begin{cases}
            \lambda(B_{\varepsilon_n}(y))^{-1} \textup{\, if \,} z \in B_{\varepsilon_n}(y),\\
            0, \textup{\, else.}
        \end{cases}
    \end{align*}
    Recall $A_{a_n, r_n} = \{ c\in \R^{k(a_n, r_n)}: c_i\geq 0, \sum_i c_i \leq 1\}$, the domain of $D_n$, and define the new decoders via
    \begin{align*}
        \tilde{D}_n:A_{a_n, r_n} &\longrightarrow \PP^{ac}_p\\
        c &\longmapsto \sum_{i=1}^{k(a_n, r_n)} c_i \mu_{x_i^{(a_n, r_n)}}^{(n)} + \left(1 - \sum_{i=1}^{k(a_n, r_n)}c_i \right) \mu_{x^*}^{(n)}.
    \end{align*}
    Hölder continuity of $\tilde{D}_n$ can be derived similarly as for $D_n$ in \Cref{coro:decoder_HoelderContinuous_Wasserstein}.
    By applying \Cref{coro:bounded_wasserstein} with $R=3r_n$, and due to $B_{\varepsilon_n}(x_i^{(a_n, r_n)}) \subset B_{1.1 r_n}(x^*)$, it follows for every $c\in A_{a_n, r_n}$ that
    \begin{align*}
        W_p(D_n c, \tilde{D}_n c)^p \leq  c_p'(3 r_n)^{p-1} W_1(D_n c, \tilde{D}_n c) \leq c_p'(3 r_n)^{p-1} \varepsilon_n,
    \end{align*}
    where the last inequality is due to the fact that for every $1$-Lipschitz $f:\R^m\to \R$ and $y\in \R^m$, it is
    \begin{align*}
        \left \vert \int_{\R^m} f \diff (\delta_y-\mu^{(n)}_y)\right \vert &\leq \lambda(B_{\varepsilon_n}(y))^{-1} \int_{B_{\varepsilon_n}(y)}\vert f(y) - f(x) \vert\diff \lambda(x)\\
        &\leq \lambda(B_{\varepsilon_n}(y))^{-1} \int_{B_{\varepsilon_n}(y)}\vert y - x \vert\diff \lambda(x) \leq \varepsilon_n.
    \end{align*}
    Therefore, for every $\mu\in \PP^{ac}_p$, it holds that
    \begin{align*}
        W_p(\mu, \tilde{D}_n \circ E_n (\mu))
        &\leq W_p(\mu, D_n \circ E_n (\mu)) + (c_p'(3 r_n)^{p-1} \varepsilon_n)^{1/p}.
    \end{align*}
    Hence, it follows from \Cref{thm:Wasserstein_DEDAP} and $\varepsilon_n r_n^{p-1}\xrightarrow{n\to\infty}0$ that $\PP_p^{ac}$ has the input-EDAP.\newline
    For the output-EDAP, we need to define decoders on $\R^{k(a_n, r_n)}$, see \Cref{def:output_EDAP}. Consider the Euclidean projection $\pi_n:\R^{k(a_n, r_n)}\to A_{a_n, r_n}$ onto the convex, closed set $A_{a_n, r_n}\subset  \R^{k(a_n, r_n)}$. We claim that $\tilde{D}_n\circ \pi_n$ and $E_n:\PP_p^{ac}\to A_{a_n, r_n}$ are suitable decoders and encoders for the output-EDAP of $\PP_p^{ac}$. Hölder-continuity of $\tilde{D}_n\circ \pi_n$ is preserved due to the Lipschitz continuity of $\pi_n$, which follows from \cite[Proposition 1.1.9]{Bertsekas_convex}. Note that $\tilde{D}_n\circ \pi_n \circ E_n= \tilde{D}_n \circ E_n$. Therefore, it only remains to prove that the encoders $E_n$ restricted to $\PP_p^{ac}$ are continuous. This follows from the Portmonteau theorem \cite[Theorem 2.1]{Billingsley_Probability_Measures}, since the $I_i^{(a_n, r_n)}$ can be constructed such that their boundaries have zero Lebesgue measure.
\end{proof}

Below, we show that for any compact metric space $(\Omega, d)$, the $p$-Wasserstein spaces also have the output-EDAP. 
Note that for compact $\Omega$, all $\PP_p(\Omega; d)$ coincide and $\PP(\Omega)\coloneqq \PP_p(\Omega;d)$ is compact with respect to every $p$-Wasserstein metric, see \cite[Remark 6.19]{Villani_OldNew}.

\begin{theorem}\label{thm:wasserstein_compact_omega}
    Let $\Omega$ be a compact metric space. Then for every $p\geq 1$, the $p$-Wasserstein space $(\PP(\Omega), W_p)$ has both the input- and output-EDAP.
\end{theorem}
\begin{proof}
    It only remains to show the output-EDAP. Let $P_{\frac{1}{n}, i}$ be the partition of unity from \Cref{lem:partition_unity} corresponding to a $\frac{1}{n}$-covering $x^{(n)}_1,...., x_{k(n)}^{(n)}$ of $\Omega$. 
    For some fixed $x^*\in \Omega$, consider the mappings
    \begin{align*}
        E_n:\PP(\Omega) &\longrightarrow \R^{k(n)} 
        & \hspace{-0.7cm}
        D_n:&\,\{c\in \R^{k(n)}: c_i\geq 0, \sum_ic_i \leq 1\} \longrightarrow \PP(\Omega) \\
        E_n(\mu)_i &=\int_\Omega P_{\frac{1}{n}, i}\diff\mu,
        & \hspace{-0.7cm}
        &c \longmapsto \sum_{i=1}^{k(n)} c_i \delta_{x_i^{(n)}} + \left( 1 - \sum_{i=1}^{k(n)} c_i\right)\delta_{x^*}.
    \end{align*}
    We show below that the encoders $E_n$ and decoders $D_n\circ\pi_n$ are suitable choices in \Cref{def:output_EDAP} of the output-EDAP, where $\pi_n$ denotes the Euclidean projection onto the closed, convex set $\{c\in \R^{k(n)}: c_i\geq 0, \sum_ic_i \leq 1\}$.
    Since $P_{\frac{1}{n}, i}\in \CC(\Omega)$ it follows that $E_n$ is continuous w.r.t. to weak convergence of probability measures and hence w.r.t. to any $p$-Wasserstein distance, see for example \cite[Theorem 7.12]{Villani_Optimal_Transport}. This verifies (i) in \Cref{def:output_EDAP}. Hölder continuity of $D_n\circ \pi_n$ follows from Hölder and Lipschitz continuity of $D_n$ and $\pi_n$, respectively.
    Further, consider any $\mu\in \PP(\Omega)$ and any 1-Lipschitz $f:\Omega \to \R$. It holds that
\begin{align*}
    \left \vert \int_\Omega f \diff(\mu - D_n \circ E_n(\mu)) \right \vert &= \left \vert \int_\Omega f(x) - \sum_i P_{\frac{1}{n},i}(x) f(x_i^{(n)})d\mu(x) \right \vert\\
    &\leq \sup_{x\in \Omega} \left \vert f(x) - \sum_i P_{\frac{1}{n},i}(x) f(x_i^{(n)}) \right \vert.
\end{align*}
For fixed $x'\in \Omega$, the set of $1$-Lipschitz functions $f$ with $f(x') = 0$ is relatively compact in $\CC(\Omega)$ due to the Arzela-Ascoli theorem. Therefore, it is due to \Cref{thm:sampling_BAP} that
\begin{align*}
    \sup_{\mu\in \PP(\Omega)} W_1(\mu, D_n\circ E_n(\mu)) = \sup_{\substack{f\in Lip_1(\Omega), \\ f(x')=0}} \sup_{x\in \Omega}\vert f(x) - \sum_i P_{\frac{1}{n},i}(x) f(x_i^{(n)}) \vert \xrightarrow{n\to \infty}0.
\end{align*}
Hence, $D_n\circ E_n = D_n \circ \pi_n\circ E_n$ converges uniformly to the identity on $\PP(\Omega)$ w.r.t. to the $1$-Wasserstein distance $W_1$. According to \Cref{lem:wasserstein_equivalence}, uniform convergence is w.r.t. to any $W_p$. Therefore, $(\PP(\Omega), W_p)$ has the output-EDAP. 
\end{proof}

\subsubsection{Skorohod spaces $\DD([0, R])$}\label{sec:skorohod}
Within this section, we consider Skorohod spaces, which are the set of c\`adl\`ag functions equipped with the so-called Skorohod topology. We show in \Cref{thm:skorohod_EDAP_R} that c\`adl\`ag functions on $[0,R]$, equipped with any metric inducing the Skorohod topology, satisfy the input-EDAP. In \Cref{sec:skorohod_infty}, we verify this for c\`adl\`ag functions on $[0, \infty)$. It remains an open question whether these spaces also fulfill the output-EDAP. For definitions and basic properties, we mainly follow \cite[Chapters 12 and 16]{Billingsley_Probability_Measures}.
\begin{definition}[C\`adl\`ag-Functions]
    Let $I\subseteq \R$ be a closed interval. A function $f:I\to \R$ is called a c\`adl\`ag function if the following holds:
    \begin{itemize}
        \item[(i)] $f$ is right-continuous;
        \item[(ii)] All left-limits exist, i.e., $\lim_{s\nearrow t} f(s)\in \R.$
    \end{itemize}
    The set of all c\`adl\`ag functions $I\to \R$ is denoted by $\DD(I)$.
\end{definition}
For simplicity, we restrict ourself to real-valued c\`adl\`ag functions. However, their definition as well as the considerations below might also be extended to c\`adl\`ag functions with values in other normed (or metric) spaces than $\R$. 
Note that for $R\in (0,\infty)$, $\DD([0,R])$ can be equipped with the supremum norm, but becomes a non-separable space. Instead, one can define the following metric $d_R$, which induces the so-called Skorohod topology, and turns $\DD([0,R])$ into a separable metric space \cite[Theorem 12.2]{Billingsley_Probability_Measures}. 
\begin{definition}
    Let $\Lambda_R$ be the set of all strictly increasing, continuous mappings $\lambda:[0,R]\to[0,R]$ with $\lambda(0)=0$ and $\lambda(R)=R$. For $f,g\in \DD([0,R])$ define
    \begin{align*}
        d_R(f,g)\coloneqq \inf_{\lambda\in \Lambda_R} \max \Big\{\sup_{t\in [0,R]}\vert t- \lambda(t)\vert, \sup_{t\in [0,R]}\vert f(t)- g(\lambda(t))\vert \Big\}.
    \end{align*}
\end{definition}
Of course, $\big(\DD([0,R]), d_R\big)$ is an $\R$-vector space with pointwise addition of c\`adl\`ag functions. However, it is not a topological vector space, as pointwise addition of c\`adl\`ag functions is not a topological group \cite[Problem 12.2]{Billingsley_Probability_Measures}.
Equipped with $d_R$, the space $\DD([0, R])$ is not complete, but there is a topologically equivalent metric $d^\circ_R$, which turns it into a complete separable space \cite[Theorem 12.2]{Billingsley_Probability_Measures}. Due to Lemma \ref{lem:EDAP_equivalent_metrics}, the input-EDAP is preserved for topologically equivalent metrics. Thus, when showing below that $(\DD([0, R]), d_R)$ has the input-EDAP, this automatically implies the input-EDAP of $(\DD([0, R]), d_R^\circ)$, where the same encoders and decoders can be chosen. 
The consideration of the following mapping $A_\sigma^{(R)}$ will be helpful, as it has a canonical decomposition into an encoder and decoder. 
\begin{definition}\label{def:A_sigma_R}
    Consider $0=s_0<s_1<...<s_k=R$, write $\sigma=\{s_0,...,s_k\}$ and define
    \begin{align*}
        A_\sigma^{(R)}:\DD([0,R]) &\longrightarrow \DD([0,R])\\
        A_\sigma^{(R)} (f)(t) &\coloneqq \begin{cases}
            f(s_{i-1}) \textup{ \, if \, } t\in [s_{i-1}, s_i) \textup{ \, for \, } i=1,...,k,\\
            f(s_k) \textup{ \, if \, } t= s_k.
        \end{cases}
    \end{align*}
    We can write $A^{(R)}_{\sigma} = D_{\sigma}\circ E_{\sigma}$, where
    \begin{align*}
    E_{\sigma}(f) &\coloneqq \Big(f(s_0),f(s_1), ..., f(s_k)\Big)^T,\\
    D_{\sigma}(a)(t) &\coloneqq \begin{cases}
        a_{i-1} \textup{ \, if \, } t\in [s_{i-1}, s_i) \textup{ \, for \, } i=1,...,k,\\
        a_k \textup{ \, if \, } t= s_k.
    \end{cases}  
\end{align*}
For readability, we omit the dependence of $E_\sigma, D_\sigma$ on $R.$ We call $E_\sigma$ and $D_\sigma$ a Skorohod encoder and decoder, respectively.
\end{definition}

Sampling, and hence $E_\sigma$ and $A_\sigma^{(R)}$, is not continuous w.r.t. the metric $d_R$ \cite[p. 134]{Billingsley_Probability_Measures}. However, we show in what follows that for suitable $\sigma$, the encoder and decoder fulfill the requirements in Definition \ref{def:input_EDAP} of the input-EDAP. Although it would be sufficient to prove continuity of the decoders, we show that they are even Lipschitz continuous.
\begin{lemma}
    $D_\sigma$ is Lipschitz continuous for any $\sigma=\{s_0,s_1,...,s_k\}\subset[0,R]$.
\end{lemma}
\begin{proof}
    For $a,b\in \R^{k+1}$, it holds that 
    \begin{align*}
        d_R(D_\sigma(a), D_\sigma(b)) &= \inf_{\lambda\in \Lambda_R} \max \Big\{\sup_{t\in [0,R]}\vert t- \lambda(t)\vert, \sup_{t\in [0,R]}\vert D_\sigma(a)(t)- D_\sigma(b)(\lambda(t))\vert \Big\} \\
        &\leq \sup_{t\in [0,R]}\vert D_\sigma(a)(t)- D_\sigma(b)(t)\vert = \vert a-b\vert_{\infty}.
    \end{align*}
\end{proof}

We need the following result, which can be found in \cite[Chapter 12, Lemma 3]{Billingsley_Probability_Measures}.
\begin{lemma}\label{lem:compact_convergence_Asigma}
    Consider $0=s_0< s_1< ...<s_k=R$ and write $\sigma=\{s_0, ..., s_k\}$. If $\vert s_i - s_{i-1}\vert < \delta$ holds for all $i=1,...,k$, it follows that 
    \begin{align*}
        d_R(A_\sigma^{(R)} f, f) \leq \max \{ \delta, \overline{w}_R(f, \delta)\},    
    \end{align*}
    where $\overline{w}_R(f, \delta)$ is the modulus of continuity of the c\`adl\`ag function $f$, see \cite[P. 171]{Billingsley_Probability_Measures}, or \cite[Eq. (12.6)]{Billingsley_Probability_Measures} for $R=1$.
\end{lemma}
According to \cite[Theorem 12.3]{Billingsley_Probability_Measures}, relatively compact subsets of $(\DD([0, R]), d_R)$ can be characterized as follows:
\begin{theorem}\label{thm:precompact_sets_skorohov}
    A set $K\subset \DD([0, R])$ is relatively compact w.r.t. $d_R$ if and only if
    \begin{align*}
        \sup_{f\in K} \sup_{t\in [0,R]} \vert f(t)\vert < \infty &&\text{and} && \sup_{f\in K} \overline{w}_R(f, \delta)\xrightarrow{\delta \to 0} 0.
    \end{align*}
\end{theorem}

An immediate consequence of \Cref{lem:compact_convergence_Asigma} and \Cref{thm:precompact_sets_skorohov} is that if $\sigma$ is fine enough within $[0,R]$, it follows that $A_\sigma^{(R)}=D_\sigma \circ E_\sigma$ uniformly approximates the identity mapping on compact sets.  \begin{corollary}\label{coro:compact_convergence_Asigma}
    For $n\in \N$, consider $0=s^{(n)}_0< s_1^{(n)}<...<s_{k(n)}^{(n)}=R$ and write $\sigma_n=\{s^{(n)}_0, s_1^{(n)},...,s_{k(n)}^{(n)}\}$. Assume that 
    \begin{align*}
        \max_{i=1,\dots,k(n)} \vert s_i^{(n)}- s_{i-1}^{(n)}\vert \xrightarrow{n \to \infty} 0.
    \end{align*}
    Then for every compact $K\subset \mathcal{D}([0,R])$, it is
    \begin{align*}
        \sup_{f\in K} d_R\left(A_{\sigma_n}^{(R)}(f), f\right) \xrightarrow{n \to \infty} 0.
    \end{align*}
\end{corollary}
\begin{theorem}\label{thm:skorohod_EDAP_R}
    The space $\mathcal{D}([0,R])$ equipped with the metric $d_R$ has the input-EDAP.
\end{theorem}
\begin{proof}
    Consider suitable sets $\sigma_n$ as in Corollary \ref{coro:compact_convergence_Asigma}. It only remains to verify property (i) in Definition \ref{def:input_EDAP} of the encoders $E_{\sigma_n}$. Choose $r:\N \to (0,\infty)$ such that for all $c\in \R^{k(n)}$ it is  
    \begin{align*}
        \vert c \vert \leq \frac{r(n)}{n}\vert c\vert_{\infty}.
    \end{align*}
    Let $K\subset \DD([0,R])$ be compact. Due to Theorem \ref{thm:precompact_sets_skorohov}, there is $N_K\in \N$ such that
    \begin{align*}
        \sup_{f\in K} \sup_{t\in [0,R]} \vert f(t)\vert \leq N_K. 
    \end{align*}
    Hence, for all $n\geq N_K$ it follows that
    \begin{align*}
        \sup_{f\in K} \vert E_{\sigma_n}(f) \vert \leq  \frac{r(n)}{n} \sup_{f\in K} \vert E_{\sigma_n}(f) \vert_{\infty} \leq \frac{r(n)}{n} \sup_{f\in K} \sup_{t\in [0,R]} \vert f(t)\vert \leq r(n).
    \end{align*}    
\end{proof}

\section{New universal operator approximation theorem}\label{sec:approx_theorem}
We recall that one of the main goals of this work is to approximate any continuous operator $G:\X\to\Y$ between suitable separable metric spaces $\X$ and $\Y$ by a sequence of encoder-decoder architectures $G_n=D_n^\Y\circ \varphi_n \circ E_n^\X:\X\to \Y$ that converges uniformly to $G$ on every compact set $K\subseteq \X$, see statement \labelcref{fig:approximation_types:(B)} of \Cref{fig:approximation_types}. In this section, we prove in \Cref{thm:sequential_density_normed} that this is possible if encoders and decoders from the input- and output-EDAP are used for the construction of $G_n$, see \Cref{def:input_EDAP,def:output_EDAP}.
If only statement \labelcref{fig:approximation_types:(A)} is desired, so when $G_n$ is allowed to depend on $K$, we further present a corresponding universal approximation theorem at the end of this section. For this weaker type of approximation, the definitions of the EDAPs can be relaxed, which would then also allow for the consideration of non-separable spaces.

Due to the encoding and decoding, the operator approximation task can be reduced to a function approximation task, that is, approximating functions between spaces $\K^a$ by $\varphi_n$ used in the encoder-decoder architecture $G_n=D_n^\Y\circ \varphi_n \circ E_n^\X$. Therefore, the functions $\varphi_n$ must be chosen from a sufficiently rich set of functions, which leads to the term of universal function approximators. A similar terminology and definition has been considered, for example, within the work of Kratsios et al. \cite[Definition 2.11]{kratsios2023}.

\begin{definition}[Universal function approximator]\label{def:function_approximator}
    A set $\FF$ of functions $\K^a\to \K^b$ is called a \emph{universal function approximator} if for any $a,b\in \N$, compact set $K\subset \K^a$ and continuous function $f:K\to \K^b$  there exists a sequence $(f_n)_{n\in \N}$ in $\FF$ such that
    \begin{align*}
        \sup_{x\in K} \big \vert f(x) - f_n(x) \big \vert \xrightarrow[]{n\to\infty} 0.
    \end{align*}
    Note that $\FF$ may contain discontinuous functions.
\end{definition} 

There are different possible choices for universal function approximators. For example, one can consider different types of neural network architectures such as fully connected or convolutional neural networks \cite{Cybenko1989, Pinkus_1999, Zhou2020}. 
Moreover, polynomials build another type of universal function approximators according to the Stone-Weierstraß theorem \cite[Section XIII.3]{Dugundji_Topologie}.

\begin{remark}
    Let $\mathcal{F}$ be a universal function approximator, which consists of continuous functions. Since $\K^a$ is hemicompact, it is due to \Cref{thm:dense_Frechet_hemicompact} that $\FF\cap \CC(\K^a, \K^b)$ is sequentially dense in $\CC(\K^a, \K^b)$. Thus, for every $f\in \CC(\K^a, \K^b)$ there is a sequence $(f_n)_{n\in \N}$ in $\mathcal{F}$ that converges uniformly to $f$ on every compact $K\subset \K^a$. This means that the approximating sequence $(f_n)_{n\in \N}$ can be chosen independently on the compact set $K$ in \Cref{def:function_approximator}.
\end{remark}

The following basic lemma will be useful for showing our operator approximation theorem. Similar versions have been derived and used in several other studies for proving universal operator approximation results, see for example \cite{Chen1995_approx_thm, Lu2021_DeepONet, Kovachki2021_approxFNO, Schwab2023, Hua2023}.
We state a general version below which is similar to \cite[Lemma 21]{Kovachki2023}, in which it was shown for Banach spaces $\X$ and $\Y$ and some continuous sequence of operators $G_n$, using a different proof. 

\begin{lemma}\label{lem:union_compact_sets}
    Let $\X$ and $(\Y, d_\Y)$ be metric spaces and $G\in \CC(\X, \Y)$. Consider any sequence $G_n:\X\to \Y$ that converges uniformly on every compact set to $G$ and assume that all $G_n$ map compact sets to relatively compact sets. Then for any compact $K\subseteq \X$ the following set is relatively compact:
    \begin{align*}
        V\coloneqq G(K)\cup \left(\bigcup_{n=1}^\infty G_n(K) \right).
    \end{align*}
    Note that $G_n$ is allowed to be discontinuous.
    \begin{proof}
    Consider a sequence $(g_n)_{n\in \N}$ in $V$. We show that it has a subsequence converging to some $g\in \Y$. \textbf{Case 1:} If infinitely many $g_n$ belong to 
    \begin{align*}
         V_N \coloneqq G(K) \cup \left(\bigcup_{n=1}^N G_n(K) \right)
    \end{align*}
    for some $N\in \N$, the existence of such a subsequence follows from the relative compactness of $V_N$. 
    \textbf{Case 2:} Otherwise, assume without loss of generality that $g_n = G_{m(n)}(f_n)$, where $m(n)\in \N$ goes to infinity and $f_n\in K$. By compactness of $K$, there exists a subsequence $(f_{n_k})_{k\in \N}$ converging to some $f\in K$. Therefore, by continuity of $G$ and uniform convergence of $G_n$ to $G$ on $K$, we conclude that 
    \begin{align*}
        d_\Y\big( g_{n_k}  \, ,\, G(f)\big) 
        &\leq  d_\Y\big( G_{m(n_k)}(f_{n_k}) \, ,\, G(f_{n_k})\big) + d_\Y\big( G(f_{n_k})  \, ,\, G(f) \big)  \\
        &\leq \sup_{z\in K} d_\Y\big( G_{m(n_k)}(z) \, ,\, G(z) \big) + d_\Y\big( G(f_{n_k}) \, ,\, G(f) \big) \xrightarrow{k\to\infty} 0.  
    \end{align*}
\end{proof}
\end{lemma}

Some consequence of the previous lemma is the following result, which shows that uniform convergence on every compact set for sequences of operators is preserved under concatenation.

\begin{corollary}\label{coro:sandwich}
    Let $\X,(\Y, d_\Y),(\Z, d_\Z)$ be metric spaces, $G\in \CC(\X, \Y)$ and $H\in \CC(\Y, \Z)$. In addition, let $G_n:\X\to \Y$ and $H_n:\Y\to \Z$ be sequences converging uniformly on all compact sets to $G$ and $H$, respectively. 
    Assume that for each compact $K\subseteq\X$, the sets $G_n(K)$ are relatively compact for $n\geq N_K\in \N$. Then for every compact $K\subseteq \X$ it holds that
    \begin{align*}
        \sup_{f\in K} d_\Z \big( (H\circ G)(f) \, ,\, (H_n\circ G_n) (f) \big) \xrightarrow{n\to\infty} 0.
    \end{align*}
    Note that $G_n, H_n$ are allowed to be discontinuous.
\end{corollary}
\begin{proof}
    Let $K\subseteq \X$ be compact and $\varepsilon>0$.Without loss of generality, assume that $G_n(K)$ is relatively compact for every $n\in \N$. Otherwise, consider the sequence $G_n$ starting at some index $n=N_K\in \N$.
    According to Lemma \ref{lem:union_compact_sets} the set 
    \begin{align*}
        V \coloneqq G(K) \cup \left( \bigcup_{n=1}^\infty G_n(K)\right)
    \end{align*}
    is relatively compact. Therefore, $H$ is uniformly continuous on $\overline{V}$, which means that there is some $\delta>0$ such that $d_\Z\big( H (f) , H (g)\big) < \varepsilon / 2$ whenever $f,g\in V$ with $d_\Y\big( f , g \big) < \delta$. Due to uniform convergence of $G_n$ to $G$ on $K$, there exists $N_G\in \N$ such that for all $n\geq N_G$ it holds that
    \begin{align*}
        \sup_{f\in K} d_\Y\big( G_n (f) \, ,\, G(f) \big) < \delta.
    \end{align*}
    Therefore, we conclude for those $n\geq N_G$ that
    \begin{align*}
        \sup_{f\in K} d_\Z\big( (H\circ G) (f) \, ,\, (H \circ G_n)(f) \big) < \frac{\varepsilon}{2}.
    \end{align*}
    Since $H_n$ converges uniformly to $H$ on the compact set $\overline{V}$, there exists some $N_H\in \N$ such that for all $n\geq N_H$ it is
    \begin{align*}
        \sup_{f\in K} d_\Z\big( (H \circ G_n) (f) \, ,\, (H_n \circ G_n)(f) \big) \leq \sup_{g\in \overline{V}} d_\Z\big(H (g)\, ,\, H_n (g) \big) < \frac{\varepsilon}{2}.
    \end{align*}
    It follows from the triangular inequality and by choosing $n\geq \max\{N_G, N_H\}$ that
    \begin{align*}
        \sup_{f\in K} d_\Z\big( (H \circ G) (f) \, ,\, (H_n\circ G_n)(f) \big) < \varepsilon,
    \end{align*}
    which shows the claim.
\end{proof}

With the preceding considerations, we are now prepared to prove our main result: a universal operator approximation theorem for operators between spaces having the EDAP. It shows that operator approximation is possible by a sequence of encoder-decoder architectures that converge uniformly to the given operator on every compact set, which is statement \labelcref{fig:approximation_types:(B)} in \Cref{fig:approximation_types}. As discussed in \Cref{sec:approximation_properties}, there are diverse choices for suitable encoders and decoders. As we will point out in Section \ref{sec:examples_architectures}, the theorem is hence applicable, in particular, to many famous architectures used in the field of operator learning, but it is not restricted to neural networks.
\begin{theorem}\label{thm:sequential_density_normed}
    Let $(\X, d_\X)$ and $(\Y,  d_\Y)$ be metric spaces having the input-EDAP and output-EDAP, respectively. Denote suitable encoders and decoders by $D_n^\X,  E_n^\X$ and $D_n^\Y, E_n^\Y$. 
    Let $\FF$ be a set of universal function approximators. Then for every $G\in \CC(\X, \Y)$ there exists a sequence $\varphi_n\in \FF$ such that for every compact $K\subseteq \X$ it holds that
    \begin{align*}
        \sup_{f\in K} d_\Y\Big( G(f) \,,\, \left(D_n^\Y \circ \varphi_n \circ E_n^\X \right)(f) \Big) \xrightarrow{n\to\infty}0.
    \end{align*}
\end{theorem}
\begin{proof}
    Consider some $G\in \CC(\X, \Y)$. 
    As all decoders $D_n^\Y$ are uniformly continuous, by \cref{def:output_EDAP} of the output-EDAP, there exist $\delta_n>0$ such that 
    \begin{align}\label{thm:sequential_density_normed_eq00}
        d_\Y\left(D_n^\Y(a), D_n^\Y(b) \right) \leq \frac{1}{n} \textup{ \, for all \, } a,b\in \K^{w_\Y(n)} \textup{ \, with \, }\vert a-b\vert \leq \delta_n.
    \end{align}
    Choose for $\X$ the function $r:\N\to (0,\infty)$ as in \Cref{def:input_EDAP}. For every $n\in \N$ define the compact sets 
    \begin{align*}
        B_n \coloneqq \overline{E_n^\X(\X)}\cap \overline{B}_{r(n)}(0) \subset \K^{w_\X(n)},
    \end{align*}
    where $\overline{B}_{r(n)}(0)$ is the closed ball around $0\in \K^{w_\X(n)}$ of radius $r(n)$. Since $B_n$ is closed and bounded, it is compact in the finite dimensional space $\K^{w_\X(n)}$. Therefore, since $D_n^\X$ and $E_n^\Y$ are continuous and by \Cref{def:function_approximator} of universal function approximators, for all $n\in \N$ there exists $\varphi_n: \K^{w_\X(n)} \to \K^{w_\Y(n)}$ in $\FF$ that satisfies
    \begin{align}\label{thm:sequential_density_normed_eq11}
        \sup_{x\in B_n} \Big \vert \left(E_n^\Y \circ G \circ D_n^\X\right)(x) - \varphi_n(x) \Big \vert \leq \delta_n,
    \end{align}
    where $\delta_n$ is chosen as in \eqref{thm:sequential_density_normed_eq00}.
    Let $K\subseteq \X$ be compact. Then there exists some $N_K\in \N$ such that for all $n\geq N_K$ it is
    \begin{align}\label{thm:sequential_density_normed_eq22}
        \sup_{f\in K}\left \vert E_n^\X (f) \right \vert \leq r(n).
    \end{align}
    The preceding observations allow us to make for $n\geq N_K$ the estimates
    \begin{align*}
        &\sup_{f\in K} d_\Y \Big( \left(D_n^\Y \circ E_n^\Y \circ G \circ D_n^\X \circ E_n^\X \right) (f)\, , \, \left(D_n^\Y \circ \varphi_n \circ E_n^\X \right)(f) \Big)  \\
        \overset{\eqref{thm:sequential_density_normed_eq22}}{\leq}&\sup_{x\in B_n} d_\Y \Big( \left(D_n^\Y \circ E_n^\Y \circ G \circ D_n^\X \right) (x)\, , \, \left(D_n^\Y \circ \varphi_n \right)(x) \Big)
        \overset{\eqref{thm:sequential_density_normed_eq11}, \eqref{thm:sequential_density_normed_eq00}}{\leq} \frac{1}{n} \xrightarrow{n\to\infty}0.
    \end{align*}
    Note that the inequality in \eqref{thm:sequential_density_normed_eq22} implies that $D_n^\X\circ E_n^\X(K)$ is relatively compact for all $n\geq N_K$, as $D_n^\X$ is continuous on $\overline{E_n^\X(\X)}$. Hence, the claim follows from the definitions of the input- and output-EDAP, multiple applications of \Cref{coro:sandwich} and the observation
    \begin{align*}
        &\phantom{+}\sup_{f\in K} d_\Y\Big( G(f) \, , \, \left(D_n^\Y \circ \varphi_n \circ E_n^\X \right)(f) \Big) \\
        \leq &\phantom{+}\sup_{f\in K} d_\Y\Big( G(f) \, , \, \left(D_n^\Y \circ E_n^\Y \circ G \circ D_n^\X \circ E_n^\X \right) (f) \Big) \\&+\sup_{f\in K} d_\Y\Big( \left(D_n^\Y \circ E_n^\Y \circ G \circ D_n^\X \circ E_n^\X \right) (f) \,,\, \left(D_n^\Y \circ \varphi_n \circ E_n^\X \right)(f) \Big).
    \end{align*}
\end{proof}

\begin{remark}\label{rem:discontinuity_chen}
    Having a closer look at the proof of \Cref{thm:sequential_density_normed}, one may notice that the  \Cref{def:output_EDAP} of the output-EDAP admits the following relaxation, which we omitted for better readability: Let $(\Z, d_\Z)$ be a metric space and $\Y\subseteq \Z$. In the definition of the output-EDAP of $(\Y, d_\Z)$, one may allow the decoder $D_n^\Y:\K^{\omega_\Y(n)}\to \Z$ to take values in the larger space $\Z$. This permits extrapolation, while \Cref{thm:sequential_density_normed} remains valid. 
    A similar relaxation is possible in \Cref{thm:density_metric} regarding the output-CEDAP.
\end{remark}

In contrast to most results within the operator learning literature, the choice of the approximating sequence $G_n$ in \Cref{thm:sequential_density_normed} is independent of the compact set $K\subseteq \X$. As pointed out in \Cref{thm:dense_Frechet_hemicompact}, the previous result is hence a stronger result for most spaces $\X$ relevant for operator approximation, for example, whenever $\X$ is an infinite-dimensional normed space.
Nevertheless, in \cite{Schwab2023}, an analogous result has been achieved for separable Hilbert spaces $\X, \Y$ and encoder-decoder architectures based on Riesz-bases, see also \Cref{sec:examples_frame} for more information. To the best of our knowledge, \cite{Schwab2023} is the only other study showing uniform convergence of encoder-decoder architectures $G_n$ towards $G$ on every compact set $K$. For other operator learning architectures, such as FNOs or WNOs, we have not found any result similar to the theorem above. However, we highly suspect that also for those approaches, it would be possible to derive such a result, which we leave as an interesting future task.

Furthermore, to the best of our knowledge, \Cref{thm:sequential_density_normed} is the first universal approximation theorem for encoder-decoder architectures that shows statement \labelcref{fig:approximation_types:(B)} for continuous operators between metric spaces (satisfying suitable encoder-decoder approximation properties). In particular, the underlying spaces need not be topological vector spaces, and the input encoders need not be continuous. This allows the consideration of $p$-Wasserstein spaces as input and output spaces; see \cref{sec:wasserstein}. Moreover, Skorohod spaces of c\`adl\`ag functions can be treated as input spaces; see \cref{sec:skorohod}.

The operators $G$ in \Cref{thm:sequential_density_normed} were required to be defined on the whole space $\X$. However, what if $G$ is only defined on a subset $A\subset \X$? 
In \Cref{thm:sequential_density_normed_eq11} in the proof of \Cref{thm:sequential_density_normed}, it was crucial that the decoder $D_n^\X$ maps into the domain of the operator $G$. 
Therefore, one cannot use, in general, the same encoders and decoders for $A$ to obtain a similar result for operators $G\in \CC(A, \Y)$. Further, it is not automatically clear whether the subset $A$ itself has the input-EDAP. We suspect that this may not be true in general, because of the following analogous fact about the $\lambda$-BAP: There exists a separable Banach space $\X$ having the $\lambda$-BAP, but there is a subspace $A\subset \X$, which has not the $\lambda'$-BAP for any $\lambda'>0$, see for example \cite[Corollary 1.13]{Figiel2011}.
Of course, if also $A$ has the input-EDAP, then \Cref{thm:sequential_density_normed} is applicable with suitable encoders and decoders on $A$.
Nevertheless, for closed $A\subset \X$, one can indeed use the encoders from the input-EDAP of $\X$ to approximate operators in $\CC(A, \Y)$, see \Cref{coro:main_dugundji_extension} below. This is due to Dugundji's extension theorem \cite[Theorem 4.1]{Dugundji_Extension}.

\begin{theorem}[Dugundji's extension theorem]
    Let $\X$ be a metric space and $\Y$ a locally convex vector space. Further, let $A\subset \X$ be closed and $f:A \to \Y$ be continuous. Then there exists a continuous extension $F:\X\to \Y$ of $f$, which means that $F(x) = f(x)$ for every $x\in A$.
\end{theorem}

\begin{corollary}\label{coro:main_dugundji_extension}
    Consider the setup as in \Cref{thm:sequential_density_normed}, but in addition $\Y$ being a locally convex metric vector space. Let $A\subset \X$ be closed. For every $G\in \CC(A, \Y)$ there exists a sequence $\varphi_n\in \FF$ such that for every compact $K\subseteq \X$ it holds that
    \begin{align*}
        \sup_{f\in K\cap A} d_\Y \Big( G(f) \, ,\, \left(D_n^\Y \circ \varphi_n \circ E_n^\X \right)(f) \Big) \xrightarrow{n\to\infty}0.
    \end{align*}
\end{corollary}
\begin{proof}
    According to Dugundji's extension theorem, the operator $G$ can be extended to a continuous operator $\X\to\Y$. Therefore, the claim follows from applying \Cref{thm:sequential_density_normed} to the extension of $G$.
\end{proof}

\subsection{Approximation depending on compact sets}\label{sec:CEDAP}
In \Cref{thm:sequential_density_normed}, the input- and output-EDAP of the input and output space has been used for constructing a sequence of encoder-decoder architectures that converges to a given operator uniformly on every compact set (statement \labelcref{fig:approximation_types:(B)} in \Cref{fig:approximation_types}). If instead the weaker statement \labelcref{fig:approximation_types:(A)} is desired, one can derive an analogous approximation theorem, see \Cref{thm:density_metric} below. For this, it is sufficient to consider the following weaker versions of the EDAPs (recall \Cref{def:input_EDAP,def:output_EDAP} of the EDAPs). 

\begin{definition}[Input-CEDAP]\label{def:input_CEDAP}
    A metric space $(\X, d)$ has the compact input-EDAP (input-CEDAP) if for every compact $K\subseteq \X$, there are sequences of mappings $E_{K, n}^\X:K\to \K^{w_\X(K, n)}$ and $D_{K, n}^\X:\overline{E_{K, n}^\X(K)}\to \X$ with the following properties:
    \begin{itemize}
        \item[(i)] $E_{K, n}^\X(K)$ is bounded for every $n\in \N$.
        \item[(ii)] $D_{K, n}^\X$ is continuous.
        \item[(iii)] The mappings $T_{K, n}^\X \coloneqq D_{K, n}^\X \circ E_{K, n}^\X$ satisfy 
        \begin{align*}
            \sup_{f\in K} d\left(f\, ,\, T_{K, n}^\X (f)\right) \xrightarrow{n\to\infty}0.
        \end{align*}
    \end{itemize}
    Note that $E_{K, n}^\X$ is allowed to be discontinuous. In contrast to \cref{def:input_EDAP}, the encoder and decoder must only be defined on $K$ and $\overline{E_{K, n}^\X(K)}$, respectively.
\end{definition}
\begin{definition}[Output-CEDAP]\label{def:output_CEDAP}
     A metric space $(\X, d)$ has the compact output-EDAP (output-CEDAP) if for every compact $K\subseteq \X$, there are sequences of mappings $E_{K, n}^\X:\X\to \K^{w_\X(K, n)}$ and $D_{K, n}^\X:\K^{w_\X(K, n)}\to \X$ with the following properties:
    \begin{itemize}
        \item[(i)] $E_{K, n}^\X$ is continuous on $K$.
        \item[(ii)] $D_{K,n}^\X$ is uniformly continuous. 
        \item[(iii)] The mappings $T_{K, n}^\X \coloneqq D_{K, n}^\X \circ E_{K, n}^\X$ satisfy 
        \begin{align*}
            \sup_{f\in K} d\left(f\, ,\, T_{K, n}^\X (f)\right) \xrightarrow{n\to\infty}0.
        \end{align*}
    \end{itemize}
\end{definition}

Evidently, if a metric space $\X$ has the input- or output-EDAP, it has also the input- or output-CEDAP, respectively.  Further, note that the AP (\Cref{def:AP}) implies both the input- and output-CEDAP. Therefore, also non-separable spaces can have the CEDAPs, whereas this is impossible for the EDAPs. The following result generalizes \cite[Lemma 22]{Kovachki2023}, in which Banach spaces $\X$ and $\Y$ having both the AP were considered. The theorem below covers several existing universal approximation theorems of encoder-decoder networks, for example, for classical DeepONets \cite[Theorem 5]{Chen1995_approx_thm} and  \cite[Theorem 2]{Lu2021_DeepONet}, special cases of MIONets \cite[Theorem 3.1]{Jin2022_MIONET} and BasisONets \cite[Theorem 2.2]{Hua2023}, see \Cref{sec:examples_architectures} for more details on these implications.
As outlined in \Cref{sec:metric_EDAP_spaces}, the following theorem also handles operators between metric input and output spaces, allowing also discontinuous encoders. For example, it enables the consideration of $p$-Wasserstein and Skorohod spaces.
\begin{theorem}\label{thm:density_metric}
    Let $\X$ and $(\Y, d_\Y)$ be metric spaces having the input- and output-CEDAP with mappings $D_{K, n}^\X,  E_{K, n}^\X$ and $D_{K', n}^\Y, E_{K', n}^\Y$, respectively. 
    Let $\FF$ be a set of universal function approximators. Then for every $G\in \CC(\X, \Y)$ and every compact $K\subseteq \X$, there exists a sequence $\varphi_{K, n}\in \FF$ and compact $K'\subseteq \Y$ such that
    \begin{align*}
        \sup_{f\in K} d_\Y\Big(Gf \, , \, \left(D_{K', n}^\Y \circ \varphi_{K, n} \circ E_{K, n}^\X \right)f\Big) \xrightarrow{n\to\infty}0.
    \end{align*}
\end{theorem}
\begin{proof}
    The proof is similar to and simpler than the proof of \Cref{thm:sequential_density_normed} due to the higher flexibility regarding the dependency on the compact sets $K$. Note that $K'$ can be chosen as
    \begin{align*}
        K'\coloneqq \overline{\bigcup\limits_{n\in \N} G\circ D_{K, n}^\X \circ E_{K, n}^\X(K)}.
    \end{align*}
\end{proof}

\begin{remark}[Neural Filters]
    If $\X$ and $\Y$ are Banach spaces - or more generally Fr\'echet spaces, see \cref{rem:schauder_basis_frechet_EDAP} - both having Schauder bases, \cref{thm:sequential_density_normed} and \cref{thm:density_metric} also yield a universal approximation theorem for Neural Filters \cite{Galimberti_2026} as in statements \labelcref{fig:approximation_types:(B)} and \labelcref{fig:approximation_types:(A)}, respectively. One can just use the basis encoders and decoders from \cref{thm:schauder_basis_BAP}.
    Nevertheless, statement \labelcref{fig:approximation_types:(A)} can alternatively be obtained for Neural Filters as a combination of
    \cite[Theorem 1]{Galimberti_2026} with \cite[Lemma 2.6]{kratsios2023}.
\end{remark}

\begin{remark}\label{rem:comparison_galimberti}
In \cite[Theorem 3.7, Lemma 2.6]{kratsios2023} and \cite[Proposition 4.15]{galimberti2024} the authors proved statement \labelcref{fig:approximation_types:(A)}, but not \labelcref{fig:approximation_types:(B)}, for broad classes of encoder-decoder architectures and diverse metric/topological input and output spaces. Nevertheless, neither their results nor \cref{thm:density_metric} are more general than the other regarding possible choices of encoders and decoders.
For instance, \cref{thm:density_metric} allows the usage of discontinuous encoders $E_{K, n}^\X$, which is not directly possible in \cite{kratsios2023, galimberti2024}. As pointed out in Section \ref{sec:metric_EDAP_spaces}, there are several natural and constructive choices for discontinuous encoders.
On the other hand, discontinuous decoders $D_{K', n}^\Y$ can be treated in \cite{kratsios2023}, but not within our framework. In \cite{galimberti2024} suitable non-metrizable topological spaces (so-called quasi-Polish spaces) can be considered as input and output spaces.
As the main focus of our study is universal operator approximation as in statement \labelcref{fig:approximation_types:(B)}, a more detailed comparison would be out of scope and we let it as an interesting future task. A comparison should take into account at least the following aspects: 1) Which input and output spaces can or cannot be handled in the different frameworks? 2) Which encoders and decoders can or cannot be considered within the respective spaces? 3) How explicit are the constructions of encoders and decoders?
\end{remark}

\section{Applications}
In \Cref{thm:sequential_density_normed} and \Cref{thm:density_metric} we proved that continuous operators $G:\X\to \Y$ can be approximated by encoder-decoder architectures if the metric spaces $\X$ and $\Y$ have the input- and output-(C)EDAP, respectively. As shown in \Cref{sec:approximation_properties}, many spaces can be considered - even Wasserstein or Skorohod spaces - and there are diverse choices for suitable corresponding encoders and decoders.

In Section \ref{sec:examples_architectures} below, we point out that our developed theoretical framework unifies and extends approximation theory for several well-known encoder-decoder architectures used in operator learning, such as classical DeepONets \cite{Chen1995_approx_thm, Lu2021_DeepONet}, architectures based on frames or Riesz bases \cite{Schwab2023}, BasisONets \cite{Hua2023} or special cases of MIONets \cite{Jin2022_MIONET}.

In Section \ref{sec:applications_OT}, we illustrate some potential applications within optimal transport. For example, we propose a combination of \emph{optimal transport neural operators} \cite{Kovachki2025_OTNO} with an encoder-decoder architecture. Further, we sketch how \Cref{thm:sequential_density_normed} may help to verify approximation capabilities of \emph{geodesic operator networks} \cite{Gracyk_geonet}. 
    
\subsection{Well-known encoder-decoder architectures}\label{sec:examples_architectures}
In this section, we demonstrate that the approximation theory in \Cref{sec:approx_theorem} is modular in the sense that different encoder-decoder architectures, as well as different input and output spaces can be treated within a single unified framework.
In contrast, within the operator learning literature, individual approximation results are derived for specific architectures and therefore do not automatically extend to other architectures. 
By considering more general architectures in \Cref{thm:sequential_density_normed} and \Cref{thm:density_metric}, we obtain a more unified theoretical framework.
Moreover, most approximation theorems in the literature establish only statement \labelcref{fig:approximation_types:(A)}, whereas \Cref{thm:sequential_density_normed} proves the topologically stronger statement \labelcref{fig:approximation_types:(B)} for, in particular, all architectures below. 
Throughout this section, we consider $\R$-vector spaces, in accordance with the standard setting in the operator learning literature. The analogous results for $\C$-vector spaces can be obtained in the same way.
\subsubsection{Classical DeepONets}\label{sec:DeepONet_classic}
As a first example for approximating continuous operators $G:\X\to \Y$, we consider classical DeepONets, which have been presented in \cite{Lu2021_DeepONet, Lanthaler2022} and are inspired by the shallow architectures constructed within \cite{Chen1995_approx_thm}.
For the definition of a classical DeepONet, assume that $\X=\X(\Omega_\X, \R)$ and $\Y = \Y(\Omega_\Y, \R)$ are normed spaces of (for readability) $\R$-valued functions with domain spaces $\Omega_\X$ and $\Omega_\Y$, where $\Omega_\X$ is a compact metric space and $\Omega_\Y\subset \R^{d_\Y}$ is compact.
\begin{definition}[Classical DeepONets]\label{def:DeepONets} 
    Fix points $y_1,\dots,y_k$ in $\Omega_\X$. Let $\varphi:\R^k\to \R^p$ and $\psi:\Omega_\Y \to \R^p$ be neural networks. A classical DeepONet is a mapping of the form
    \begin{align*}
    \hat{G}: \X(\Omega_\X, \R) &\longrightarrow \Y(\Omega_\Y, \R)\\
        \big(\hat{G}(f) \big)(z) &= \sum_{i=1}^p \varphi_{i}\big(f(y_1),\dots,f(y_k)\big) \psi_i(z),
    \end{align*}
    where $\varphi_i,\psi_i$ denote the i-th output coordinates of $\varphi$ and $\psi$, respectively.
\end{definition}
In parts of the literature, the term \emph{DeepONet} is used more broadly to include architectures whose encoder is not necessarily based on point evaluations of the input function $f$. To distinguish these settings, we use the term \emph{classical DeepONet} for architectures that encode $f$ through finitely many samples. This terminology allows for a more precise discussion of approximation-theoretic results for different encoder classes.

To obtain universal approximation theorems for classical DeepONets as in statement \labelcref{fig:approximation_types:(B)}, one can simply interpret them as encoder-decoder architectures $\hat{G} = D\circ \varphi \circ E$ and apply \Cref{thm:sequential_density_normed} as follows:
\begin{itemize}
    \item Choose sampling encoders on $\X(\Omega_\X, \R)$ from \Cref{thm:sampling_BAP} or \Cref{coro:sampling_BAP}; \\
    (Also possible: Encoders from \Cref{app:coro_Ck} which use samples of $f'$, or sampling encoders from \Cref{thm:skorohod_EDAP_R} or \Cref{thm:skoro_infty_EDAP} if $\X$ is a Skorohod space.)
    \item Choose $\varphi$ from a universal function approximator $\FF$ which consists of neural networks, see \Cref{def:function_approximator}. Examples for $\FF$ can be found, e.g., in \cite{Cybenko1989, Pinkus_1999, Zhou2020};
    \item Choose dense decoders on $\Y$ from \Cref{thm:Tn_substitute_dense_set}, that correspond to a dense set $S\subset \Y(\Omega_\Y, \R)$ of neural networks, so all $\psi_i$ of the DeepONet should belong to $S$. Examples for $S$ can be found, e.g., in \cite{Pinkus_1999, DeVore2021relu_overview} if $\Y$ is a Lebesgue space, Sobolev space or space of continuously differentiable functions.
\end{itemize}

In particular, \Cref{thm:sequential_density_normed} implies and extends \cite[Theorem 2]{Lu2021_DeepONet}, in which the weaker statement \labelcref{fig:approximation_types:(A)} was shown for classical DeepONets and the special case of $\X=\CC(\Omega_\X, \R)$, $\Y=\CC(\Omega_\Y, \R)$, and $\Omega_\X$ being a compact subset of a Banach space.

The same setup for $\X$ and $\Y$ has been considered within \cite[Theorem 5]{Chen1995_approx_thm} to derive statement \labelcref{fig:approximation_types:(A)} for shallow classical DeepONets. Noteworthy, the constructed decoders for $\Y=\CC(\Omega_\Y, \R)$ can be expressed as
\begin{align*}
    (D c)(x) = \sum_i c_i g(w_i \cdot x + b_i),
\end{align*}
where $g:\R \to \R$ is a suitable activation function and the parameters $w_i\in \R^{d_\Y}, b_i\in \R$ depend on the compact set $K\subset \Y$ on which approximation is desired. Note that these decoders are not dense decoders as in \Cref{thm:Tn_substitute_dense_set} since the set $\{ x\mapsto g(w\cdot x+b): w\in \R^{d_\Y}, b\in \R\}$ is not dense in $\Y=\CC(\Omega_\Y, \R)$. Nevertheless, it follows from \cite[Theorem 3]{Chen1995_approx_thm} that there exist encoders $E$, such that $D\circ E$ are suitable choices for the output-CEDAP of $\Y$. 
Hence, \cite[Theorem 5]{Chen1995_approx_thm} is a special case of \Cref{thm:density_metric}.
Note that in \cite{Chen1995_approx_thm}, the activation function $g$ is allowed to be discontinuous. Nevertheless, according to \Cref{rem:discontinuity_chen}, one can just consider the decoder $D$ as a mapping into the larger space of bounded, possibly discontinuous, functions $\Omega_\Y\to \R$ and still obtain an approximation result by \Cref{thm:density_metric}. For that, one needs to assume, however, that $g$ maps compact sets into bounded sets.

\begin{remark}
    Instead of classes of neural networks for $\FF$ or $S\subset \Y(\Omega_\Y, \R)$ in \Cref{def:DeepONets}, one can also use other classes of functions and obtain a universal approximation result by \Cref{thm:sequential_density_normed}. For example, one could consider polynomials \cite[Section XIII.3]{Dugundji_Topologie}, splines \cite{Birman_1967, Anderson2014}, or wavelets \cite{Meyer_1993}.
\end{remark}

\begin{remark}
    If $\Y(\Omega_\Y, \R)$ has a Schauder basis and $S\subset \Y$ is a dense set, there exists a Schauder basis of $\Y$ from $S$. This is because Schauder bases are stable against small perturbations, see for example \cite[Proposition I.1.7]{Lindenstrauss_classical_banach_spaces}.
    Therefore, instead of using the dense decoders, one could also use the basis decoders from \Cref{thm:schauder_basis_BAP} corresponding to a Schauder basis of, for example, neural networks and obtain again an approximation result for classical DeepONets.
\end{remark}

\begin{remark}
    Using a sampling encoder, as for classical DeepONets, may not be reasonable in every function space. For example, if the input space is a Lebesgue space $\X = L^2(\Omega_\X, \R)$, not every $f\in \X$ has a continuous representative, which makes sampling not well-defined. As a way out, one could use an encoder corresponding to a Schauder basis, see Section \ref{sec:schauder_DeepONets}. In Hilbert spaces, such as $L^2(\Omega_\X, \R)$, one can also make use of frame encoders, which is outlined in Section \ref{sec:examples_frame}.
\end{remark}

\subsubsection{MIONets}\label{sec:schauder_DeepONets}   
    In this section, we discuss \emph{multi input operator networks} (MIONets), which were introduced in \cite[Section 3.1]{Jin2022_MIONET}. MIONets are designed to approximate operators of the form $\X_1\times ... \times \X_N\to \Y(\Omega_\Y, \R)$, where the $\X_i$ are Banach spaces having Schauder bases. To simplify the presentation, we restrict our attention to the case $N=2$. Moreover, we consider only low-rank MIONets, which constitute the default architecture in \cite{Jin2022_MIONET} due to their lower computational cost compared to high-rank variants. Let $\Y(\Omega_\Y, \R)$ be again a normed (or metric) space of functions defined on a compact set $\Omega_\Y \subset \R^{d_\Y}$.

    \begin{definition}[Low-rank MIONets]\label{def:mionets} 
        Denote the coefficient functionals with respect to some Schauder bases of $\X_j$ by $c_i^{(j)}:\X \to \R$. Let $\varphi^{(1)}:\R^k\to \R^p,$ $\varphi^{(2)}:\R^l\to \R^p$ and $\psi:\Omega_\Y \to \R^p$ be neural networks. A (low-rank) MIONet is a mapping of the form
        \begin{align*}
            \hat{G}: \X_1 \times \X_2 &\longrightarrow \Y(\Omega_\Y, \R)\\
            \big(\hat{G}(f_1,f_2) \big)(z) &= \sum_{i=1}^p  \varphi_{i}^{(1)}\begin{pmatrix}
                c_1^{(1)}(f_1) \\ \vdots \\c_k^{(1)}(f_1)
            \end{pmatrix}\varphi_{i}^{(2)}\begin{pmatrix}
                c_1^{(2)}(f_2) \\ \vdots \\c_l^{(2)}(f_2)
            \end{pmatrix}  \psi_i(z),
        \end{align*}
        where $\varphi_i^{(1)}, \varphi_i^{(2)},\psi_i$ denote the i-th output coordinates of $\varphi^{(1)}, \varphi^{(2)}$ and $\psi$, respectively.
    \end{definition}
    In \cite[Theorem 3.1]{Jin2022_MIONET}, it has been shown that continuous operators $G:\X_1\times ... \times \X_N\to \Y(\Omega_\Y, \R)$ can be approximated by MIONets as in statement \labelcref{fig:approximation_types:(A)}, in which the choice of the approximating sequence of MIONets $\hat{G}_n$ depends on the compact set $K\subset \X_1\times...\times \X_N$. For $N=1$, \Cref{thm:sequential_density_normed} yields the stronger approximation property described in statement \labelcref{fig:approximation_types:(B)} for low-rank MIONets as follows:
    \begin{itemize}
        \item Choose basis encoders on $\X_1$ from \Cref{thm:schauder_basis_BAP};
        \item As in Section \ref{sec:DeepONet_classic}, choose $\varphi^{(1)}$ from a universal function approximator $\FF$ which consists of neural networks;
        \item As in Section \ref{sec:DeepONet_classic}, choose dense decoders on $\Y$ from \Cref{thm:Tn_substitute_dense_set}, that correspond to a dense set $S\subset \Y(\Omega_\Y, \R)$ of neural networks.
    \end{itemize}

    For $N>1$, \Cref{thm:sequential_density_normed} can likewise be applied to approximate continuous operators $\X_1\times ... \times \X_N\to \Y$. Indeed, one may view the Cartesian product as a single Banach space $\X=\X_1\times\cdots\times\X_N$. It is straightforward to verify that $\X$ inherits a Schauder basis from the spaces $\X_i$. Applying \Cref{thm:sequential_density_normed} using Schauder basis encoders on $\X$ then leads to architectures with a single $\varphi\in \FF$ that receives the Schauder basis coefficients associated with all spaces $\X_i$. Consequently, the resulting architecture is not a low-rank MIONet.
    
    To handle low-rank MIONets for $N>1$, one can however combine \Cref{thm:sequential_density_normed} with the idea in \cite[Section 2.5]{Jin2022_MIONET} of identifying $\CC(K_1\times K_2, \Y)$ with the injective tensor product $\big(\CC(K_1, \R)\hat{\otimes}_\varepsilon\CC(K_1, \R)\big)\hat{\otimes}_\varepsilon \Y$ for compact sets $K_i\subseteq\X_i$. By that, one obtains an approximation \Cref{thm:mionet} as in statement \labelcref{fig:approximation_types:(A)}, which extends \cite[Theorem 3.1]{Jin2022_MIONET} in that more general input spaces can be considered, and different encoders than basis encoders can be used, for example, sampling encoders or encoders from Section \ref{sec:wasserstein}. 
    We suspect that for $N>1$, it is not possible, in general, to prove approximation by MIONets as in statement \labelcref{fig:approximation_types:(B)}. This is because the construction of isomorphisms between $\big(\CC(K_1, \R)\hat{\otimes}_\varepsilon\CC(K_1, \R)\big)\hat{\otimes}_\varepsilon \Y$ and $\CC(K_1\times K_2, \Y)$ depends on the compact sets $K_1$ and $K_2$.

    \begin{theorem}\label{thm:mionet}
    Let $\X_1,\X_2$ be metric spaces having the input-EDAP with encoders $E_{n}^{(1)}$ and $E_{n}^{(2)}$, respectively. Let $\FF$ be a universal function approximator and $S$ be a dense subset of a Banach space $\Y$. Then for every $G\in \CC(\X_1\times \X_2, \Y)$ and compact $K\subseteq \X_1\times \X_2$, there exist $\varphi_n^{(1)}, \varphi_n^{(2)}\in \FF$  and $\psi_{n, 1},..., \psi_{n, p_n}\in S$ such that
    \begin{align*}
        \sup_{(f_1, f_2)\in K} \left\Vert G(f_1, f_2) -  \sum_{i=1}^{p_n} \varphi_{n,i}^{(1)}\big(E_n^{(1)}(f_1)\big) \varphi_{n,i}^{(2)}\big(E_n^{(2)}(f_2)\big) \psi_{n, i}\right\Vert  \xrightarrow{n\to\infty}0,
    \end{align*}
    where $\varphi_{n, i}^{(1)}$ and $\varphi_{n, i}^{(2)}$ denote the i-th output coordinates of $\varphi_n^{(i)}$ and $\varphi_n^{(2)}$, respectively. Note that $\varphi_n^{(1)}, \varphi_n^{(2)}, \psi_{n, i}$ depend on $K$, which is omitted within the notation.
\end{theorem}
\begin{proof}
    A proof is provided in Appendix \ref{sec:proof_mionet}.
\end{proof}

\begin{remark}\label{rem:deeponet_vs_mionet}
    As we have seen in this and the previous section, \Cref{thm:sequential_density_normed} covers both approximation by classical DeepONets and MIONets ($N=1$), where the difference is only within the choice of the encoders: sampling versus basis encoders. In contrast, \cite[Theorem 3.1]{Jin2022_MIONET} does not directly provide an approximation result for classical DeepONets from \Cref{def:DeepONets}. The reason is that, in general, Schauder basis encoders cannot coincide with sampling encoders, see \Cref{thm:sampling_schauder_differ_main}. Therefore, additional arguments would be required to obtain an approximation result for classical DeepONets from \cite[Theorem 3.1]{Jin2022_MIONET}. 
    In \cite[Example 2.8]{Jin2022_MIONET}, the authors state that their theory immediately yields an approximation result for DeepONets. As discussed above, this conclusion does not directly apply to the classical DeepONet architecture of \Cref{def:DeepONets}, which employs sampling encoders. However, the term DeepONet is often used more broadly - including in \cite{Jin2022_MIONET} - to encompass architectures with different choices of encoders. In the context of \cite[Example 2.8]{Jin2022_MIONET}, the statement appears to concern DeepONets with basis encoders.
\end{remark}

\subsubsection{BasisONets}\label{sec:basisONet}
For classical and Schauder basis DeepONets, sampling and basis encoders were considered. In this section, we consider BasisONets, which were introduced in \cite{Hua2023} as mappings between Lebesgue spaces $\X=L^2(\Omega_\X, \R)$ and $\Y = L^2(\Omega_\Y, \R)$, where $\Omega_X\subset \R^{d_\X}$ and $\Omega_\Y\subset \R^{d_\Y}$ are compact sets. For BasisONets, the encoder computes inner products with a finite selection of neural networks. We consider a more general setup for the function spaces $\X$ and $\Y$. Namely, assume that $\X = \X(\Omega_\X, \R)$ is a separable Hilbert space and $\Y=\Y(\Omega_\Y, \R)$ is a normed space. 
\begin{definition}[BasisONet]
    Let all $\varphi:\R^k \to \R^p$, $u_1,...,u_k\in \X(\Omega_\X, \R)$ and $\psi_1,..., \psi_p\in \Y(\Omega_\Y, \R)$ be neural networks. A BasisONet is a mapping
    \begin{align*}
        \hat{G}:\X(\Omega_\X, \R) &\longrightarrow \Y(\Omega_\Y, \R)\\
        f &\longmapsto \sum_{i=1}^p \varphi_{i}\big(\langle f, u_1 \rangle_\X, ..., \langle f, u_k\rangle_\X \big) \psi_i,
    \end{align*}
    where $\varphi_i$ denotes the i-th output coordinate of $\varphi$.
\end{definition}

To obtain universal approximation theorems for BasisONets as in statement \labelcref{fig:approximation_types:(B)}, one can simply interpret them as encoder-decoder architectures $\hat{G} = D\circ \varphi \circ E$ and apply \Cref{thm:sequential_density_normed} as follows:
\begin{itemize}
    \item Choose dense encoders on $\X(\Omega_\X, \R)$ from \Cref{thm:Tn_substitute_dense_set_HilbertSpace} corresponding to a dense set $S_\X\subset \X$ of neural networks. If $\X=L^2(\Omega_\X, \R)$ or $\X$ being a Sobolev Hilbert space, $S_\X$ can, for example, be chosen as a set of ReLU-networks \cite{DeVore2021relu_overview};
    \item As in Section \ref{sec:DeepONet_classic}, choose $\varphi$ from a universal function approximator $\FF$ which consists of neural networks;
    \item As in Section \ref{sec:DeepONet_classic}, choose dense decoders on $\Y$ from \Cref{thm:Tn_substitute_dense_set}, that correspond to a dense set $S_\Y\subset \Y(\Omega_\Y, \R)$ of neural networks.
\end{itemize}

In particular, \Cref{thm:sequential_density_normed} implies and extends \cite[Theorem 2.2]{Hua2023}, in which the weaker statement \labelcref{fig:approximation_types:(A)} was shown for BasisONets and the special case of $\X=L^2(\Omega_\X, \R)$ and $\Y=L^2(\Omega_\Y, \R)$. Note that in \cite{Hua2023}, BasisONets are defined with $u_j$ and $\psi_i$ being so-called neural bases. That is, they approximate finitely many elements of an orthonormal basis. This is analogous to the construction of dense encoders and decoders in the proof of \Cref{thm:Tn_substitute_dense_set_HilbertSpace}.

\subsubsection{Frame Architectures}\label{sec:examples_frame}
In this section, we make use of the frame encoders and decoders introduced in Section \ref{sec:Tn_frames} to define a class of encoder--decoder architectures based on frames. Special cases of these architectures were studied in \cite{Schwab2023} and \cite[Definition 4.12]{Castro2023_ONBs}, where Riesz bases and orthonormal bases were used in place of frames, respectively.
Throughout this section, let $\X$ and $\Y$ be infinite-dimensional separable Hilbert spaces. The finite-dimensional case can be treated analogously.
        \begin{definition}[Frame Architectures]\label{def:frame_architectures}
            Consider a frame $(f_i)_{i\in \N} \subset \X$ of $\X$ along with a dual frame $(f_i^*)_{i\in\N}\subset \X$. Let $(g_i)_{i\in \N}\subset \Y$ be a frame of $\Y.$ 
            Further, let $\varphi:\K^k \to \K^p$ be a neural network. A frame architecture is a mapping
            \begin{align*}
                \hat{G}:\X &\longrightarrow \Y\\
                f &\longmapsto \sum_{i=1}^p \varphi_{i} \left( \langle f, f^*_1 \rangle, \dots, \langle f, f^*_k \rangle \right) g_i,
            \end{align*}
             where $\varphi_i$ denotes the i-th output coordinate of $\varphi$
        \end{definition}
        To obtain a universal approximation theorem for frame architectures as in statement \labelcref{fig:approximation_types:(B)}, we can interpret them as encoder-decoder architectures $\hat{G} = D\circ \varphi \circ E$ as follows and apply \Cref{thm:sequential_density_normed}:
\begin{itemize}
    \item Choose frame encoders on $\X$ and frame decoders on $\Y$ from \Cref{thm:frames};
    \item As in Section \ref{sec:DeepONet_classic}, choose $\varphi$ from a universal function approximator $\FF$ which consists of neural networks.
\end{itemize}

In particular, \Cref{thm:sequential_density_normed} implies \cite[Theorem 3.1]{Schwab2023}. More precisely, both results establish the stronger approximation statement \labelcref{fig:approximation_types:(B)}, whereas \Cref{thm:sequential_density_normed} applies in the more general setting of frames rather than being restricted to Riesz bases.
Furthermore, \Cref{thm:sequential_density_normed} yields an approximation theorem for Deep-H-ONets \cite[Definition 4.12]{Castro2023_ONBs}, whose encoders and decoders are induced by orthonormal bases. Since orthonormal bases are a particular class of Riesz bases, this result is likewise covered by \cite[Theorem 3.1]{Schwab2023}.

\subsection{Applications in optimal transport}\label{sec:applications_OT}
We have seen in Sections \ref{sec:approx_theorem} and \ref{sec:wasserstein} that encoder-decoder architectures can be used to approximate continuous operators $\X\to \Y$, where $\X$ or $\Y$ are $p$-Wasserstein-spaces $(\PP_p(\Omega), W_p)$. In this section, we present some applications which may benefit from our developed theory.
\subsubsection{Surrogates for (entropic) optimal transport maps}
In this section, we discuss some operators arising in \emph{linear optimal transport} that fit into our approximation framework.
The concept of linear optimal transport has been proposed in \cite{Wang2012_LOT}, and the basic idea is to embed the non-linear space $(\PP_p(\Omega), W_p)$ into the Hilbert space $L^2(\Omega, \R^d; \rho)$ to benefit from the latter's structure. This is typically done by fixing a reference measure $\rho$ and then identifying a measure $\mu$ with the corresponding optimal transport map $T_{\rho, \mu}\in L^2(\Omega, \R^d; \rho)$ from $\rho$ to $\mu$.
We observe below that these embedding operators $G_\rho: \mu \mapsto T_{\rho, \mu}$ can be approximated by encoder-decoder architectures, according to \Cref{thm:sequential_density_normed}.

Therefore, it will be interesting to investigate in future work whether encoder-decoder surrogates for $G_\rho$ can be beneficial in methods that require solving many optimal transport problems, that is, when $G_\rho$ often needs to be evaluated at different $\mu$. Since repeatedly solving optimal transport problems can become computationally demanding, the encoder-decoder surrogates may reduce the overall computational cost.
For example, one could investigate replacing the classical solver for estimating the optimal transport maps within \emph{optimal transport neural operators} (OTNOs) \cite{Kovachki2025_OTNO}. OTNOs are a recent operator learning approach for learning PDE solutions on varying domains by using optimal transport maps to move between complex geometries $(\mu)$ and a fixed reference domain $(\rho)$. On the reference domain, a Fourier Neural Operator is then applied.

Let us state the following version of Brenier's theorem \cite[Theorem 1.16]{Chewi_Statistical_OT}.
\begin{theorem}
    Let $\rho, \mu \in \PP_2(\R^d)$ and assume that $\rho$ is absolutely continuous w.r.t. the Lebesgue measure. Then there exists a convex function $\varphi_{\rho, \mu}:\R^d\to \R$ such that the push-forward of $\rho$ under $\nabla \varphi_{\rho, \mu}$ (defined $\rho$-almost everywhere) is $\mu$ and
    \begin{align*}
        W_2(\rho, \mu)^2 = \int_{\R^d}\Vert x - \nabla \varphi_{\rho, \mu}(x)\Vert^2 d\rho(x).
    \end{align*}
    If $\psi_{\rho, \mu}$ is another such convex function, then $\nabla \psi_{\rho, \mu} = \nabla \varphi_{\rho, \mu}$ holds $\rho$-almost everywhere.
\end{theorem}

Therefore, the optimal transport map $T_{\rho, \mu} \coloneqq \nabla \varphi_{\rho, \mu}$ is uniquely determined up to a $\rho$-null set. $T_{\rho, \mu}$ is also called \emph{Brenier map} \cite{Villani_Optimal_Transport, Delalande2023_stability_brenier, letrouit2025} or \emph{Monge map} \cite{Peyre2019} and is the unique solution of the Monge transport problem. 
Note that the assumption on $\rho$ being absolutely continuous with respect to the Lebesgue measure can be relaxed, see e.g. \cite[Theorem 9.4]{Villani_OldNew}.
It has been noted in \cite{Delalande2023_stability_brenier, letrouit2025} that $G_\rho:\mu\mapsto T_{\rho, \mu}$ is continuous, which we state in the theorem below. A detailed proof can be found in \cite[Proposition 1.4]{letrouit2023_lecturenotes}.
\begin{theorem}
    Let $\rho\in \PP_2(\R^d)$ be absolutely continuous w.r.t. the Lebesgue measure. Then the following operator is continuous with respect to the $2$-Wasserstein distance:
    \begin{align*}
        G_\rho:\PP_2(\R^d)&\longrightarrow L^2(\R^d, \R^d; \rho)\\
        \mu &\longmapsto T_{\rho, \mu}.
    \end{align*}
\end{theorem}
Note that for every $\mu, \nu\in \PP_2(\R^d)$, it holds that $W_2(\mu, \nu)\leq \Vert G_\rho(\mu)-G_\rho(\nu)\Vert_{L^2}$ \cite{Delalande2023_stability_brenier, letrouit2025}. This implies that $G_\rho$ is injective and its inverse on $G_\rho(\PP_2(\R^d))$ is Lipschitz continuous. Therefore, $G_\rho$ is an embedding of $\PP_2(\R^d)$ into $L^2(\R^d, \R^d; \rho)$. 
In linear optimal transport, the embedding $\tilde{G}_\rho: \mu\mapsto (G_\rho(\mu) - \textup{id})$ is considered, see for example \cite[Section 2.3]{Wang2012_LOT}.
Since the operators $G_\rho$ and $\tilde{G}_\rho$ are continuous, they can be approximated by encoder-decoder architectures according to \Cref{thm:sequential_density_normed} as follows:
\begin{itemize}
    \item Since $(\PP_2(\R^d), W_2)$ has the input-EDAP, choose encoders as in \Cref{thm:Wasserstein_DEDAP};
    \item As in Section \ref{sec:DeepONet_classic}, choose $\varphi$ from a universal function approximator $\FF$ which consists of, for example, neural networks;
    \item Since $L^2(\R^d, \R^d; \rho)$ is a separable Hilbert space, choose frame decoders from \Cref{thm:frames}, or dense decoders from \Cref{thm:Tn_substitute_dense_set} corresponding to a dense set $S\subset L^2(\R^d, \R^d; \rho)$ of, for example, neural networks.
\end{itemize}

\begin{remark}
    A very recent and active field of research is to inspect Hölder continuity of $G_\rho$ under additional assumptions on $\rho$ and $\mu$. For example, if $\Omega\subset \R^d$ is compact and convex and if the density of $\rho$ is bounded away from zero and infinity, it has been shown that $G_\rho:\PP(\Omega)\to L^2(\R^d, \R^d; \rho)$ has Hölder exponent $\frac{1}{6}$ with respect to any $p$-Wasserstein distance \cite[Theorem 4.2]{Delalande2023_stability_brenier}.
    If $\mu$ is mapped to the \emph{Brenier potential} $\varphi_{\rho, \mu}$, instead of to $T_{\rho, \mu} =\nabla \varphi_{\rho, \mu}$, the Hölder exponent is $\frac{1}{2}$. For a review on Hölder continuous embeddings of $\PP_p(\Omega)$ into $L^2(\R^d, \R^d;\rho)$, we refer to \cite{letrouit2025}. Stronger regularity properties of $G_\rho$ may improve quantitative approximation rates of encoder-decoder architectures.
\end{remark}
\begin{remark}[Entropic optimal transport]
    A regularized version of the Brenier map $T_{\rho, \mu}$ is the so-called \emph{entropic map} $T_{\rho, \mu}^\varepsilon$ \cite{Pooladian2024}, also called \emph{entropic Brenier map} \cite{Divol2025_stability_brenier}. As shown in \cite{Pooladian2024}, it serves as an approximation of $T_{\rho ,\mu}$ with better computational properties.
    One theoretical advantage is that the requirements on $\rho$ - being absolutely continuous - can be relaxed for obtaining well-posedness and Hölder continuity of $G_\rho^\varepsilon:\mu\mapsto T_{\rho, \mu}^\varepsilon$. Consider any ball $\Omega=\overline{B}_R(0)\subset \R^d$. Then for any $\rho\in \PP(\Omega)$, it is due to \cite[Corollary 3.3]{Divol2025_stability_brenier} that $G_\rho^\varepsilon:\PP(\Omega)\to L^2(\R^d, \R^d;\rho)$ is even Lipschitz continuous with respect to the $2$-Wasserstein distance. 
    Hence, encoder-decoder architectures can also approximate $G_\rho^\varepsilon$.
    A similar Lipschitz continuity result has been derived for Schrödinger maps \cite[Corollary 2.4]{Carlier2024}.
\end{remark}

\subsubsection{Geodesic operator networks}
Gracyk et al. introduced \emph{geodesic operator networks (GeONets)} in \cite{Gracyk_geonet} to learn the operator which maps a pair of absolutely continuous probability measures $\mu_0, \mu_1\in\PP_2^{ac}(\Omega)$ to the corresponding Wasserstein geodesic $\mu:[0,1]\to \PP_2^{ac}(\Omega)$, where $\Omega\subseteq \R^m$ is equipped with the Euclidean norm $\vert \cdot \vert$.
To the best of our knowledge, it has not been shown yet that GeONets can indeed approximate the geodesic operator $(\mu_0, \mu_1)\mapsto \mu$ with respect to a suitable topology. We briefly outline how the theory developed in this work could provide a route towards such an approximation result.

In \cite[Appendix B]{Gracyk_geonet}, optimality conditions were derived for the Wasserstein geodesic problem of finding $\mu$ from $\mu_0, \mu_1$. GeONets then seek to learn solutions $\rho, u:[0,1]\times \Omega\to \R$ of these conditions, which read as follows:
\begin{align*}
    \begin{cases}
        \partial_t \rho + \textup{div}(\rho\nabla u) = 0 &\textup{(continuity equation)}; \\
        \partial_t u+ \frac{1}{2}\Vert \nabla u \Vert^2_2 = 0 &\textup{(Hamilton-Jacobi equation)};\\
        \rho(\cdot, 0) = \rho_0, \rho(\cdot, 1) = \rho_1,
    \end{cases}
\end{align*}
where $\rho_0, \rho_1$ are the densities of $\mu_0, \mu_1$, respectively. GeONets consist of two encoder-decoder architectures that approximate the solution operators
\begin{align*}
    G_{CE}: \PP_2^{ac}(\Omega)\times \PP_2^{ac}(\Omega) &\longrightarrow \Y_{CE}([0,1]\times \Omega, \R)\\
    (\mu_0, \mu_1)&\longmapsto \rho, \\
    G_{HJ}: \PP_2^{ac}(\Omega)\times \PP_2^{ac}(\Omega) &\longrightarrow \Y_{HJ}([0,1]\times \Omega, \R)\\
    (\mu_0, \mu_1)&\longmapsto u,
\end{align*}
corresponding to the continuity and Hamilton-Jacobi equation, respectively. 
$\Y_{CE}$ and $\Y_{HJ}$ denote function spaces for the respective PDE solution functions. $G_{CE}(\mu_0, \mu_1) = \rho$ is then the density representation of the desired geodesic $\mu$.

Both GeONet encoder-decoder architectures are closely related to the architectures covered by \Cref{thm:mionet}, which are generalizations of the low-rank MIONets from \Cref{def:mionets}. Indeed, the GeONet architectures are both of the form
\begin{align*}
    \sum_{i=1}^{p} \varphi_{i}^{(1)}\big(E^{(1)}(\mu_0)\big) \varphi_{i}^{(2)}\big(E^{(2)}(\mu_1)\big) \psi_{i},
\end{align*}
where $\varphi^{(1)}, \varphi^{(2)}$ and $\psi_i:[0,1]\times \Omega\to \R$ are neural networks. The pair of GeONet encoders $E^{(1)}, E^{(2)}$ evaluate the densities $\rho_0, \rho_1$ of $\mu_0, \mu_1$ at finitely many points $x_k\in \Omega$. This is similar to the encoders from \Cref{def:wasserstein_encoder_decoder} and \Cref{coro:wasserstein_abs_cont_EDAP}, which compute $\mu_0(I_k)=\int_{I_k} \rho_0(x) d\lambda(x)$ for finitely many disjoint sets $I_k\subset \Omega$.

Some things remain to be verified in order to get an approximation result by \Cref{thm:mionet} for the approximation of the solution operators $G_{CE}$ and $G_{HJ}$ by the two GeONet encoder-decoder architectures. 
\begin{itemize}
    \item Show that the solutions of the PDEs lay within Banach spaces $\Y_{CE}, \Y_{HJ}$ of functions. Further, check whether there exist dense subsets $S_{CE}\subset \Y_{CE}$ and $S_{HJ}\subset \Y_{HJ}$ of neural networks. These can then be used in \Cref{thm:mionet}.
    \item Show that both solution operators $G_{CE}$ and $G_{HJ}$ are continuous with respect to the $2$-Wasserstein distance. For each operator, then use two encoders from \Cref{coro:wasserstein_abs_cont_EDAP} in \Cref{thm:mionet}.
\end{itemize}
\begin{remark}
    If it turned out that the operators $G_{CE},G_{HJ}$ were not continuous, but Bochner integrable, one could combine our results with the ideas in \cite[Theorem 3.1]{Lanthaler2022} to approximate these operators by the GeONet's encoder-decoder architectures, but with respect to the Bochner norm.
\end{remark}

\section{Conclusion}
We studied the approximation theory of encoder-decoder architectures in the topology of uniform convergence on compact sets. In this setting, two notions of universal approximation naturally arise; see statements \labelcref{fig:approximation_types:(A)} and \labelcref{fig:approximation_types:(B)} in \Cref{fig:approximation_types}. The first requires that, for every continuous operator and compact subset of its input space, there exists a sequence of approximators converging uniformly on that compact set. The second requires the existence of a single sequence that converges uniformly on every compact subset to the operator. While statement \labelcref{fig:approximation_types:(A)} is commonly considered in the operator-learning literature, we established \labelcref{fig:approximation_types:(B)} in a new \Cref{thm:sequential_density_normed} for diverse classes of encoder-decoder architectures and choices for input and output spaces. We proved in \cref{thm:dense_Frechet_hemicompact} that this type of universal operator approximation is a topologically stronger statement in typical operator learning settings, namely whenever $\X$ is an infinite-dimensional normed space (or more generally whenever $\X$ is not hemicompact).
A further advantage of \labelcref{fig:approximation_types:(B)} over \labelcref{fig:approximation_types:(A)} is that it may facilitate the derivation of approximation theorems in Bochner spaces. Indeed, uniform convergence on every compact subset implies pointwise convergence, allowing one to invoke the dominated convergence theorem. This suggests a possible route for extending the universal Bochner approximation theorem \cite[Theorem 3.1]{Lanthaler2022} for classical DeepONets to broader classes of encoder-decoder architectures and more general input and output spaces. We leave this direction for future research.

Moreover, we introduced the EDAPs and CEDAPs, sufficient properties of the underlying normed or metric input and output spaces, that enable universal operator approximation by encoder-decoder architectures as in statements \labelcref{fig:approximation_types:(B)} and \labelcref{fig:approximation_types:(A)}; see \Cref{thm:density_metric,thm:mionet,thm:sequential_density_normed}.
These properties are fulfilled by many spaces considered within operator learning frameworks, for example, normed function spaces such as Lebesgue spaces, Sobolev spaces, spaces of continuously differentiable functions, or - more generally - normed spaces having the (bounded) approximation property.
Beyond normed spaces, the EDAPs and CEDAPs are sufficiently general to include also metric spaces. For instance, we showed that Fr\'echet spaces having Schauder bases as well as $p$-Wasserstein spaces of probability measures can be considered as input or output spaces. Further, Skorohod spaces of c\`adl\`ag functions are suitable input spaces in our operator approximation framework.
These examples not only demonstrate the flexibility of the (C)EDAPs regarding suitable input and output spaces, but also regarding possible choices of encoders and decoders. For example, they allow nonlinear encoders and decoders, discontinuous encoders and non-Lipschitz decoders, thereby relaxing assumptions that are often imposed in the literature.
To the best of our knowledge, the existing operator learning theory for statement \labelcref{fig:approximation_types:(B)} handles operators between
separable Hilbert spaces using Riesz basis encoders and decoders.
Our \cref{thm:sequential_density_normed} hence significantly extends the existing theory for encoder-decoder architectures regarding possible input and output spaces, as well as possible encoder and decoder constructions. Regarding the theory related to statement \labelcref{fig:approximation_types:(A)}, we also make some contributions. For example, some
of our naturally constructed encoders for Wasserstein and Skorohod spaces are currently not covered by the existing theory
for statement (A), which is due to their discontinuity.

Rather than proving separate approximation theorems for individual architectures, our approximation theorems apply to many well-known encoder-decoder architectures at once, for example, classical DeepONets, MIONets, BasisONets, Deep-H-ONets, or Riesz bases architectures. 
We pointed out that our framework not only unifies existing universal approximation theorems for these models, but also extends them by allowing more general input and output spaces and greater flexibility in the choice of encoders and decoders.

We conclude by highlighting four directions for future research.
First, it would be interesting to inspect which other metric spaces have the (C)EDAPs.
Second, if the EDAPs or CEDAPs were refined to incorporate quantitative approximation rates for the identity operator - for instance, rates depending on covering numbers of compact sets - it would be natural to investigate whether analogous rates can be transferred to the approximation of operators satisfying additional regularity assumptions, such as Lipschitz continuity or Fr\'echet differentiability.
Third, for different topologies than uniform convergence on compact sets, such as Bochner spaces $L^p(\X, \Y; \mu)$, it would be interesting to explore analogues of the (C)EDAPs for $\X$ and $\Y$ that directly lead to approximation capabilities of encoder-decoder architectures in the respective topology.
Last but not least, the discussion in Section \ref{sec:applications_OT} indicates that our framework may also be useful for future developments in optimal transport, including the approximation-theoretic analysis of GeONets and related architectures.

\appendix
\section{Supplement topology}
    \subsection{Proof of Lemma \ref{lem:adherent_point_compact_open_topology}}\label{subsec:proof_adherent_point_compact_open_topology}
        \begin{proof}
            For the proof, we follow the arguments used by Dugundji \cite[Chapter XII, Theorem 7.2]{Dugundji_Topologie}, where the convergence of sequences in the compact-open topology is characterized (see also \Cref{lem:convergence_compact_open_topology}). \newline
            Assume that $G\in \CC(\X, \Y)$ is an adherent point of $S$. Let $K\subseteq \X$ be compact and $\varepsilon>0$. By continuity of $G$, for every $x\in K$ there is some $\delta_x>0$ such that 
            \begin{align*}
                G(B_{\delta_x}(x) ) \subseteq B_\varepsilon(G(x)).
            \end{align*}
            Since $K$ is compact, there are finitely many $x_1,\dots,x_n\in K$ such that 
            \begin{align*}
                K \subseteq \bigcup_{i=1}^n B_{\frac{1}{2}\delta_{x_i}}(x_i).
            \end{align*}
            For each $i=1,\dots,n,$ we define the compact set $K_i \coloneqq K \cap \overline{B}_{\frac{1}{2}\delta_{x_i}}(x_i)$, which is contained in $B_{\delta_{x_i}}(x_i)$ and satisfies $K = \cup_{i=1}^n K_i.$ Consider the set
            \begin{align*}
                U \coloneqq \bigcap_{i=1}^n \Big\{ f\in \CC(\X, \Y): f(K_i) \subseteq B_\varepsilon(G(x_i)) \Big\},
            \end{align*}
            which is open in the compact-open topology by \Cref{def:compact_open_topology} and contains $G$. Since $G$ is an adherent point of $S$, there exists $f\in S\cap U$. Furthermore, each $x\in K$ must belong to some $K_i$. Therefore, it follows from the triangular inequality that 
            \begin{align*}
                d_\Y(f(x)\, ,\, G(x)) \leq d_\Y(f(x)\, ,\, G(x_i)) + d_\Y(G(x_i)\, ,\, G(x)) < 2\varepsilon
            \end{align*}
            for all $x\in K.$ As a consequence,
            \begin{align*}
                \sup_{x\in K} d_\Y(f(x)\, ,\, G(x)) \leq 2 \varepsilon.
            \end{align*}
            Hence, since $\varepsilon$ has been chosen arbitrarily, we can find a sequence $(f_n)_{n\in\N}$ in $S$ converging uniformly to $G$ on $K$.
            
            For the other direction, consider any finitely many compact sets $K_1,\dots,K_n\subseteq \X$ and open sets $V_1,\dots,V_n\subseteq \Y$ such that $G(K_i)\subseteq V_i$. In other words, 
            \begin{align*}
                G\in U\coloneqq \bigcap_{i=1}^n U_{K_i, V_i},
            \end{align*}
            where $U_{K_i, V_i}\coloneqq \{f\in \CC(\X, \Y): f(K_i) \subseteq V_i\}$ belongs to the subbase of the compact-open topology from \Cref{def:compact_open_topology}. Note that any open neighborhood of $G$ can be written as a union of sets like $U$. Therefore, in order to show that $G$ is an adherent point of $S$, it suffices to show that there is always a $\Tilde{G}\in S\cap U$.\newline
            As $G\in \CC(\X, \Y)$ is continuous, the sets $G(K_i)$ are compact. Hence, according to \Cref{coro:compact_set_uniform_extension}, there is some $\varepsilon_i>0$ such that for all $i\in \{1,\dots,n\}$
            \begin{align*}
                \bigcup_{x\in K_i} B_{\varepsilon_i}(G(x)) \subseteq V_i.
            \end{align*}
            Let $K$ denote the union of all $K_i$, which is also compact as it is the finite union of compact sets. By assumption, there exists a sequence $(G_n)_{n\in \N}$ in $S$ such that
            \begin{align*}
                \sup_{x\in K} d_\Y\big(G(x)\, ,\, G_n(x) \big) \xrightarrow{n\to\infty} 0.
            \end{align*}
            Therefore, there exists $\Tilde{G}\in S$ such that
            \begin{align*}
                \sup_{x\in K} d_\Y\big(G(x)\, ,\, \Tilde{G}(x) \big) < \varepsilon\coloneqq \min \big \{\varepsilon_1,\dots,\varepsilon_n\big\}.
            \end{align*}
            As a consequence, for every $x\in K_i$ it holds that $\Tilde{G}(x)\in B_\varepsilon(G(x))\subseteq V_i,$ which implies $\Tilde{G}\in U$. Thus, $G$ is an adherent point of $S$ with respect to the compact-open topology.
        \end{proof}

        In the following, we state the auxiliary results from basic topology, which are needed for the proof of \Cref{lem:adherent_point_compact_open_topology}. The first lemma can, for example, be found in \cite[Chapter XI.4]{Dugundji_Topologie}.

        \begin{lemma}\label{lem:distance_compact_closed_sets}
            Let $(\X, d)$ be a metric space, $K\subseteq \X$ compact and $A\subseteq \X$ be closed such that $K\cap A = \emptyset$. Then the distance $\dist(K, A)$ between $K$ and $A$ is positive, i.e.,
            \begin{align*}
                0 < \dist(K, A) = \inf \{d(x,y): x\in K, y\in A\}.
            \end{align*}
        \end{lemma}
        
        \begin{corollary}\label{coro:compact_set_uniform_extension}
            Let $\X$ be a metric space, $K\subseteq \X$ compact and $U\subseteq \X$ open such that $K\subseteq U$. Then there exists some $\varepsilon>0$ such that 
            \begin{align*}
                \bigcup_{x\in K} B_\varepsilon(x) \subseteq U, 
            \end{align*}
            where $B_\varepsilon(x)$ denotes the open ball around $x$ with radius $\varepsilon$.
        \end{corollary}
        \begin{proof}
            If $U=\X$, the result is trivial, so assume $\X \setminus U\neq \emptyset$. As $\X \setminus U$ is closed and has empty intersection with $K$, \Cref{lem:distance_compact_closed_sets} implies that $\dist(K, \X \setminus U)>0.$ Therefore, by choosing $0<\varepsilon<\dist(K, \X \setminus U)$ we obtain
            $\bigcup_{x\in K} B_\varepsilon(x) \subseteq U$.
        \end{proof}

    \subsection{Further Results}

        \begin{lemma}\label{lem:metric_separability}
            A metric space $\X$ is separable if and only if for every $\varepsilon>0$ there exists a sequence $(x_n)_{n\in\N}$ in $\X$ such that 
            \begin{align*}
                \X = \bigcup_{n=1}^\infty B_\varepsilon(x_n).
            \end{align*}
        \end{lemma}
        \begin{proof}
            Assume that for every $\varepsilon>0$ there is a sequence $(x_n(\varepsilon))_{n\in \N}$ in $\X$ such that 
            \begin{align*}
                \X = \bigcup_{n=1}^\infty B_\varepsilon(x_n(\varepsilon)).
            \end{align*}
            In particular, for $\varepsilon_m=\frac{1}{m}$, such a sequence exists. It is easy to verify that the countable set 
            \begin{align*}
                \left\{ x_n\left(\varepsilon_m\right): n,m\in \N \right\}
            \end{align*}
            is dense in $\X$, implying separability of $\X$. The converse is evident.
        \end{proof}
        
        \begin{theorem}\label{thm:X_nonseparableMetricSpace_equivalence}
            Let $\X$ be a non-seperable metric space and $\Y$ be a normed space which has at least one dimension. Then there exists a set $S\subset \CC(\X, \Y)$ that is dense with respect to the compact-open topology, but $S$ is not sequentially dense.
        \end{theorem}
        \begin{proof}
            Since $\X$ is not separable, it is due to  \Cref{lem:metric_separability} that there exists some $\varepsilon>0$ such that for every sequence $(x_n)_{n\in \N}$ in $\X$ it holds that 
            \begin{align}\label{eq:choice_eps_nonseparable}
                \X \neq \bigcup_{n=1}^\infty B_\varepsilon(x_n).
            \end{align}
            Define the set
            \begin{align*}
                S \coloneqq \left \{ f\in \CC(\X, \Y): f=0 \textup{ \, on \, } \X \setminus\left( \bigcup_{n=1}^\infty B_\varepsilon(x_n)\right)  \textup{ \, for \, } x_n\in \X\right \}.
            \end{align*}
            Consider $f\in \CC(\X, \Y)$ and some compact $K\subseteq \X$. The compactness of $K$ guarantees the existence of finitely many $x_1,\dots, x_n\in K$ such that 
            \begin{align*}
                K \subseteq V\coloneqq \bigcup_{i=1}^n B_\varepsilon(x_i).
            \end{align*}
            By Urysohn's lemma, see for example \cite[Chapter 4]{Dugundji_Topologie}, there exists a continuous mapping $\varphi:\X \to [0,1]$ such that $\varphi(x) = 1$ for all $x\in K$, and $\varphi(x) = 0$ for all $x\in \X \setminus V$.
            Consequently, the function $\varphi f$ is continuous and belongs to $S$. Further, $f(x) - \varphi(x)f(x) = 0$ for all $x\in K$, which shows that $S$ is dense in $\CC(\X, \Y)$ in the compact-open topology according to \Cref{thm:density_compact_open_topology}.
            On the other hand, consider a constant function $f=a$ for some $a\in \Y\setminus \{ 0\}$, and let $(f_n)_{n\in \N}$ be an arbitrary sequence in $S$. 
            By definition of $S$, for each $n\in \N$, there must exist a sequence $(x_\ell^{(n)})_{\ell\in \N}$ in $\X$ such that 
            \begin{align*}
                f_n(x) = 0 \textup{\, for all \,} x\in \X \setminus \left( \bigcup_{\ell=1}^\infty B_\varepsilon(x_\ell^{(n)})\right).
            \end{align*}
            By the choice of $\varepsilon$ in \Cref{eq:choice_eps_nonseparable}, there exists some $x'\in \X$ that does not belong to any of the balls $B_\varepsilon(x_\ell^{(n)})$ for $n,\ell\in \N$. Hence, for every $n\in \N$ it is $f_n(x') = 0 \neq a = f(x')$ which implies that $f_n$ does not converge uniformly to $f$ on the compact set $\{x'\}$. Thus, $S$ cannot be sequentially dense in $\CC(\X, \Y)$ in the compact-open topology, according to \Cref{thm:sequential_density_compact_open_topology} (not even in the topology of point-wise convergence).
        \end{proof}
        
        It is well-known that compact metric spaces are separable, see for example \cite[Chapter XI, Theorem 4.1.]{Dugundji_Topologie}.
        Below we give a simple proof based on \Cref{lem:metric_separability}.\begin{lemma}\label{lem:compact_metric_separable}
            Every compact metric space $\X$ is separable.
        \end{lemma}
        \begin{proof}
            Let $\varepsilon>0$. Then $\X = \cup_{x\in \X}B_\varepsilon(x).$ By compactness of $\X$ there are finitely many $x_n\in \X$ for $n\leq N\in \N$ such that $\X = \cup_{n\leq N} B_\varepsilon(x_n)$. Hence, $\X$ is separable according to \Cref{lem:metric_separability}.
        \end{proof}
        
        \begin{lemma}\label{lem:example_dense_FU_not_FU}
            Let $(\Z, \mathfrak{T})$ be a topological space, which is not a Fréchet-Urysohn space. For any object $x\notin \Z$ define $\Z' = \Z \cup \{ x\}$ and the topology
            \begin{align*}
                \mathfrak{T}' = \{ \emptyset\} \cup \big\{U\cup \{x\}: U\in \mathfrak{T} \big\}.
            \end{align*}
            on $\Z'$. Then $(\Z', \mathfrak{T}')$ is not a Fréchet-Urysohn space, but every dense set $D\subseteq \Z$ is sequentially dense.
        \end{lemma}
        \begin{proof}
            It is straightforward to verify that $\mathfrak{T}'$ defines a topology on $\Z'$. Let $D\subseteq \Z'$ be dense. Since $\{x\} = (\emptyset\cup\{x\})\in \mathfrak{T}'$, it follows that $x\in D$. Due to the fact that $x$ belongs to every non-empty open set $U'\in \mathfrak{T}'$, it follows that the constant sequence $x_n= x\in D$ converges to every $z\in \Z'$ in the topology $\mathfrak{T}'$. Therefore, $D$ must also be sequentially dense in $\Z'$.
            Since $(\Z, \mathfrak{T})$ is not a Fréchet-Urysohn space, there is some $S\subseteq \Z$ which has an adherent point $s\in \overline{S}$, but no sequence in $S$ converges to $s$ in the topology $\mathfrak{T}$. Therefore, and due to $x\notin \Z$, no sequence in $S$ converges to $s$ in the topology $\mathfrak{T}'$. Hence, $(\Z', \mathfrak{T}')$ is not a Fréchet-Urysohn space.
        \end{proof}
    \begin{lemma}\label{lem:compact_set_convergent_sequence}
        Let $(\X, d)$ be a metric space with metric $d$ and $(x_n)_{n\in\N}\subset \X$ be a convergent sequence in $\X$ with limit $x\in \X$. Then the set $C\coloneqq\{x\}\cup \{x_n: n\in \N\}$ is compact.
    \end{lemma}
    \begin{proof}
        Let $\{U_i : i\in I\}$ be an open cover of $C$ with index set $I$. The aim is to find a finite subcover. As $x\in C\subseteq \bigcup_{i\in I }U_i,$ there exists an index $i_x\in I,$ such that $x\in U_{i_x}.$ As $U_{i_x}$ is an open neighbourhood of $x$ and $x_n \xrightarrow[]{n\to \infty}x,$ there exists an index $N\in\N$ such that $x_n\in U_{i_x}$ for all $n\geq N.$ The remaining elements $\{x_1,\dots,x_{N-1}\}$ can be covered by a finite number of open sets leading to the desired finite subcover of $C.$
    \end{proof}

\section{Notes on sampling}
\subsection{BAP of $\CC^{k}$-spaces}\label{app:sampling_Ck}
Consider the space $\CC^k(\Omega, \R)$ equipped with the $\CC^k$-norm.
For simplicity, we only consider $k=1$, $\Omega=[0,1]$, but it might be possible to derive similar results in more general settings.
\begin{corollary}\label{app:coro_Ck}
    Consider $\CC^{1}([0,1], \R)$ equipped with the $\CC^1$-norm. For $n\in \N$ let $\left\{y_1^{(n)},\dots,y_{k(n)}^{(n)} \right\}$ be an $1/n$-covering for $[0,1]$ and $P_{\frac{1}{n}, i}$ the partition of unity on $[0,1]$ from \Cref{lem:partition_unity}.
    Then the operators
    \begin{align*}
        \tilde{\T}_n: \CC^{1}([0,1], \R) &\longrightarrow \CC^{1}([0,1], \R)\\
        f &\longmapsto \left( y \mapsto f(0) + \sum_{i=1}^{k(n)} f'(y_i^{(n)}) \int_0^y  P_{\frac{1}{n}, i}(x) \diff x \right)
    \end{align*}
    are linear, bounded, have finite rank and converge uniformly to the identity operator on every compact $K\subset \CC^{1}([0,1], \R)$, with respect to the the $\CC^1$-norm. In other words, $\tilde{\T}_n$ are suitable mappings in \Cref{def:BAP} of the $\lambda$-BAP.
\end{corollary}
\begin{proof}
    Using the fundamental theorem of calculus $f(y) = f(0) + \int_0^y f'(x) \diff x,$ we compute
    \begin{align*}
        \vert f(y) - (\tilde{\T}_n f)(y) \vert &= \left\vert \int_0^y f'(x) - \sum_{i=1}^{k(n)} f'(y_i^{(n)}) P_{\frac{1}{n}, i}(x) \diff x \right\vert\\
        &=\left\vert \int_0^y f'(x) - (T_nf')(x) \diff x \right\vert\\
        &\leq (1-0) \sup_{y\in [0,1]} \vert f'(y) - (T_nf')(y) \vert \xrightarrow{n\to \infty}0,
    \end{align*}
    where the convergence to zero for $n\to\infty$ is obtained by applying \Cref{thm:sampling_BAP} with $f'\in \CC([0,1], \R)$ and $T_n$ defined as in \Cref{thm:sampling_BAP}. 
\end{proof}

\subsection{Schauder basis coefficients and sampling}\label{sec:appendix_coeff_sampling}
Below, we provide a comparison between the sampling encoders $E_n$ from Section \ref{sec:Tn_sampling} with the basis encoders $\tilde{E}_n$ introduced in Section \ref{sec:Tn_schauder}. Since both of them can be used in encoder-decoder architectures for approximating operators, see \Cref{thm:sequential_density_normed}, it is natural to ask whether sampling encoders can coincide with basis encoders for every $n\in \N$.
Since the sampling encoders $E_n$ and the basis encoders $\tilde{E}_n$ have different output dimensions, it only makes sense to compare $E_n$ with $\tilde{E}_{k(n)}$, where $k(n)$ is the number of sampling points used for the definition of $E_n$. We have not found a complete answer to the above question. However, under certain conditions on the sampling points, the following theorem reveals that $E_n$ cannot coincide with $\tilde{E}_{k(n)}$ for sufficiently large $n\in \N$.

\begin{theorem}\label{thm:sampling_schauder_differ_main}
    Let $\Omega$ be a compact metric space without isolated points, assume that $\CC(\Omega, \K)$ has a Schauder basis $(b_n)_{n\in \N}$ and denote the basis encoders by $\tilde{E}_n$. Consider a sequence $(y_n)_{n\in \N}$ in $\Omega$ and further an unbounded sequence $k(n)$ of natural numbers. Denote the sampling encoders corresponding to the sampling points $\{y_1,...,y_{k(n)}\}$ by $E_n$. Then there is some $f\in \CC(\Omega, \K)$ and $N\in \N$ such that for all $n\geq N$ it is
    \begin{align*}
        E_n(f) \neq \tilde{E}_{k(n)}(f).
    \end{align*}
\end{theorem}
\begin{proof}
    Without loss of generality, assume that $k(n)$ is strictly increasing to infinity. Otherwise, consider a subsequence with that property. 
    Due to the monotonicity of $k(n)$ and since the set of sampling points for $E_{n+1}$ contain the sampling points for $E_n$, it suffices to show existence of some $f$ and $N$ such that $E_N(f) \neq \tilde{E}_{k(N)}(f)$. 
    Assume, for the sake of contradiction, that $E_n(f) = \tilde{E}_{k(n)}(f)$ holds for every $n\in \N$ and every $f\in \CC(\Omega, \K)$, which means that 
    \begin{align*}
        \left( f(y_1),\dots,f(y_{k(n)}) \right)^\intercal = \left( c_1(f), \dots, c_{k(n)}(f) \right)^\intercal.
    \end{align*}
    Therefore, it follows from the monotonicity of $k(n)$ and \Cref{thm:schauder_basis_BAP} that 
    \begin{align*}
        f = \sum_{i=1}^\infty c_i(f) b_i = \sum_{i=1}^\infty f(y_i) b_i
    \end{align*}
    holds for every $f\in \CC(\Omega, \K)$. However, such a Schauder basis representation, in which the unique coefficient functionals $c_i$ sample the function $f$, is impossible, as shown in \Cref{thm:sampling_schauder_differ}.
\end{proof}
In the theorem above, the sampling points for $E_{n+1}$ were a refinement of the sampling points for $E_n$. If, more generally, individual sampling points $\{y_1^{(n)}, ..., y_{k(n)}^{(n)}\}$ are chosen for each $n\in \N$, it remains an open question whether the previous theorem still holds. 
Let us close the discussion about \Cref{thm:sampling_schauder_differ_main} with a final remark on the assumption of $k(n)$ to be unbounded. If $\{y_1^{(n)}, ..., y_{k(n)}^{(n)}\}$ is an $\frac{1}{n}$-covering of a set $\Omega\subseteq \K^d$ having a non-zero Lebesgue measure $\textup{Vol}(\Omega)>0,$ the unboundedness of $k(n)$ follows from the translation-invariance and $\sigma$-subadditivity of the Lebesgue measure: If $k(n)$ had an upper bound $s<\infty$, it would follow that
\begin{align*}
    \textup{Vol}(\Omega) \leq \sum_{i=1}^{k(n)} \textup{Vol}\left(B_{\frac{1}{n}}(y_i^{(n)})\right) \leq s \textup{Vol}\left(B_{\frac{1}{n}}(y_1^{(1)})\right)\xrightarrow{n\to\infty}0,
\end{align*}
which would be a contradiction to $\Omega$ having positive measure.

\begin{theorem}\label{thm:sampling_schauder_differ}
    Let $\Omega$ be a metric space without isolated points. Assume that there is a Schauder basis $(b_n)_{n\in \N}$ of the space of bounded continuous functions $\CC_b(\Omega, \K)$, equipped with the supremum-norm.
    Then for any sequence of points $(y_n)_{n\in \N}$ in $\Omega$ there exists some $f\in \CC_b(\Omega, \K)$ such that
    \begin{align*}
        f \neq \sum_{n=1}^\infty f(y_n) b_n.
    \end{align*}
\end{theorem}
\begin{proof}
    Assume there was such a Schauder basis and a sequence of sampling points such that $f = \sum_{n=1}^\infty f(y_n) b_n$ holds for all $f\in \CC_b(\Omega, \K)$. Define the set of sampling points $Y\coloneqq\{ y_n: n\in \N\}$, which are either dense or not dense in $\Omega$. In both cases this will lead to a contradiction: 
    
    If $Y$ is not dense in $\Omega$, there must exist some $z\in \Omega$ and $\varepsilon>0$ such that the open ball $B_\varepsilon(z)$ and $Y$ are disjoint. In this case one can construct an $f\in \CC_b(\Omega, \K)\setminus\{0\}$ with support in $B_\varepsilon(z)$. This implies $f(y_n)=0$ for all $n\in \N$, which yields
    \begin{align*}
        0\neq f = \sum_{n=1}^\infty f(y_n) b_n = 0,
    \end{align*}
    a contradiction.

    Assume that $Y$ is dense in $\Omega$. Consider, for example,  the representation of the first Schauder basis function 
    \begin{align*}
        b_1 = \sum_{n=1}^\infty b_1(y_n) b_n.
    \end{align*}
    Due to the uniqueness of Schauder basis representations it follows that $b_1(y_n) = 0$ for all $n\in \N\setminus\{1\}$.
    Since $\Omega$ has no isolated points, $Y\setminus\{y_1\}$ is still dense in $\Omega$ (see previous \Cref{lem:density_single_point}). Therefore, continuity of $b_1\in \CC_b(\Omega, \K)$ implies $b_1 = 0$, which is again a contradiction. Hence, there cannot exist such a Schauder basis representation.
\end{proof}

\begin{lemma}\label{lem:density_single_point}
    Let $\Omega$ be a topological space satisfying the following properties:
    \begin{itemize}
        \item[(i)] For every $x\in \Omega$ the set $\{x\}$ is closed. (Note that in metric spaces points are always closed sets.)
        \item[(ii)] $\Omega$ has no isolated points, i.e. for any $x\in \Omega$ and any open neighborhood $U$ of $x$, the set $U\setminus\{x\}$ is non-empty. 
    \end{itemize}
    Further, assume that $Y\subseteq \Omega$ is dense in $\Omega$. Then $Y \setminus \{ y'\}$ is still dense in $\Omega$ for any $y'\in Y$.
\end{lemma}
\begin{proof}
    Assume $Y\setminus \{y'\}$ was not dense in $\Omega$. Hence, there must be some $z\in \Omega$ and an open neighborhood $U\subset \Omega$ of $z$ such that 
    \begin{align}\label{lem:eq1:density_single_point}
        U \cap \Big( Y\setminus \{y'\} \Big) = \emptyset.
    \end{align}
    Since $Y$ is dense in $\Omega$ it must hold that $U\cap Y\neq \emptyset$, which implies $y'\in U$. This means that $U$ is an open neighborhood of $y'$. Due to (ii) we conclude that $U\setminus\{y'\}$ cannot be empty, so there must be some $x'\in U\setminus\{y'\}$. From (i) it follows that $U\setminus \{y'\}$ is open and thereby it is an open neighborhood of $x'$. 
    We observe
    \begin{align*}
         \Big(U\setminus \{ y'\} \Big)\cap Y =  U \cap \Big( Y\setminus \{y'\} \Big) \overset{\eqref{lem:eq1:density_single_point}}{=} \emptyset,
    \end{align*}
    which is a contradiction to the density of $Y$ in $\Omega$.
\end{proof}

\section{The Skorohod space $\DD([0, \infty))$}\label{sec:skorohod_infty}
For the definition of a metric $d_\infty$ inducing the Skorohod topology on $\DD([0,\infty))$, we follow \cite[Chapter 16]{Billingsley_Probability_Measures}.
Beforehand, for $R>0$ we define the mapping
\begin{align*}
    \psi_R: \DD([0,\infty)) &\longrightarrow \DD([0, R])\\
    \psi_R f(t) &=\begin{cases}
        f(t)\phantom{0(m-t)} \textup{ if } t\leq R-1,\\
        f(t)(R-t)\phantom{0} \textup{ if } R-1<t\leq R.
    \end{cases}
\end{align*}
Note that for any $f\in \DD([0, \infty))$ we can write $\psi_Rf(t) = \varphi_R(t)f(t)$ with the mapping
\begin{align*}
    \varphi_R:[0, R] &\longrightarrow[0, 1]\\
    t&\longmapsto \begin{cases}
        1 \textup{ if } t\leq R-1,\\
        R-t \textup{ else.}
    \end{cases}
\end{align*}

\begin{definition}
    For $f,g\in \DD([0,\infty))$ define
    \begin{align*}
        d_\infty(f,g)\coloneqq \sum_{R=1}^\infty 2^{-R} \min\Big\{ 1, d_R\big(\psi_R(f), \psi_R(g)\big) \Big\}.
    \end{align*}
\end{definition} 

In order to show that $\DD([0, \infty))$ equipped with $d_\infty$ has the input-EDAP, we slightly adapt \Cref{def:A_sigma_R}. Below, we define a mapping $A^{(\infty)}_{\sigma, n}$, which also has a canonical decomposition in terms of an encoder and a decoder. To simplify the notation, we define the clipping function $\clip:\R\times [0, \infty) \to \R$ by
\begin{align*}
    \clip(x, n)\longmapsto\begin{cases}
        n \textup{ \, if \, } x\geq n,\\
        -n \textup{ \, if \, } x\leq -n,\\
        x,\textup{ \, else.}
    \end{cases}
\end{align*}
\begin{definition}\label{def:A_sigma_infty}
    Let $n\in \N$, consider points $0=s_0<s_1<...<s_k$ in $[0, \infty)$, write $\sigma=\{s_0,...,s_k\}$ and define
    \begin{align*}
        A_{\sigma,n}^{(\infty)}: \DD([0, \infty)) &\longrightarrow \DD([0, \infty)) \\
        A_{\sigma,n}^{(\infty)}(f)(t) &\coloneqq \begin{cases}
            \clip(f(s_{i-1}), n)\textup{ \, if \, } t\in [s_{i-1}, s_i) \textup{ \, for \, } i=1,...,k,\\
            \clip(f(s_k), n) \textup{ \, if \, } t\geq s_k.
        \end{cases}
    \end{align*}
    We can write $A^{(\infty)}_{\sigma, n} = D^{(\infty)}_{\sigma}\circ E^{(\infty)}_{\sigma, n}$ where
    \begin{align*}
    E_{\sigma, n}^{(\infty)}(f) &\coloneqq \Big(\clip(f(s_{0}), n), ..., \clip(f(s_{k}), n)\Big)^T,\\
    D_{\sigma}^{(\infty)}(a)(t) &\coloneqq \begin{cases}
        a_{i-1} \textup{ \, if \, } t\in [s_{i-1}, s_i) \textup{ \, for \, } i=1,...,k,\\
        a_k \textup{ \, if \, } t\geq s_k.
    \end{cases}  
\end{align*}
Note that the decoder does not depend on $n$.
\end{definition}

\begin{lemma}\label{lem:skoro_D_lipschitz}
    $D^{(\infty)}_{\sigma}$ is Lipschitz continuous for any $\sigma=\{s_0,s_1,...,s_k\}\subset[0,\infty)$.
\end{lemma}
\begin{proof}
    For better readability, we simply write $D=D_{\sigma}^{(\infty)}$. It holds that 
    \begin{align*}
        d_\infty(D(a), D(b)) &\leq \sum_{R=1}^{\infty}2^{-R} d_R\big(\psi_R[D(a)], \psi_R[D(b)]\big)\\
        &\leq \sum_{R=1}^{\infty}2^{-R} \sup_{t\in [0, R]} \big\vert \varphi_R(t)[D(a)(t)] - \varphi_R(t)[D(b)(t)]\big \vert\\
        &\leq\sum_{R=1}^{\infty}2^{-R} \sup_{t\in [0, R]} \big\vert D(a)(t) - D(b)(t)\big \vert \\
        &\leq \sum_{R=1}^{\infty} 2^{-R} \vert a-b\vert_{\infty} =  \vert a-b\vert_{\infty}.
    \end{align*}
\end{proof}

Next, we prove an analogous result to Lemma \ref{lem:compact_convergence_Asigma} for the space $\DD([0, \infty))$. Again, this result will be the key for approximating the identity by $A^{(\infty)}_{\sigma, n}$ uniformly on compact sets, see Corollary \ref{coro:skorohd_identity_infty}.
\begin{lemma}\label{lem:Asigma_psi_commutes}
    Consider $0=s_0<s_1<s_2<...<s_k$ and write $\sigma=\{s_0,...,s_k\}$. Let $R, \delta>0$ and $n\in \N$. Assume that $R\in \sigma$ and $\vert s_{i-1}- s_i\vert < \delta$ for all $i$.  
    Then for every $f\in \DD([0, \infty))$ it holds that
    \begin{align*}
        &d_R\left(\psi_R(f), \psi_R (A^{(\infty)}_{\sigma,n}(f))\right)\\ \leq &\max\{\delta, \overline{w}_R(\psi_Rf, \delta)\} + \sup_{t\in [0, R]}\Big( \delta\vert f(t) \vert + \vert f(t) -\clip(f(t), n) \vert\Big).
    \end{align*}
\end{lemma}
\begin{proof}
    Define $\sigma_R\coloneqq \{s_i\in \sigma: s_i\leq R\}$ and recall $A^{(R)}_{\sigma_R}$ from Definition \ref{def:A_sigma_R}.
    It follows from Lemma \ref{lem:compact_convergence_Asigma} that 
    \begin{align*}
        d_R\left(\psi_R(f), A^{(R)}_{\sigma_R}\psi_R f\right) \leq \max\{\delta, \overline{w}_R(\psi_Rf, \delta)\}.
    \end{align*}
    Recall the mapping $\varphi_R:[0, R]\to [0,1]$ defined at the beginning of this section, which allows us to rewrite $\psi_R(g)(t) = \varphi_R(t) g(t)$ for $g\in \DD([0, \infty))$.
    Furthermore, we have for $t\in [0, R]$ that
    \begin{align*}
        &\left\vert  A^{(R)}_{\sigma_R}\psi_R f(t) - \psi_R (A^{(\infty)}_{\sigma,n}(f))(t)\right\vert\\ = &
        \begin{cases}
            \big\vert \varphi_R(s_{i-1})f(s_{i-1}) - \varphi_R(t) \clip(f(s_{i-1}), n) \big\vert, \textup{ if } t\in [s_{i-1}, s_i) \textup{ and } s_i\leq  R;\\
            \big\vert \varphi_R(R)f(R) - \varphi_R(R) \clip(f(R), n) \big\vert \textup{ if } t=R
        \end{cases}\\
        \leq & \sup_{t\in [0, R]}\Big(\delta\vert f(t) \vert + \vert f(t) -\clip(f(t), n) \vert\Big),
    \end{align*}
    due to the triangular inequality and the facts that $\varphi_R$ maps into $[0,1]$ and is 1-Lipschitz. Therefore,
    \begin{align*}
        d_R\left(A^{(R)}_{\sigma_R}\psi_R f, \psi_R (A^{(\infty)}_{\sigma,n}(f))\right) 
        &\leq \sup_{t\in [0, R]}\left\vert  A^{(R)}_{\sigma_R}\psi_R f(t) - \psi_R (A^{(\infty)}_{\sigma,n}(f))(t)\right\vert \\
        &\leq \sup_{t\in [0, R]}\Big(\delta\vert f(t) \vert + \vert f(t) -\clip(f(t), n) \vert\Big).
    \end{align*}
    Hence, the claim follows from the triangular inequality.
\end{proof}

The following characterization of relatively compact sets in $\DD([0,\infty))$ can be found in \cite[Theorem 16.4]{Billingsley_Probability_Measures}.
\begin{theorem}\label{thm:precompact_skorohod_Dinfty}
    A set $K\subset \DD([0,\infty))$ is relatively compact w.r.t. $d_\infty$ if and only if for every $R\in \N$, the set $\psi_R(K)$ is relatively compact in $\DD([0, R])$ w.r.t. $d_R$.
\end{theorem}

\begin{corollary}\label{coro:skorohd_identity_infty}
    Let $\delta_n>0$ be a null sequence. Write $\sigma_n=\{ s_0^{(n)},...,s^{(n)}_{k(n)}\}$ after choosing $0=s_0^{(n)}<s_1^{(n)}<...<s_{k(n)}^{(n)}$ which fulfill the following three conditions:
    \begin{itemize}
        \item[(i)] $\sigma_n$ contains all $R\in \N$ with $ R\leq s_{k(n)}^{(n)}$;
        \item[(ii)] $s_{k(n)}^{(n)}$ strictly increases to infinity;
        \item[(iii)] $\vert s_{i-1}^{(n)}- s_i^{(n)}\vert < \delta_n$ for all $i=1,...,k(n).$
    \end{itemize}
    Then for every compact $K\subset \DD([0,\infty))$ it is
    \begin{align*}
        \sup_{f\in K} d_\infty( f, A_{\sigma_n, n}^{(\infty)}f) \xrightarrow{n\to \infty} 0.
    \end{align*}
\end{corollary}
\begin{proof}
    Let $\varepsilon>0$ and choose $R_\varepsilon\in \N$ such that 
    \begin{align*}
        \sum_{R=R_\varepsilon+1}^\infty 2^{-R} < \frac{\varepsilon}{3}.
    \end{align*}
    Since $s_{k(n)}^{(n)}$ strictly increases to infinity (property (ii)), there is $N_\varepsilon>0$ such that for all $n\geq N_\varepsilon$ it is $R_\varepsilon\leq s_{k(n)}^{(n)}$. Therefore, and due to properties (i) and (iii) of $\sigma_n$, we can apply \Cref{lem:Asigma_psi_commutes} for any $R\leq R_\varepsilon$ and $n\geq N_\varepsilon$ to conclude that
    \begin{align*}
        &\sum_{R=1}^{R_\varepsilon}d_R(\psi_Rf, \psi_R A^{(\infty)}_{\sigma_n,n}f) \\ 
        \leq &R_\varepsilon \left( \sup_{t\in [0, R_\varepsilon]}\delta_n\vert f(t)\vert + \vert f(t) - \clip(f(t), n)\vert\right) + \sum_{R=1}^{R_\varepsilon} \max\{\delta_n, \overline{w}_R(\psi_R f, \delta_n)\}.
    \end{align*}
    Let $K\subset \DD([0, \infty))$ be compact. Hence, for every $R>0$, the set $\psi_R(K)$ is relatively compact in $\DD([0, R])$, according to Theorem \ref{thm:precompact_skorohod_Dinfty}. 
    In particular, $\psi_{R_\varepsilon + 1}(K)$ is bounded with respect to the supremum norm, see \Cref{thm:precompact_sets_skorohov}. Therefore, and since $\delta_n$ is a null sequence, there is $N_{K, \varepsilon}>0$ such that for all $n\geq N_{K, \varepsilon}$ it is 
    \begin{align*}
        \sup_{f\in K}R_\varepsilon \left( \sup_{t\in [0, R_\varepsilon]}\delta_n\vert f(t)\vert + \vert f(t) - \clip(f(t), n)\vert\right) = R_\varepsilon \delta_n \sup_{f\in K} \sup_{t\in [0, R_\varepsilon]}\vert f(t)\vert < \frac{\varepsilon}{3}.
    \end{align*}

    By applying Theorem \ref{thm:precompact_sets_skorohov} to all finitely many $R\in \N$ with $R\leq R_\varepsilon$, we conclude that there is some $N'_{K, \varepsilon}>0$ such that for all $n\geq N'_{K, \varepsilon}$ we have 
    \begin{align*}
        \sup_{f\in K} \sum_{R=1}^{R_\varepsilon} \max\{\delta_n, \overline{w}_R(\psi_R f, \delta_n)\} < \frac{\varepsilon}{3},
    \end{align*}
    Therefore, for all $n\geq \max \{N_\varepsilon, N_{K, \varepsilon}, N_{K, \varepsilon}'\}$, it follows that
    \begin{align*}
        \sup_{f\in K} d_\infty(f, A^{(\infty)}_{\sigma_n,n}f) &= \sup_{f\in K} \sum_{R=1}^\infty 2^{-R} \min\{1, d_R(\psi_R f,\psi_R A^{(\infty)}_{\sigma_n,n}f)\}\\
        &\leq \sum_{R=R_\varepsilon+1}^\infty 2^{-R} + \sup_{f\in K} \sum_{R=1}^{R_\varepsilon}d_R(\psi_Rf, \psi_RA^{(\infty)}_{\sigma_n,n}f)< \varepsilon.
    \end{align*}
\end{proof}

Finally, we obtain the desired result.
\begin{theorem}\label{thm:skoro_infty_EDAP}
    $\DD([0, \infty))$ has the input-EDAP.
\end{theorem}
\begin{proof}
    It only remains to show that the encoder $E_{\sigma_n, n}^{(\infty)}$ fulfills property (i) in \Cref{def:input_EDAP}. This follows from the observation that for every $f\in \DD([0, \infty))$ it is
    \begin{align*}
        \vert E_{\sigma_n, n}^{(\infty)} f \vert \leq c_n n,
    \end{align*}
    where $c_n>0$ is such that $\vert x \vert \leq c_n \vert x \vert_{\infty}$ for all $x\in \R^{k(n)}.$
\end{proof}

\begin{remark}
    It is due to \Cref{lem:subspace_EDAP} that for any $b>0$, the set $M_b\coloneqq \{f\in \DD([0, \infty)): \Vert f \Vert_\infty \leq b\}$ equipped with $d_\infty$ has the input-EDAP, too. For $M_b$, one can also omit the clipping of the encoder in \Cref{def:A_sigma_infty}, as this was only required for verifying property (i) in \Cref{def:input_EDAP}.
\end{remark}
\begin{remark}
    On $\DD([0, \infty))$, one could also omit the clipping of the encoders in \Cref{def:A_sigma_infty}. However, the encoders and decoders would only fulfill the weaker requirements of the input-CEDAP, see \Cref{def:input_CEDAP}. Still, this would allow for operator approximation as in statement \labelcref{fig:approximation_types:(A)}, with the Skorohod space being the input space. 
\end{remark}

\section{Approximation by MIONets}\label{sec:proof_mionet}
In this section, we provide a proof of \Cref{thm:mionet}, which applies to low-rank MIONets from \Cref{def:mionets}, when basis encoders are used for the input spaces $\X_i$, and $S$ is a dense subset of $\Y$ that consists of neural networks. Note that it would be sufficient to assume that $\X_1$ and $\X_2$ have the input-CEDAP, respectively. For a simpler notation, we assume them to have the input-EDAP.

Beforehand, we introduce some notation for the tensor products of vector spaces, where we follow \cite[Chapter 1.1]{ryan_tensor_products}.
The tensor product between vector spaces $\Z_1$ and $\Z_2$ is denoted by $\Z_1\otimes\Z_2$. Note that every $u\in \Z_1\otimes\Z_2$ has a representation
\begin{align*}
    u = \sum_{i=1}^p v_i \otimes w_i,
\end{align*}
where $v_i\otimes w_i$ denotes the tensor product of some $v_i\in \Z_1$ and $w_i\in \Z_2.$ 

\begin{proof}[Proof of \Cref{thm:mionet}]
    Let $G\in \CC(\X_1\times \X_2, \Y)$ and consider some compact $K\subseteq \X_1\times \X_2$. First, observe that there are compact $K_1\subseteq \X_1$ and $K_2\subseteq \X_2$ such that $K\subseteq K_1\times K_2$. According to \cite[Chapter 3.2]{ryan_tensor_products}, the mapping
    \begin{align*}
        J: \big(\CC(K_1, \R) \otimes \CC(K_2, \R) \big)\otimes \Y &\longrightarrow \CC(K_1\times K_2, \Y)\\
        J\left( \sum_{i=1}^p (v_i\otimes w_i) \otimes y_i\right)(f_1, f_2) &\coloneqq \sum_{i=1}^p v_i(f_1) w_i(f_2) y_i
    \end{align*}
    has dense range in $\CC(K_1\times K_2, \Y)$, the latter being equipped with the supremum norm. Therefore, for any $\delta>0$, there exists some  $v_i\in \CC(K_1, \R)$, $w_i\in \CC(K_2, \R)$ and $y_i\in \Y$ such that
    \begin{align*}
        \sup_{(f_1, f_2)\in K_1\times K_2}\left\Vert G(f_1, f_2) -  \sum_{i=1}^p v_i(f_1) w_i(f_2)y_i \right \Vert_\Y < \delta.
    \end{align*}
    The claim then follows by approximating $y_i$ by $\psi_i\in S$, as well as applying \Cref{coro:main_dugundji_extension} for the approximation of $v=(v_1,...,v_p)\in \CC(K_1, \R^p)$ and $w=(w_1,...,w_p)\in \CC(K_2, \R^p)$ by a concatenation of $\varphi^{(1)}, \varphi^{(2)}\in \FF$ with suitable encoders from the input-EDAP of $\X_1$ and $\X_2$. Note that on $\R^p$, the decoders in \Cref{coro:main_dugundji_extension} can be chosen as the identity mapping.
\end{proof}

\section{Further useful results}
\begin{lemma}\label{lem:Lipschitz_K_independence}
    Let $(\X, d_\X)$ and $(\Y, d_\Y)$ be metric spaces, where $\X$ is separable. Consider some Lipschitz continuous $F\in \CC(\X, \Y)$. Assume that for every compact $K\subseteq \X$ there is some sequence $(F_{K, n})_{n\in \N}$ in $\CC(\X, \Y)$ which satisfies
    \begin{itemize}
        \item[(i)] All $F_{K, n}$ are Lipschitz continuous with Lipschitz constant $L>0$, which is independent on $n$ and $K$.
        \item[(ii)] It holds that 
        \begin{align*}
            \sup_{x\in K} d_\Y\big(F(x) \, , \, F_{K, n}(x)\big) \xrightarrow{n\to\infty} 0.
        \end{align*}
    \end{itemize}
    Then there exists a strictly increasing sequence $(a_n)_{n\in\N}\subset \N$ as well as a sequence of compact sets $K_n$ such that $F_n \coloneqq F_{K_n, a_n}$ converges uniformly to $F$ on every compact $K$.
\end{lemma}
\begin{proof}
    By separability of $\X$, there exists a sequence $(x_n)_{n\in \N}$ which is dense in $\X$. Define the compact sets $K_n \coloneqq \{x_k: k\leq n\}$. By property (ii) there is some strictly increasing sequence of natural numbers $(a_n)_{n\in\N}$ such that for all $n\in \N$ it holds that
    \begin{align*}
        \sup_{x\in K_n} d_\Y\big( F_{K_n, a_n} (x) \, , \, F(x)\big) \leq \frac{1}{n}.
    \end{align*}
    Define the operators $F_n \coloneqq F_{K_n, a_n}.$ We first show pointwise convergence, so let $x\in \X$ and $\varepsilon>0$.
    Due to density of $\{x_n: n\in \N\}$ and continuity of $F$ in $x$, there exists $i\in \N$ such that
    \begin{equation*}
        d_\X( x\, , \, x_i) \leq \frac{\varepsilon}{3L} \quad \text{and} \quad d_\Y\big( F(x_i) \, , \, F(x)\big) \leq \frac{\varepsilon}{3}.
    \end{equation*}
    Choose $N\in \N$ large enough such that for all $n\in \N$ we have
    \begin{align*}
        x_i\in K_n\quad \text{and} \quad \sup_{x\in K_n} d_\Y\big( F_{n} (x) \, , \, F(x)\big) \leq \frac{\varepsilon}{3}.
    \end{align*}
    It follows then from (i) that
    \begin{align*}
        d_\Y\big( F(x) \, , \, F_n(x) \big) &\leq d_\Y\big( F(x) \, , \, F(x_{i}) \big) + d_\Y\big( F(x_{i}) \, , \, F_n(x_{i}) \big) + d_\Y\big( F_n(x_{i}) \, , \, F_n(x) \big) \\
        &\leq \frac{\varepsilon}{3} + \frac{\varepsilon}{3} + Ld_\X( x_{i} \, , \, x ) \leq \varepsilon,
    \end{align*}
    which shows pointwise convergence of $F_n$ to $F$.
    
    Let $K\subseteq \X$ be compact and $\varepsilon>0$. Let $L_0>0$ be a Lipschitz constant of $F$. By compactness of $K$, there exists $y_1,\dots,y_p\in K$ such that 
    \begin{align*}
        K \subset \bigcup_{i=1}^p B_r(y_i) \textup{\, with radius }r\coloneqq \frac{\varepsilon}{3(L_0+L)}
    \end{align*}
    Since $F_n$ converges pointwise to $F$, choose $N\in \N$ large enough such that for all $n\geq N$ and every $i=1, \dots,p$ it is
    \begin{align*}
        d_\Y\big( F_n(y_i) \, , \, F(y_i) \big) \leq \frac{\varepsilon}{3}.
    \end{align*}
     Given any $x\in K$ there exists $y_{i^*}$ with $i^*\in \{1,\dots,p\}$ such that $d_\X( x\, , \,y_{i^*}) \leq r$. Therefore,
    \begin{align*}
        d_\Y\big( F(x) \, , \, F_n(x) )&\leq d_\Y\big( F(x) \, , \, F(y_{i^*}\big) + d_\Y\big( F(y_{i^*}) \, , \, F_n(y_{i^*}) \big)  + d_\Y\big( F_n(y_{i^*}) \, , \, F_n(x) \big) \\
        &\leq L_0 d_\X( x \, , \, y_{i^*} ) + \frac{\varepsilon}{3} + L d_\X( y_{i^*} \, , \, x) \leq \varepsilon.
    \end{align*}
    Taking the supremum over all $x\in K$ shows that $F_n$ converges uniformly to $F$ on $K$.
    \end{proof}

\begin{table}[!htb]
    \caption{Overview of the provided encoders and decoders in \Cref{sec:approximation_properties}.}
    \label{tab:encoder_decoder_overview}
    \renewcommand\bottomrule{\specialrule{1pt}{2pt}{3pt}}
    \raa{1.1}
    \centering
    \resizebox{0.9\columnwidth}{!}{
    \pgfplotstabletypeset[
        col sep=&, row sep=\\,
        string type,
        every head row/.style={
            before row={
              \toprule 
            },
            after row=\midrule,
        },
        every last row/.style={
            after row=\bottomrule},
        every even row/.style={
            before row={\rowcolor[gray]{0.95}}
        },
        columns/ansatz/.style     ={column name=\textbf{Ansatz},column type=l},
        columns/encoder/.style    ={column name=\textbf{Encoder $E_n$},column type=l},
        columns/decoder/.style    ={column name=\textbf{Decoder $D_n$},column type=l},
        ]{
            ansatz & encoder & decoder \\
            \makecell[{{p{2.5cm}}}]{Schauder Basis\\\scriptsize(\Cref{thm:schauder_basis_BAP})} & 
            \makecell[{{p{6cm}}}]{$
            \begin{aligned}
                & \\[-22pt]
                \X &\longrightarrow \K^n\\[1.8ex]
                f &\longmapsto \big(c_1(f),\dots,c_n(f)\big)^\intercal
            \end{aligned}
            $} &
            \makecell[{{p{4cm}}}]{$
            \begin{aligned}
                & \\[-9pt]
                \K^n &\longrightarrow \X \\
                \mu &\longmapsto \sum_{i=1}^n \mu_i b_i\\[2pt]
            \end{aligned}
            $}\\
            \makecell[{{p{2.5cm}}}]{Frames\\\scriptsize(\Cref{thm:frames})} &
            \makecell[{{p{6cm}}}]{$
            \begin{aligned}
                & \\[-20pt]
                \X &\longrightarrow \K^n\\[1.8ex]
                f &\longmapsto \big( \langle f, f^*_1 \rangle, \dots, \langle f, f^*_n \rangle \big)^\intercal
            \end{aligned}
            $} &
            \makecell[{{p{4cm}}}]{$
            \begin{aligned}
                & \\[-9pt]
                \K^n &\longrightarrow \X \\
                \mu &\longmapsto \sum_{i=1}^n \mu_i f_i\\[2pt]
            \end{aligned}
            $}\\
            \makecell[{{p{2.5cm}}}]{Sampling\\\scriptsize(\Cref{thm:sampling_BAP})} & 
            \makecell[{{p{6cm}}}]{$
            \begin{aligned}
                & \\[-18pt]
                \CC(\Omega, \K) &\longrightarrow \K^{k(n)}\\[0.25cm]
                f &\longmapsto \Big(f(y_1^{(n)}),\dots, f(y_{k(n)}^{(n)})\Big)^\intercal
            \end{aligned}
            $} &
            \makecell[{{p{4cm}}}]{$
            \begin{aligned}
                & \\[-9pt]
                \K^{k(n)} &\longrightarrow \CC(\Omega, \K) \\
                \mu &\longmapsto \sum_{i=1}^{k(n)} \mu_i P_{\frac{1}{n}, i}\\[2pt]
            \end{aligned}
            $}\\
            \makecell[{{p{2.5cm}}}]{Density in\\ Normed Spaces\\\scriptsize(\Cref{thm:Tn_substitute_dense_set})} & 
            \makecell[{{p{6cm}}}]{$
            \begin{aligned}
                & \\[-18pt]
                \X &\longrightarrow \K^{\vert I_n\vert}\\[0.25cm]
                f &\longmapsto \Big(c^{(n)}_1(T_n f),\dots, c^{(n)}_{\vert I_n\vert}(T_n f)\Big)^\intercal
            \end{aligned}
            $} &
            \makecell[{{p{4cm}}}]{$
            \begin{aligned}
                & \\[-9pt]
                \K^{\vert I_n\vert} &\longrightarrow \X \\
                \mu &\longmapsto \sum_{i=1}^{\vert I_n\vert} \mu_i v_i^{(n)}\\[2pt]
            \end{aligned}
            $}\\
            \makecell[{{p{2.5cm}}}]{Density in\\ Hilbert Spaces\\\scriptsize(\Cref{thm:Tn_substitute_dense_set_HilbertSpace})} & 
            \makecell[{{p{6cm}}}]{$
            \begin{aligned}
                & \\[-18pt]
                \X &\longrightarrow \K^{n}\\[0.25cm]
                f &\longmapsto \Big(\langle f, v_1^{(n)}\rangle,\dots,\langle f, v_{n}^{(n)}\rangle\Big)^\intercal
            \end{aligned}
            $} &
            \makecell[{{p{4cm}}}]{$
            \begin{aligned}
                & \\[-9pt]
                \K^{n} &\longrightarrow \X \\
                \mu &\longmapsto \sum_{i=1}^{n} \mu_i v_i^{(n)}\\[2pt]
            \end{aligned}
            $}\\
        }
    }
\end{table}

\section*{Acknowledgements}
Janek Gödeke acknowledges funding by the Deutsches Zentrum für Luft- und Raumfahrt (grant no. 50 EE 2204). Further, Janek Gödeke and Pascal Fernsel acknowledge funding by the Deutsche Forschungsgemeinschaft (DFG, German Research Foundation, Project Number 281474342/GRK2224/2).

We thank Peter Maaß from the University of Bremen and Maarten de Hoop from the Rice University for fruitful discussions about operator learning and approximation. We also thank Nihat Ay for his invitation to TU Hamburg for discussions, and Paweł Przybyłowicz from AGH University of Krakow for his question whether the EDAP-theory is applicable to Skorohod spaces.
Further, we thank Hendrik Vogt from the University of Bremen for our discussions and his contribution to \Cref{thm:X_nonseparableMetricSpace_equivalence}. 

\bibliographystyle{siam}  
\bibliography{references}

\end{document}